\documentclass[reqno]{amsart}
\usepackage{latexsym}
\usepackage{amsmath}
\usepackage{amsthm}
\usepackage{amssymb}
\usepackage{color}
\usepackage{enumitem}
\usepackage{verbatim}
\usepackage{mathrsfs}
\usepackage{hyperref}
\newcommand{\vip}{\vskip0.1cm}

\newcommand{\indiq}{1\!\! 1}
\newcommand{\e}{{\epsilon}}

\newcommand{\ip}[1]{\langle #1 \rangle}
\newcommand{\E}{{\mathbb{E}}}
\newcommand{\R}{{\mathbb{R}}}

\newcommand{\be}{{\bf e}}
\newcommand{\bv}{{\bf v}}
\newcommand{\bV}{{\bf V}}
\newcommand{\bw}{{\bf w}}
\newcommand{\bW}{{\bf W}}
\newcommand{\bY}{{\bf Y}}

\newcommand{\abs}[1]{\lvert#1\rvert}

\newcommand{\cF}{{{\mathcal F}}}
\newcommand{\cW}{{{\mathcal W}}}
\newcommand{\cL}{{{\mathcal L}}}

\newcommand{\cP}{{{\mathcal P}}}

\newcommand{\Sp}{{\mathbb{S}}}
\newcommand{\bbP}{{\mathbb{P}}}
\newcommand{\rd}{{\mathbb{R}^3}}

\newcommand{\intot}{\int_0^t}

\newcommand{\tf}{{\tilde f}}

\newcommand{\tv}{{\tilde v}}
\newcommand{\tx}{{\tilde x}}

\newcommand{\tW}{{\tilde W}}
\newcommand{\tM}{{\tilde M}}

\newcommand{\norm}[1]{\left\|#1 \right\|}

\newcommand{\sm}{{s-}}

\newcommand{\beqn}{\begin{equation}}
\newcommand{\eeqn}{\end{equation}}
\newcommand{\bear}{\begin{eqnarray}}
\newcommand{\eear}{\end{eqnarray}}
\newcommand{\bean}{\begin{eqnarray*}}
\newcommand{\eean}{\end{eqnarray*}}

\newcommand{\bbd}{\mathbf}
\newcommand{\beq}{\begin{equation}}
\newcommand{\eeq}{\end{equation}}
\newcommand{\benu}{\begin{enumerate}}
\newcommand{\eenu}{\end{enumerate}}
\newcommand{\blem}{\begin{lem}}
\newcommand{\elem}{\end{lem}}

\newcommand{\bb}{\mathbb}
\newcommand{\ca}{\mathcal}

\newcommand{\scr}{\mathscr}

\newenvironment{preuve}{\vip\noindent \textit{Proof.}}{\hfill$\square$\vip}

 \makeatletter
\@namedef{subjclassname@2020}{ \textup{2020} Mathematics Subject Classification}
\makeatother

\newtheorem{theo}{Theorem}[section]
\newtheorem{prop}[theo]{Proposition}
\newtheorem{rem}[theo]{Remark}
\newtheorem{lem}[theo]{Lemma}

\newtheorem{cor}[theo]{Corollary}
\newtheorem{nota}[theo]{Notation}

\begin{document}

\title[Kac particle systems for the Boltzmann equation]
{Quantitative propagation of chaos  for the Boltzmann equation with moderately soft potentials}

\author[C. Liu]{Chenguang Liu}
\author[L. Xu]{Liping Xu}
\author[A. Zhang]{An Zhang}

\address{Delft Institute of Applied Mathematics, EEMCS, TU Delft, 2628 Delft, The Netherlands}
\email{liucg92@gmail.com}

\address{School of Mathematical Sciences, Beihang University, PR China}
\email{xuliping.p6@gmail.com}
\address{School of Mathematical Sciences, Beihang University, PR China}
\email{anzhang@buaa.edu.cn}



\begin{abstract}
We study the Kac particle system associated with the spatially homogeneous Boltzmann equation with non-cutoff collision kernels in the moderately soft potential regime $-1<\gamma<0$. We prove uniqueness in law for the particle system and, in particular, establish  an explicit  convergence rate from the empirical measure of the Kac particle system to the weak solution of the Boltzmann equation in Wasserstein-2 distance, assuming that the initial datum $f_0$ has  finite Fisher information and polynomial moments. To the best of our knowledge, this provides the first quantitative convergence rate for the Kac particle system in the soft potential setting.

\end{abstract}

\maketitle

\tableofcontents

\section{Introduction}
\subsection{The Boltzmann equation} The Boltzmann equation, introduced by Boltzmann in 1872, is a cornerstone of kinetic theory. It describes the statistical evolution of a dilute gas undergoing binary collisions and serves as a bridge between microscopic particle dynamics and macroscopic hydrodynamic behavior.  We consider the \emph{three-dimensional spatially homogeneous Boltzmann equation}, which governs the evolution of  the velocity distribution function  $f_t(v)$ for particles with velocity $v\in\bb{R}^3$ at time $t\geq 0$. It takes the form
\beq\label{Bol}
\partial_tf_t(v)=\int_{\bb{R}^3}dv_{\ast}\int_{\bb{S}^2}d\sigma B(|v-v_{\ast}|,\theta)[f_{t}(v^{\prime})f_{t}(v_{\ast}^{\prime})-f_t(v)f_t(v_{\ast})],
\eeq
where 
\beq\label{Bol1}
v^{\prime}=v^{\prime}(v,v_*,\sigma)=\frac{v+v_{\ast}}{2}+\frac{|v-v_{*}|}{2}\sigma,\quad  v_{\ast}^{\prime}=v_{\ast}^{\prime}(v,v_*,\sigma)
=\frac{v+v_{\ast}}{2}-\frac{|v-v_{\ast}|}{2}\sigma,
\eeq
and $\theta$ is the \emph{deviation angle}, defined as $\cos\theta=\frac{v-v_{\ast}}{|v-v_{\ast}|}\cdot\sigma$. The term $B(|v-v_{\ast}|,\theta)\geq 0$ is the \emph{collision kernel}, which  depends on the  interaction between particles  and determines the  collision rate for a pair of particles. 
\vip
A fundamental physical property of the Boltzmann equation is that its solution conserves mass, momentum, and kinetic energy, while the entropy is nonincreasing over time. By Galilean invariance and  a suitable normalization of mass and energy, we may assume that the initial data are centered and normalized. Consequently, for all $t\geq 0$,
\begin{align}\label{eq: inv3}
\int_{\R^3} f_t(v)dv=1,\qquad
\int_{\R^3} v f_t(v)dv=&0,\qquad
\int_{\R^3} |v|^2 f_t(v)dv=3,\nonumber \\
\text{and}\qquad \int_{\R^3} f_t\log(f_t)dv\le& \int_{\R^3} f_0\log(f_0)dv.   \quad \text{(Entropy decay)}
\end{align}

\subsection{Assumptions}
We always assume that there is a measurable function $b:[-1,1]\rightarrow\bb{R}_+$ such that
\beq\label{con}
\left\{\begin{array}{l} B(|v-v_{\ast}|,\theta)={|v-v_{\ast}|}^{\gamma}b(\cos\theta),\\[3pt]
\exists~0<c_0<c_1, ~s.t. ~\forall~\theta\in(0,\pi/2), ~c_0\theta^{-2-\nu}\leq b(\cos\theta)\leq c_1\theta^{-2-\nu},\\[3pt]
\forall~\theta\in[\pi/2,\pi],~b(\cos\theta)=0,
\end{array}
\right.\eeq
with some parameters $\gamma$ and  $\nu$ satisfying 
\beq\label{pa-range}
-1<\gamma<0, \quad 0<\nu<1 \quad  \text{and}\quad  0<\gamma+\nu<1.
\eeq
\begin{rem}
The first two conditions in \eqref{con} correspond to a classical physical example, inverse power-law interactions. More precisely, if particles interact in pairs through a repulsive force proportional to $1/r^s$ for some $s>2$, then the corresponding collision kernel satisfies \eqref{con} with $\gamma=(s-5)/(s-1)$ and $\nu=2/(s-1)$. The parameter range \eqref{pa-range} then corresponds exactly to the case $3<s<5$, which is commonly referred to as \textbf{moderately soft potentials}. For further details, see, e.g., \cite[Section~4]{villani24}. Besides, the vanishing condition $b(\cos\theta)=0$ for $\theta\in[\pi/2,\pi]$ is not restrictive,  it can be imposed without loss of generality by symmetry, as noted in \cite{MR1765272}.
\end{rem}

\vip

\subsection{Change of variables}

Performing the following change of variables transfers the singularity of the kernel from $\theta=0$ to $z=\infty$. This  is introduced solely for  convenience and is not essential to the argument. Define 
\[\beta(\theta):= b(\cos\theta)\sin\theta, \quad H(\theta):=\int_\theta^{\pi/2}\beta(x)dx \qquad \theta\in(0,\pi/2),\]  and  the inverse function
\beq\label{nota}
G(z):=\theta=H^{-1}(z),  \qquad z\in[0,\infty).
\eeq
The second condition in \eqref{con} implies that  $$c_0\theta^{-1-\nu}/2\le\beta(\theta)\le c_1\theta^{-1-\nu},\qquad \theta\in(0,\pi/2).$$
Indeed, under the condition \eqref{con}, the function $H$ is continuous, decreasing, and valued in $[0,\infty)$. Consequently, it admits an inverse function $G:[0,\infty)\mapsto(0,\pi/2)$, which is continuous and decreasing, such that $G(H(\theta))=\theta$ and $H(G(z))=z$. Moreover, there exist constants $c_3>c_2>0$ such that for all $z>0$,
\beq\label{con1}
c_2(1+z)^{-1/\nu}\le G(z)\le c_3(1+z)^{-1/\nu}.
\eeq
Furthermore, it follows from \cite[Lemma 1.1]{MR2398952} that there exists a constant $c_4>0$ such that for all $x,y\in\bb{R}_+$,
\begin{align}
  \int_0^{\infty}\left[G(z/x)-G(z/y)\right]^2dz\le c_4\,\frac{(x-y)^2}{x+y}.  \label{ineq}
\end{align}
We also observe that, if $x\ge y>0$, the monotonicity of $G$ implies that $G(z/x)\ge G(z/y)$ for all $z\ge 0$ and 
\begin{align} \label{ineq1}
\int_0^{\infty}|G(z/x)-G(z/y)|\,dz=
(x-y)\int_0^{\infty}G(z)\,dz
\le c_5|x-y|, 
\end{align}
for some constant $c_5>0.$ By symmetry, \eqref{ineq1} remains valid for all $x,y> 0$.

\subsection{Spherical parameterization}\label{sec: sph par}

We use the spherical parameterization of \eqref{Bol1} introduced in \cite{MR1885616}. For each $x \in\bb{R}^{3}\setminus{\{0\}}$, we introduce   two vectors $I(x), J(x)\in\bb{R}^3$ such that  $(\frac{x}{|x|},\frac{I(x)}{|x|},\frac{J(x)}{|x|})$ forms an orthonormal basis of $\bb{R}^3$. For $x,v,v_\ast\in\bb{R}^3$,
$\theta\in(0,\pi],~\varphi\in[0,2\pi)$, set the parameterization 
\beq\label{z}
\left\{\begin{array}{l} \Gamma(x,\varphi):=(\cos\varphi) I(x)+(\sin\varphi) J(x),\\[3pt]
v^\prime(v,v_\ast,\theta,\varphi):=v-\frac{1-\cos\theta}{2}(v-v_\ast)
+\frac{\sin\theta}{2}\Gamma(v-v_\ast,\varphi),\\[3pt]
a(v,v_\ast,\theta,\varphi):=v^\prime(v,v_\ast,\theta,\varphi)-v.
\end{array}
\right.
\eeq
Although $\Gamma(x,\varphi)$ is not continuous in $x$, Tanaka's trick (see \cite[Lemma 2.3]{MR3456347}) yields a measurable function $\varphi_0: \R^3\times \R^3\to [0,2\pi)$ such that, for all $x,y\in \R^3$ and $\varphi\in[0,2\pi)$,
\beq\label{Gamma}
|\Gamma(x,\varphi)-\Gamma(y,\varphi+\varphi_0(x,y))|\le |x-y|.
\eeq
Since $\Gamma(x,\varphi)$ is orthogonal to $x$ and has the same norm as $x$, one easily checks that
\beq\label{num}
|a(v,v_*,\theta,\varphi)|=\sqrt{\frac{1-\cos\theta}{2}}|v-v_*|.
\eeq
By parameterization \eqref{z}, we can write the unit vector $\sigma\in\bb{S}^2$ as  $$\sigma=\frac{v-v_*}{|v-v_*|}\cos\theta + \frac{I(v-v_*)}{|v-v_*|}\sin\theta\cos\varphi + \frac{J(v-v_*)}{|v-v_*|}\sin\theta\sin\varphi.$$
Using the change of variables $\theta=G(z/|v-v_\ast|^\gamma)$, we further introduce, for $v,v_\ast\in\mathbb{R}^3$, $z\ge0$ and $\varphi\in[0,2\pi)$, that
\begin{align}\label{atc}
c(v,v_\ast,z,\varphi)
:=a\bigl(v,v_\ast,G(z/|v-v_\ast|^\gamma),\varphi\bigr).
\end{align}
Following \cite{MR3313757}, we introduce a modified collision kernel featuring an angular cutoff and a velocity truncation. Consider  the collision kernel $B$ satisfying \eqref{con} and $G$ defined in \eqref{nota}. For $K\ge1$,  $r>0$, $\theta\in(0,\pi/2)$ and $\varphi\in[0,2\pi)$, we introduce
\begin{align}\label{Def: kernel cut}
   B_K(r ,\theta):= B(r\lor K^{1/\gamma},\theta)\indiq_{\{\theta\ge G(K)\}}=(r^\gamma\land K) b(\cos\theta) \indiq_{\{\theta\ge G(K)\}}. 
\end{align}
This immediately implies that $\int_{\Sp^2} B_K(r,\theta)d\sigma=2\pi K(r^\gamma\land K)\le CK^2$. Correspondingly,  using the change of variables
$z=(|v-v_*|^\gamma\wedge K)H(\theta)$, for $v, v_*\in\bb{R}^3$, $z\ge0$ and $\varphi\in[0,2\pi)$, the truncated jump amplitude is denoted by 
 \beq\label{CK}
c_K(v,v_*,z,\varphi)=a[v,v_*,G(z/(|v-v_*|^\gamma\land K)),\varphi]\indiq_{\{z\le K(|v-v_*|^\gamma\land K)\}}.
\eeq

\subsection{The Kac particle system}
The Kac particle system was introduced by Kac in his seminal work \cite{MR84985} as a tool for deriving the spatially homogeneous Boltzmann equation. We now construct this particle system as the solution to a system of stochastic differential equations (SDEs).
\vip
Let $N\ge2$ and let $(O_{ij})_{1\le i\le j\le N}$ be a family of i.i.d. Poisson random measure with intensity $N^{-1}ds dz d\varphi$. For $1\le j<i\le N$, we impose the symmetry
\begin{align}\label{mea: oij}
    O_{ij}(ds,dz,d\varphi)=O_{ji}(ds,dz,d\varphi),
\end{align}
so that particles do not collide with themselves, as the collision kernel vanishes for identical velocities.  Let $(V_0^i)_{i=1,\ldots,N}$ be a family of i.i.d. $f_0$-distributed random variables and set $\bV_0:=(V_0^1,\ldots,V_0^N)$. Then the random vector $\bV_0$ has distribution $F_0^N(d\bv)=f_0^N(d\bv)=f_0^{\otimes N}$. According to \cite[Theorem 4.1]{Fournier2025},   for any $T>0$ and $i=1,\dots,N$,  there exists an 
exchangeable solution $(V_t^i)_{t\in[0,T]}$ solving 
\begin{align}\label{mainsde}
V^{i}_t
= V^i_0 
+ \sum_{j=1}^N \int_0^t \int_0^\infty \int_0^{2\pi} 
c\bigl(V^{i}_{s-},V^{j}_{s-},z,\varphi\bigr)\,
O_{ij}(ds,dz,d\varphi),
\end{align}
where $c =a[v,v_*,G(z/|v-v_*|^\gamma),\varphi]$ defined in \eqref{atc}.

 \vip

We now briefly outline the microscopic interaction mechanism underlying the particle system. Consider a system of $N$ particles with velocities $\bv=(v_1,\dots,v_N)\in\R^{3N}$.  Each unordered pair of particles $(i,j)$ undergoes binary collisions governed by the Boltzmann collision mechanism. Specifically, for pre-collisional velocities $(v_i,v_j)$, the pair collides with deviation angle  $\theta$ at rate $\frac{1}{2N}\,B(|v_i-v_j|,\theta),$
for each $\sigma\in\Sp^2$.  During such a collision, the velocities are updated simultaneously according to the scattering rule \eqref{Bol1}.
\vip

In this work, we focus on the physically relevant \emph{non-cutoff} case, where the angular kernel $\beta(\theta)$ has a non-integrable singularity as $\theta\to 0$.
This corresponds to the regime of grazing collisions, in which particles experience infinitely many interactions with very small deflection angles. 
As a consequence, individual velocities evolve as pure jump processes with infinite activity but very small deviation angles.

\subsection{Kac master equation} 
Let $\ca{C}_b^2(\R^{n})$ denotes the set of functions whose derivatives  up to order two are continuous and bounded.
\vip
By It\^o's formula, the particle system \eqref{mainsde} can also be formulated weakly through the master equation: for any test function $\phi\in \mathcal{C}_b^2(\R^{3N})$, the law $F_t^N$ of the particle system $\bV_t^N=(V_t^i)_{1\le i\le N}$ satisfies
\begin{align}\label{eq:  wmaster}
 \partial_t \ip{F_t^N,\phi}= \ip{F_t^N,\mathcal L_N\phi}    := \frac 1 {2N} \sum_{i,j=1}^N \ip{F_t^N,\mathcal L_N^{i,j}\phi},     \qquad F_0^N(d\bv)\in  \ca{P}_2(\bb{R}^{3N}),
\end{align}
where $\ip{F_t^N,\phi}:= \int_{\R^{3N}}\phi(\bv) F_t^N(d\bv)$ is the dual integral,  and $\mathcal L_N$ is the generator with piecewise operator $ \cL^{ij}_{N}$ defined as
\begin{align}\label{operator}
 \cL^{ij}_{N}\phi(\bv) := \int_{\Sp^2} 
[\phi(\bv + (v'(v_i,v_j,\sigma)-v_i)\be_i
+(v_*'(v_j,v_i,\sigma)-v_j)\be_j) 
-\phi(\bv)] B(|v_i-v_j| ,\theta)d\sigma,
\end{align} 
where we use the notation: for any $h\in \mathbb R^3$, $h\mathbf e_i:=(0,\cdots,0, h, 0,\cdots, 0)\in \mathbb R^{3N}$ with $h$ put in the $i$-th place. Moreover, if the measure $F_t^N(d\bv)$ is absolutely continuous with respect to Lebesgue measure, i.e. $F_t^N(d\bv)= F_t^N(\bv)d\bv$, a change of variables gives the following master equation:
\begin{align}\label{eq: master}
    \partial_t F_t^N(\bv)= \cL_{N} F_t^N(\bv) = \frac 1 {2N} \sum_{i,j=1}^N \cL^{ij}_{N} F_t^N(\bv).
\end{align}
This equation has the same main physical properties as the Boltzmann equation since its solutions also conserve mass, momentum and kinetic energy and have a decaying entropy.  Using the spherical parameterization \eqref{z} and \eqref{atc}, we  rewrite the operator \eqref{operator} as
\begin{align}\label{eq: para}
    \cL^{ij}_{N}\phi(\bv)=&\int_0^{\frac{\pi}{2}} \int^{2\pi}_{0} 
[\phi(\bv + a(v_i,v_j,\theta,\varphi)(\be_i-\be_j)) 
-\phi(\bv)] B(|v_i-v_j| ,\theta)\sin\theta\, d\theta\, d\varphi\nonumber\\
=& \int_0^{\infty} \int^{2\pi}_{0} 
[\phi(\bv + c(v_i,v_j,z,\varphi)(\be_i-\be_j)) 
-\phi(\bv)] \,dz\,d\varphi.
\end{align}

\vip

 We now state the main results.

\subsection{Main results}
Let $\ca{P}(\bb{R}^n)$ denotes the set of probability measures on $\bb{R}^n$. When $f\in\ca{P}(\bb{R}^n)$ has a density, i.e. $f\in\ca{P}(\bb{R}^n)\cap L^1(\bb{R}^n)$, we still denote this density by $f$.
For $q>0$, we set
\[\ca{P}_q(\bb{R}^n)=\{f\in\ca{P}(\bb{R}^n): m_q(f)<\infty\} \qquad\text{with}\quad m_q(f):=\int_{\bb{R}^n}|v|^q f(dv).\] 
We write $\mathcal{P}^{\mathrm{sym}}(\mathbb{R}^n)$ and $\mathcal{P}_{2}^{\mathrm{sym}}(\mathbb{R}^n)$ for the sets of exchangeable probability measures in $\mathcal{P}(\mathbb{R}^n)$ and $\mathcal{P}_{2}(\mathbb{R}^n)$, respectively. For example, $F\in\mathcal{P}_{2}^{\mathrm{sym}}(\mathbb{R}^n)$ means that $F\in\mathcal{P}_{2}(\mathbb{R}^n)$ is symmetric, i.e.,  for every permutation
$\tau$ of $\{1,\dots,n\}$,
\[
F(dv_1\cdots dv_n) = F(dv_{\tau(1)}\cdots dv_{\tau(n)}).
\]

For any $g \in \mathcal{P}^{\mathrm{sym}}(\mathbb{R}^{3N})$, we define its Fisher information  
\begin{align*}
    I_N(g):= \begin{cases}
\int_{\R^{3N}}\frac{|\nabla g|^2}{g}d\bv,&\text{ if }  g\in L^1(\R^{3N})\ \hbox{and} \int_{\R^{3N}}|\nabla g|d\bv<+\infty,\\
        +\infty,&\text{ otherwise}.
        \end{cases}.
\end{align*}

For $j\ge 1$, let $F^{N:j}$ denote the $j$-marginal of the joint law $F^N\in\mathcal{P}^{\mathrm{sym}}(\mathbb{R}^{3N})$, that is,
\begin{align*}
    F^{N:j}(v^1,\cdots,v^j):=\int_{\R^{3(N-j)}} F^N(v^1,\cdots,v^N)dv^{j+1}\cdots dv^N,
\end{align*}
which is regarded as a probability density on $\R^{3j}$. Recalling  \cite[Theorem 3]{Carlen91}, see also \cite[Lemma 3.7]{MR3188710} and \cite[Lemma 2.3]{Fournier2025}, the Fisher information has the following properties:
\beq\label{I-marg}
I_j(F^{N:j})\le \frac{j}{N}I_N(F^N).
\eeq
Moreover,  if $f$ is a probability density on $\R^3$, then for all $N\ge1$, 
\beq\label{tensor}
I_N(f^{\otimes N})=NI_1(f).
\eeq
Our first main result is the uniqueness for the Kac particle system.
\begin{theo}\label{kpst}
Let the collision kernel $B$ satisfy \eqref{con}, $\gamma$, $\nu$ satisfy \eqref{pa-range}. Assume that $f_0\in L^1(\bb{R}^{3})\cap \ca{P}_{2}(\bb{R}^{3})$ has finite Fisher information, i.e. $I_1(f_0)<\infty$. For every $N\ge2$,  the  solution $(V_t^i)_{1\le i\le N, t\in[0,T]}$ of \eqref{mainsde}, starting from a family  of i.i.d $f_0$-distributed random variables $(V_0^i)_{1\le i\le N}$,  is pathwise unique.  \end{theo}

 Our second main result is the following quantitative propagation of chaos.

\begin{theo}\label{th: prop chaos}
Let $(f_t)_{t\geq0}\in C\bigl([0,T], L^3(\mathbb R^3)\cap \ca{P}_2(\mathbb R^3)\bigr)$ be the unique solution to  \eqref{Bol}  satisfying  \eqref{eq: inv3}. Assume that the hypotheses of Theorem \ref{kpst} hold.  In addition, suppose that $\int_{\R^3} |x|^qf_0(x)dx<\infty$ for some $q>8/(1+\gamma)$. For any  $N\ge2$, let $\bV^N_t=(V_t^1,\dots,V_t^N)$ be the Markov process solving \eqref{mainsde} starting from  i.i.d $f_0$-distributed initial random variables $(V_0^i)_{i=1,\dots,N}$.  Then, for every $T > 0$, there exists a constant $C_{T,q,f_0}>0$ (depending only on $m_q(f_0), I_1(f_0)$ and  $T$) such that  
    \begin{align}\label{eq:resu}
        \sup_{0\le t\le T}\E[\cW_2^2(\mu^N_{\bV_t},f_t) ]\le C_{T,q,f_0}\Big(N^{-1/3}+N^{-\ell(q,\gamma)}\Big),
    \end{align}
    where  $\mu^N_{\bV_t}= N^{-1}\sum_{i=1}^N \delta_{V_t^{i}}$ is the associated empirical measure, and  $$\ell(q,\gamma)= \frac{2(1+\gamma)(q-6)((1+\gamma)q-8)}
{(1+\gamma)(5+2\gamma)q^2-(44+32\gamma+12\gamma^2)q+48(1+\gamma)}.$$ 
\end{theo}
\begin{rem}
    1. The condition $q>8/(1+\gamma)$, together with $\gamma>-1$, guarantees the positivity of $\ell(q,\gamma)$. Moreover,  for $\gamma>-1/4$ and $q$ sufficiently large , we have   $\ell(q,\gamma)>\frac{1}{3}.$ Specifically, asymptotically  $|\gamma|\ll 1$ and $q\gg 1$, it follows that   $\ell(q,\gamma)\sim 2/5>1/3.$
Consequently, when the potential is not too soft and the initial data $f_0$ has sufficient decay, the convergence rate is governed by $N^{-1/3}$. This result matches, in terms of the order in $N$, the rate obtained for Maxwellian molecules in \cite{MR3769742}.
\end{rem}

\begin{rem}
 Most of the results in this paper  can be extended to the range
$\gamma\in(-2,-1]$ and $\nu\in[1,3/2)$, which corresponds to another class of moderately soft potentials. In particular, the estimates involving terms of the form
$|v+c|^2-|v|^2$ remain valid in this regime. We strongly believe that Theorem \ref{kpst} remains valid for  $\gamma\in(-2,-1]$.  However, for $\gamma\in(-2,-1]$, we are unable to obtain a quantitative propagation of chaos analogous to that in Theorem \ref{th: prop chaos}. The main difficulty is that an inequality of the form 
$$|a^\gamma-b^\gamma|\le C|a^{-1}-b^{-1}|\min\{a^{\gamma+1},b^{\gamma+1}\}$$  is no longer available in this range. Instead, we only have the weaker bound
 $$|a^\gamma-b^\gamma|\le C|a^{-1}-b^{-1}|(a^{\gamma+1}+b^{\gamma+1}).$$ Inserting this weaker estimate into the proof of Theorem \ref{th: prop chaos} produces an unmanageable term  $|V^{1}_s-W^{1,K}_s|^2|V^{1}_s-V^{2}_s|^\gamma$. Controlling this term appears to be beyond the scope of the present methodology. 
 \end{rem}
 
 \subsection{Strategy}
In this paper, our contribution lies in two directions: 
\begin{itemize}
    \item We establish the pathwise uniqueness  of the particle system $(\bV^N_t=(V_t^1,\dots,V_t^N))_{t\ge 0}$ solving \eqref{mainsde} (Theorem \ref{kpst}). The existence of this system, as well as the conservation of momentum and energy and the monotonicity of the Fisher information of its law, has already been established in \cite[Theorem 4.1]{Fournier2025} and \cite[Proposition 4.11]{Tabary26}.
    
   \vip
    
    \item  We investigate the quantitative propagation of chaos for moderately soft potentials, namely for $\gamma\in(-1,0)$ (Theorem \ref{th: prop chaos}), which is considerably more challenging than the case $\gamma\ge0$.  To the best of our knowledge, very few quantitative results  are available in this regime. We obtain the explicit convergence rate $N^{-1/3}+N^{-\ell(q,\gamma)}$, although is not sharp.  Existing works concern mainly Nanbu’s particle system (a Kac-type model with a less direct physical interpretation): see \cite{MR3784497} for $\gamma\in(-1,0)$ with a quantitative convergence rate. 
\end{itemize}
\vip

To prove the quantitative propagation of chaos for $\gamma\in(-1,0)$, we combine ideas from \cite{MR3784497}, devoted to Nanbu's particle system for the moderately soft potential, and from \cite{MR3769742,lxz}, which treat the Kac particle system for $\gamma\in[0,1]$. A crucial ingredient is the decay of the Fisher information for the Boltzmann equation established in \cite{IMBERT24}, which enables us to control the
singularity of the interaction. The starting point is a coupling between the particle system and a family of real particles governed by the Boltzmann equation on the same probability space. To construct this coupling, we use an optimal transport map between $f_t$ and the empirical measure of the real particles, see Lemma \ref{ww2p}. Since the particles in the coupled system are not independent, we then introduce a second coupling with a family of independent  particles, see Lemma \ref{coupling1},  thereby obtaining a suitable decoupling procedure.

\vip
The main difficulty stems from the singular factor involving a negative power of the relative velocity. To handle this singularity, we need precise estimates on the approximation of empirical measures, following the strategy developed in \cite[Section 5]{MR3784497}, see also \cite[Proposition 5.5]{MR3572320}. However, unlike in these works, the particle system under consideration is not exchangeable. Consequently, the standard exchangeability arguments cannot be applied directly. To overcome this obstacle, we partition the particle system into several blocks. Although global exchangeability is lost, exchangeability is preserved separately among particles inside a given block and among particles outside that block. This partial exchangeability allows us to analyze the whole system through the particles contained in a given block and to derive the required approximation estimates for the corresponding empirical measures.
\vip

Finally, the finite Fisher information assumption provides additional regularity for solutions to the Boltzmann equation due to the decay of the Fisher information, allowing us to considerably simplify several technical arguments compared with those in \cite[Section 5]{MR3784497}. More precisely, it yields the uniform estimate $\sup_{t\ge 0}\|f_t^K\|_{L^2}<\infty$, whereas only local in time bound are available in \cite{MR3784497}. This improvement eliminates the need to partition the time interval and to separately analyze the small and large jumps of the process.

 \subsection{Background and related works} 
 The Boltzmann equation is introduced in \cite{boltzmann} to describe the statistical behavior of a dilute gas out of equilibrium. One of the core questions in kinetic theory is its rigorous derivation from Newtonian particle dynamics, a problem that has remained a major challenge for more than a century. The first breakthrough was due to Lanford \cite{Lanford}, who derived the Boltzmann equation from hard sphere dynamics, but only on a short time interval. Very recently, Deng, Hani and Ma \cite{dengma} established a long time derivation of the Boltzmann equation from hard sphere dynamics, thereby extending Lanford’s result to macroscopic times.

\vip

Alternatively, Kac introduced his seminal program to derive the spatially homogeneous Boltzmann equation from an $N$-particle binary collision process, and justified propagation of chaos in the Maxwell case \cite{MR84985}. Over the past decades, Kac’s program has been extensively developed, especially for hard potentials and Maxwell molecules. Quantitative results were first obtained in \cite{MR334788,MR224348} for bounded collision kernels, and in \cite{MR753814} for unbounded kernels in the hard spheres case. Using Tanaka’s coupling method, Graham and Méléard \cite{MR1428502} derived the first quantitative propagation of chaos estimates for the Boltzmann equation with cutoff in the Maxwell setting. A major advance was achieved by Mischler and Mouhot \cite{MR3069113}, who proved strong, uniform-in-time (though not sharp) quantitative results via a purely abstract analytic method. In \cite{MR3069113}, the rate of chaos is of order $N^{-1/(6+\delta)}$ (for any $\delta>0$) for Maxwell molecules, and of order $(\log N)^{-r}$ for some $r>0$ in the hard spheres case. This rate for Maxwell molecules was later substantially improved by Cortez and Fontbona \cite{MR3769742} to an almost \emph{optimal} rate of $N^{-1/3}$. For hard potentials, Heydecker \cite{MR4419606} obtained a rate of order $(\log N)^{1-1/\nu}$, and more recently Liu, Xu and Zhang \cite{lxz} proved a rate of order $N^{-1/3}$ under stronger assumptions on the initial data.
\vip

Unlike in the hard potential case, {\it propagation of chaos} for soft potentials for both the Landau and Boltzmann equations remains a challenging problem for a long time. The main difficulties stem from the limited regularity of solutions and the singularities appearing both in the relative velocity and in the collision angle.
A major breakthrough was recently achieved by Guillen and Silvestre \cite{Guillen}, who established Fisher  information estimates for the Landau equation. This was followed by another remarkable breakthrough of Imbert, Silvestre, and Villani \cite{IMBERT24}, who proved the decay of Fisher information for the Boltzmann equation; see also \cite{villani24}. These results revealed the fundamental role of Fisher information in controlling the singular structure of kinetic equations and have become a key ingredient in several recent advances on propagation of chaos. Building upon this new regularity framework,  Fournier and Mischler \cite{Fournier2025} recently established the propagation of chaos in the case of moderate soft potentials ($\gamma\in(-2,0)$) via a martingale method. And more recently, Tabary \cite{Tabary26} extended this result to the very soft potential case ($\gamma\in(-3,-2]$) by developing new estimates on the dissipation of the Fisher information. As a consequence, propagation of chaos is now known for the whole physically relevant soft potential range ($\gamma\in(-3,0)$).

\vip
Besides the Kac particle system, another closely related stochastic approximation of the Boltzmann equation, namely the Nanbu particle system \cite{Nanbu83}, has also been studied, where only one particle is updated at each collision. Notably, propagation of chaos for this system has been proved by Fournier and Mischler \cite{MR3456347} for $\gamma\in[0,1]$, Xu \cite{MR3784497} for $\gamma\in(-1,0)$ and Salem \cite{Salem19} for more singular but non-physical potentials.

\vip

We also mention several results on propagation of chaos for the Landau equation. In \cite{MR3621429}, Fournier and Guillin used a coupling strategy to obtain quantitative chaos rates for the Landau equation in the hard potentials and Maxwell molecules cases. In \cite{MR3572320}, Fournier and Hauray treated the case of singular soft potentials. More recently, in \cite{Carrillo}, the bound for the relative entropy between the joint law of the particle system and the tensorized law of the Landau equation was established for Maxwellian Molecules. For hard potentials, Guo \cite{Guo2025} obtained  the propagation of chaos for the Landau equation using BBGKY hierarchy method. Carrillo and Guo \cite{guo25} obtained Fisher-information decay for a Kac model associated with the Landau equation. Propagation of chaos for Kac’s particle system converging to the Landau equation has been established by Tabary \cite{Tabary25} and Feng–Wang \cite{fengwang25} in the very soft potentials and Coulomb regimes.

\subsection{Plan of the paper}

The paper is organized as follows. In Section 2, we collect several preliminary
computations and estimates that will be used throughout the paper. Section 3 is
devoted to the proof of pathwise uniqueness  for the Kac particle system, stated in
Theorem \ref{kpst}.
\vip
Sections 4 and 5 are devoted to the construction of the Boltzmann processes
needed for the coupling argument. In Section 4, we recall and prove several
properties of the solution to the cutoff Boltzmann equation. In Section 5, we
construct the coupled cutoff Boltzmann processes and introduce the corresponding
decoupled processes. In contrast with \cite{MR3769742,MR3621429,lxz}, the
decoupled processes $\tW$ are used not only to control the Wasserstein error in
Proposition \ref{thep}, but also to estimate the negative moments of the coupled
processes $W$.
\vip
Section 6 is inspired by the arguments of \cite{MR3572320,MR3784497} but simpler, where we establish  the approximation of the $L^2$ norm of the empirical measure. Finally, in Section 7, we prove the quantitative propagation of chaos estimate
stated in Theorem \ref{th: prop chaos}.

\section{A priori estimates}
In this section, we collect several a priori estimates that will be used later. 
The following result is a direct consequence of \cite[Lemma 2.3]{MR3784497}, after a minor adaptation to the present definition of  $c_K(v,v_*,z,\varphi)$  in \eqref{CK}.

\blem\label{lem:c} 
Assume \eqref{con}, \eqref{pa-range} and recall  \eqref{atc} and  \eqref{CK}. 
 \vip
{\em (i) (Lemma 2.3-(2.8), \cite{MR3784497})} There is a constant $C>0$ such that, for any $v, v_*\in \R^3$,
\begin{align*}
\int_0^\infty  \int_0^{2\pi} |c(v,v_*,z,\varphi)| d\varphi  dz \le C|v-v_*|^{\gamma+1},\quad
\int_0^\infty  \int_0^{2\pi} |c(v,v_*,z,\varphi)|^2 d\varphi  dz \le C |v-v_*|^{\gamma+2}.
\end{align*}

{\em(ii)}  There is a constant $C>0$ such that, for any $v, v_*\in \R^3$, any $K\ge1$,
\begin{align*}
  \int_0^\infty  \int_0^{2\pi} |c_K(v,v_*,z,\varphi)| d\varphi  dz \le C|v-v_*|^{\gamma+1},\;
  \int_0^\infty \int_0^{2\pi} |c_K(v,v_\ast,z,\varphi)|^2 \,d\varphi\,dz
    \le C\,|v-v_\ast|^{\gamma+2}.
\end{align*}

\elem

\begin{lem}\label{lem: 3ccut}
 Assume \eqref{con}, \eqref{pa-range} and let $K\geq 1$. Recall \eqref{atc} and  \eqref{CK}.  Then there exists a constant $C>0$ independent of $K$, such that for any $v, v_*,\tilde{v}, \tilde{v}_*\in \R^3$,
 \begin{align}
 \Big|\int_0^\infty\int_0^{2\pi}\Delta_{c_K}d\varphi dz\Big|&\le C |v-v_*-(\tilde v-\tilde v_*)| \min\{ |v-v_*|^\gamma,|\tilde v-\tilde v_*|^\gamma\}\label{ineq: KKp22cut1},\\
 \Big|\int_0^\infty\int_0^{2\pi}\tilde\Delta_{c,c_K}d\varphi dz\Big| &\le C |v-v_*-(\tilde v-\tilde v_*)| \min\{ |v-v_*|^\gamma,|\tilde v-\tilde v_*|^\gamma\}\nonumber\\
&\hskip3.5cm+CK^{1-1/\nu}(1+|v-v_*|^{1+\gamma})\label{ineq: KKp22cut11},\\
\int_0^\infty\int_0^{2\pi}\big|\Delta_{c_K}\big|d\varphi dz &\le C |v-v_*-(\tilde v-\tilde v_*)|\min\{ |v-v_*|^\gamma,|\tilde v-\tilde v_*|^\gamma \}\notag\\
&\hskip3.5cm+CK^{1-1/\nu}(1+|v-v_*|+|\tv-\tv_*|)\label{ineq: KKp22cut3},
\end{align}
\begin{align}
\int_0^\infty\int_0^{2\pi}\big|\Delta_{c_K}\big|^2d\varphi dz &\le C |v-v_*-(\tilde v-\tilde v_*)|^2\min\{ |v-v_*|^\gamma,|\tilde v-\tilde v_*|^\gamma\}\notag\\
&\hskip3cm+CK^{1-1/\nu}(1+ |v-v_*|^{2}+|\tv-\tv_*|^{2})\label{ineq: KKp22cut},\\
\int_0^\infty\int_0^{2\pi}\big|\tilde\Delta_{c,c_K}\big|^2\, d\varphi dz &\le C |v-v_*-(\tilde v-\tilde v_*)|^2\min\{ |v-v_*|^\gamma,|\tilde v-\tilde v_*|^\gamma\}\notag\\
&\hskip3cm+CK^{1-1/\nu}(1+ |v-v_*|^{2}+|\tv-\tv_*|^{2})\label{ineq: KKp22cut2},
\end{align}
where  
\begin{align*}
&\Delta_{c_K}:=c_K(v,v_*,z,\varphi)-c_K(\tilde{v}, \tilde{v}_*,z,\varphi+\varphi_0(v-v_*, \tilde{v}-\tilde{v}_*)),\\
& \tilde\Delta_{c,c_K}:=c(v,v_*,z,\varphi)-c_K(\tilde{v}, \tilde{v}_*,z,\varphi+\varphi_0(v-v_*, \tilde{v}-\tilde{v}_*)).
 \end{align*}
\end{lem}

\begin{rem}\label{inequ:c}
    Letting $K\to\infty$ in \eqref{ineq: KKp22cut1} and \eqref{ineq: KKp22cut}, we recover the corresponding estimates for the non-cutoff collision increment. More precisely, denote by  $\Delta_c:=c(v,v_*,z,\varphi)-c\bigl(\tilde v,\tilde v_*,z,
\varphi+\varphi_0(v-v_*,\tilde v-\tilde v_*)\bigr)$, then
    \begin{align*}
        &\Big|\int_0^\infty\int_0^{2\pi}\Delta_{c}d\varphi dz\Big|\le C |v-v_*-(\tilde v-\tilde v_*)| \min\{ |v-v_*|^\gamma,|\tilde v-\tilde v_*|^\gamma\},\\ 
&\int_0^\infty\int_0^{2\pi}\big|\Delta_{c}\big|^2d\varphi dz\le C |v-v_*-(\tilde v-\tilde v_*)|^2\min\{ |v-v_*|^\gamma,|\tilde v-\tilde v_*|^\gamma\}.
    \end{align*}
\end{rem}

\begin{proof}[Proof of Lemma \ref{lem: 3ccut}]  The proof is adapted from that of  \cite[Lemma 3.1]{MR3456347}, see also \cite[Lemma 2.3]{MR3784497}.
Recalling $G$ defined in \eqref{nota}, \eqref{atc}  and \eqref{CK}, we introduce the shortened notation $x_K=|v-v_*|\lor K^{1/\gamma}$, $\tilde{x}_K=|\tilde{v}-\tilde{v}_*| \lor K^{1/\gamma}$,  $\varphi_0=\varphi_0(v-v_*, \tilde{v}-\tilde{v}_*)$, $c=c(v,v_*,z,\varphi)$,  $c_K=c_K(v,v_*,z,\varphi)$ and $\tilde c_K=c_K(\tilde{v}, \tilde{v}_*,z,\varphi+\varphi_0)$. Then, 
 \[\Delta_{c_K}=c_K-\tilde c_K, \quad  \tilde\Delta_{c,c_K}=c-\tilde c_K,\quad  x_K^\gamma=|v-v_*|^\gamma\land K, \quad \tilde {x}_K^\gamma=|\tilde v-\tilde v_*|^\gamma\land K.
 \]

\vip
{\bf Step 1.}
Recall  \eqref{CK} that $c_K= -(v-v_*)\frac{1-\cos G(z/x_K^\gamma)}{2}
    +\frac{\sin G(z/x_K^\gamma)}{2}\Gamma(v-v_*,\varphi)$  if $z\le Kx_K^\gamma$, and $c_K=0$ otherwise. Since  $\int_0^{2\pi} \Gamma(v-v_*,\varphi)\, d\varphi=0$, we have 
\begin{align*}
    \int_0^\infty\int_0^{2\pi} c_Kd\varphi dz
    = -\pi(v-v_*)x_K^\gamma\int_0^K (1-\cos G(z))dz.
\end{align*}
Analogously,
\begin{align*}
\int_0^\infty\int_0^{2\pi} \tilde{c}_Kd\varphi dz&=-\pi(\tilde{v}-\tilde{v}_*)\tilde{x}_K^\gamma\int_0^K(1-\cos G(z)) dz,\\
\int_0^\infty\int_0^{2\pi} c \; d\varphi dz&=-\pi(v-v_*)|v-v_*|^\gamma\int_0^\infty(1-\cos G(z)) dz.
\end{align*}
Since $\int_0^K (1-\cos G(z)) dz \le \int_0^\infty (1-\cos G(z))dz < +\infty$ due to \eqref{con1}, it then follows immediately from  $\Delta_{c_K} = c_K - \tilde{c}_K$ that
 \begin{align*}
 \Big|\int_0^\infty\int_0^{2\pi}\Delta_{c_K}d\varphi dz\Big|
 \le  C\Big|(v-v_*)x_K^\gamma - (\tv-\tv_*)\tilde {x}_K^\gamma\Big|.
\end{align*}
Proceeding as in the proof of \cite[Lemma 2.3, P.8]{MR3784497}, we 
thereby establish \eqref{ineq: KKp22cut1}.

For \eqref{ineq: KKp22cut11},  noting that  $\tilde\Delta_{c,c_K}=c-\tilde c_K$  and $\Delta_{c_K}=c_K-\tilde c_K$, we have
\begin{align*}
  \Big|\int_0^\infty\int_0^{2\pi} \tilde \Delta_{c,c_K}d\varphi dz\Big|
    \le \Big|\int_0^\infty\int_0^{2\pi} \Delta_{c_K}d\varphi dz\Big|+\Big|\int_0^\infty\int_0^{2\pi} (c- c_K)d\varphi dz\Big|.
    \end{align*}
Since  $\int_K^\infty(1-\cos G(z)) dz\le CK^{1-2/\nu}$ by \eqref{con1} and $x_K^\gamma=|v-v_*|^\gamma\land K$, we then have
\begin{align*}
 &\Big|\int_0^\infty\int_0^{2\pi} (c-c_K)d\varphi dz\Big|\\
 =&\pi \Big|(v-v_*)\Big(|v-v_*|^\gamma\int_0^\infty(1-\cos G(z)) dz
 -x_K^\gamma\int_0^K(1-\cos G(z)) dz\Big)\Big|\\
    \le& \pi |v-v_*|\Big||v-v_*|^\gamma-x_K^\gamma\Big|\int_0^\infty(1-\cos G(z))dz+\pi|v-v_*|^{1+\gamma}\int_K^\infty(1-\cos G(z)) dz\\ 
    \le& CK^{1+1/\gamma}+C|v-v_*|^{1+\gamma}K^{1-2/\nu}.
\end{align*}
Due to \eqref{pa-range}, we have $K^{1+1/\gamma}\le K^{1-1/\nu}$ and $K^{1-2/\nu}\le K^{1-1/\nu}$, 
which together with \eqref{ineq: KKp22cut1} gives \eqref{ineq: KKp22cut11}.

\vip
{\bf Step 2.}
In this step, we prove  \eqref{ineq: KKp22cut3}. Denote by $$\theta_K:=G\big(z/(|v-v_*|^\gamma\land K)\big)=G\big(z/x_K^\gamma\big), \; \;\tilde\theta_K:=G\big(z/(|\tilde v-\tilde v_*|^\gamma\land K)\big)=G\big(z/\tilde {x}_K^\gamma\big),$$
and $\Gamma:=\Gamma(v-v_*,\varphi),\;\;  \tilde \Gamma:=\Gamma(\tv-\tv_*,\varphi+\varphi_0)$.
  It follows from \eqref{z} amd \eqref{CK} that,  for $z\le K\min\{x_K^\gamma, \tilde{x}_K^\gamma\}$,
\begin{align*}
    c_K-\tilde c_K =\big [-(v-v_*)+(\tilde v-\tilde v_*)\big] \Big(\frac{1-\cos\theta_K}{2} \Big)-& (\tilde v-\tilde v_*)\Big(\frac{\cos\tilde \theta_K-\cos \theta_K}{2} \Big)\\
    +&\frac{\sin\theta_K}{2}(\Gamma- \tilde \Gamma)+ \Big(\frac{\sin\theta_K-\sin\tilde \theta_K}{2} \Big) \tilde \Gamma.
\end{align*}
Since $\sin^2\theta\le \theta$, $|\tilde \Gamma|=|\tv-\tv_*|$ and \eqref{Gamma},  it follows that
 \begin{align*}
    |c_K-\tilde c_K|\le |v-v_*-(\tilde v-\tilde v_*)|G(z/x_K^\gamma)
   + |\tilde v-\tilde v_*| \big|G(z/x_K^\gamma)-G(z/\tilde{x}_K^\gamma)\big|.
\end{align*}
Since $\int_0^\infty G(z)\,dz<\infty$  by \eqref{con1}, we see that 
\[ \int_0^\infty G(z/x_K^\gamma) dz
\le C x_K^\gamma \le C |v-v_*|^\gamma.\]
Moreover, by \eqref{ineq1},
\begin{align*}
\int_0^\infty \big| G\big(z/x_K^\gamma\big)-G\big(z/ \tilde{x}_K^\gamma\big)  \big| dz
\le c_5 | x_K^\gamma - \tilde{x}_K^\gamma|\le C\big| |v-v_*|^\gamma-|\tilde v-\tilde v_*|^\gamma \big|.
\end{align*}
  Using  $|a^\gamma-b^\gamma|\le C|a^{-1}-b^{-1}|\min\{a^{\gamma+1},b^{\gamma+1}\}$ for all $a,b>0$ when $\gamma\in(-1,0)$, we have 
  \begin{align}\label{G-gam}
  \int_0^\infty\big|G(z/x_K^\gamma)-G(z/\tilde{x}_K^\gamma)\big| dz
\le  C |v-v_*-(\tilde v-\tilde v_*)| |\tv-\tv_*|^{-1} | |v-v_*|^{\gamma}.
  \end{align}
Thus, 
\begin{align*}
    &\int_0^{K\min\{x_K^\gamma, \tilde{x}_K^\gamma\}}\int_0^{2\pi} |c_K-\tilde c_K|\, d\varphi dz \le C |v-v_*-(\tilde v-\tilde v_*)||v-v_*|^\gamma.
\end{align*}
By symmetry, we obtain
\begin{align}\label{half}
    \int_0^{K\min\{x_K^\gamma, \tilde{x}_K^\gamma\}}\int_0^{2\pi} |c_K-\tilde c_K| \, d\varphi dz
    \le C |v-v_*-(\tilde v-\tilde v_*)|\min\{|v-v_*|^\gamma,|\tv-\tv_*|^\gamma\}.
\end{align}

On the other hand, if $x_K^\gamma = \tilde{x}_K^\gamma$, then
 \begin{align*}
    |c_K-\tilde c_K|=0, \quad \text{for all}\; z\ge K\min\{x_K^\gamma, \tilde{x}_K^\gamma\},
\end{align*}
and the proof is complete. Otherwise,  by symmetry, we may assume that $x_K^\gamma<\tilde{x}_K^\gamma$, so that  $|v-v_*|^\gamma < \tilde{x}_K^\gamma\le K$. It follows that $x_K^\gamma=|v-v_*|^\gamma,$
and therefore
$c_K=0$ for  $z\ge K|v-v_*|^\gamma$. Then recalling \eqref{num},  we have
\begin{align*}
  \int^\infty_{K\min\{x_K^\gamma, \tilde{x}_K^\gamma\}} \int_0^{2\pi}|c_K-\tilde c_K|\, d\varphi dz=2\pi\int^\infty_{K|v-v_*|^\gamma}|\tilde c_K| dz \le C |\tv-\tv_*|^{1+\gamma}\int^\infty_{K|v-v_*|^\gamma\tilde{x}_K^{-\gamma}} G(z) dz,
\end{align*}
which  is bounded by as follows (using \eqref{con1}):
\begin{align*}
\frac{C}{K^{1/\nu-1}} |v-v_*|^{\gamma-\frac{\gamma}{\nu}}|\tv-\tv_*|^{1+\frac{\gamma}{\nu}}&\le \frac{C}{K^{1/\nu-1}}( |v-v_*|^{1+\gamma}+|\tv-\tv_*|^{1+\gamma}),
\end{align*}
due to  $\min\{\gamma-\frac{\gamma}{\nu}, 1+\frac{\gamma}{\nu}\}> 0$ and Young's inequality. We thus have 
\[ \int^\infty_{K\min\{x_K^\gamma, \tilde{x}_K^\gamma\}} \int_0^{2\pi}|c_K-\tilde c_K|\, d\varphi dz\le\frac{C}{K^{1/\nu-1}}(1+|v-v_*|+|\tv-\tv_*|),\]
which combining with \eqref{half} leads to  \eqref{ineq: KKp22cut3}.

\vip

{\bf Step 3.}
Let us denote 
$$I_K:=\int_0^{K\min\{x_K^\gamma, \tilde {x}_K^\gamma\}}\int_0^{2\pi} |c_K-\tilde c_K|^2 d\varphi dz, \quad  J_K:=\int^\infty_{K\min\{x_K^\gamma, \tilde {x}_K^\gamma\}}\int_0^{2\pi} |c_K-\tilde c_K|^2 d\varphi dz.$$
We first prove
\begin{align*}
I_K \le  C |v-v_*-(\tilde v-\tilde v_*)|^2 \min\{|v-v_*|^\gamma,|\tv-\tv_*|^\gamma\}.
\end{align*}
 For $ z\le K\min\{x_K^\gamma, \tilde {x}_K^\gamma\}$, proceeding as in the proof  \cite[Lemma 3.1, P. 603]{MR3456347}, we see that 
 $$I_K\le A_1^K+A_2^K,$$
 where 
 \begin{align*}
 &A_1^K=2|v-v_*||\tv-\tv_*|\int_0^{K\min\{x_K^\gamma, \tilde {x}_K^\gamma\}}\big(G(z/x_K^\gamma) -G(z/\tilde{x}_K^\gamma )\big)^2,\\
 &A_2^K=\pi \int_0^{K\min\{x_K^\gamma, \tilde {x}_K^\gamma\}} (v-v_*-\tv+\tv_*)\cdot\Big[(v-v_*)\big(1-\cos G(z/x_K^\gamma)\big)\\
 &\hskip7cm -(\tv-\tv_*)(1-\cos G\big(z/\tilde{x}_K^\gamma)\big)\Big]\, dz.
 \end{align*}
 First, using \eqref{ineq}, we immediately obtain
\begin{align*}
A_1^K &\le 2|v-v_*||\tv-\tv_*| \int_0^\infty\big(G(z/x_K^\gamma)-G(z/\tilde{x}_K^\gamma)\big)^2 dz\\
&\le  2c_4|v-v_*||\tv-\tv_*|\frac{\big(x_K^\gamma-\tilde{x}_K^\gamma\big)^2}{x_K^\gamma+\tilde{x}_K^\gamma}\\
&\le C|v-v_*||\tv-\tv_*|(x_K\tx_K)^{-1}(x_K-\tilde x_K)^2 \min\{x_K^{\gamma},\tilde x_K^{\gamma}\}\\
    &\le C |v-v_*-(\tilde v-\tilde v_*)|^2\min\{ |v-v_*|^\gamma,|\tilde v-\tilde v_*|^\gamma\}.
\end{align*}
To justify the third inequality,  since  $|x_K^\gamma -  \tilde{x}_K^\gamma|\le C|x_K^{-1} - \tilde{x}_K^{-1}|\min\{x_K^{\gamma+1}, \tilde{x}_K^{\gamma+1}\}$. Hence,
\[
\frac{(x_K^\gamma - \tilde{x}_K^\gamma)^2}{x_K^\gamma+\tilde x_K^\gamma}
\le C\frac{(x_K-\tilde x_K)^2} {(x_K\tilde x_K)^2(x_K^\gamma+\tilde x_K^\gamma)}
(\min\{x_K,\tilde x_K\})^{2\gamma+2}.
\]
Moreover, a simple verification shows that 
\[
x_K^\gamma+\tilde x_K^\gamma
\ge \max\{x_K^\gamma,\tilde x_K^\gamma\}=(\min\{x_K,\tilde x_K\})^\gamma,
\]
and
\[ 
(\min\{x_K,\tilde x_K\})^{\gamma+2} \le x_K\tilde x_K(\max\{x_K,\tilde x_K\})^\gamma.\]
Consequently, we obtain
\[
\frac{(x_K^\gamma-\tilde x_K^\gamma)^2}{x_K^\gamma+\tilde x_K^\gamma}(\min\{x_K,\tilde x_K\})^{2\gamma+2}
\le C(x_K\tilde x_K)^{-1}(x_K-\tilde x_K)^2\min\{x_K^\gamma,\tilde x_K^\gamma\},
\]
which  precisely implied  the desired inequality.

\vip

On the other hand, since $1-\cos \theta\le \theta^2/2$, it follows that 
\begin{align*}
   & \Big|(v-v_*)\big(1-\cos G(z/x_K^\gamma)\big)-(\tv-\tv_*)\big(1-\cos G(z/ \tilde{x}_K^\gamma)\big)\Big|\\
    =& |v-v_*-\tv+\tv_*|\big(1-\cos G(z/\tilde{x}_K^\gamma)\big)
    +|v-v_*||\cos G(z/ \tilde{x}_K^\gamma)-\cos G(z/x_K^\gamma)|\\
    \le& C |v-v_*-\tv+\tv_*| G^2(z/ \tilde{x}_K^\gamma)+|v-v_*|| G\big(z/ \tilde{x}_K^\gamma\big)-G\big(z/x_K^\gamma\big)|.
\end{align*}
Furthermore, \eqref{con1} implies that
$\int_0^\infty G^2(z/ \tilde{x}_K^\gamma) dz\le C\tilde{x}_K^\gamma$, which together with   \eqref{G-gam} yields
\[A_2^K\le C |v-v_*-(\tilde v-\tilde v_*)|^2|\tilde v-\tilde v_*|^\gamma,\]
By symmetry,  
\[A_2^K\le C |v-v_*-(\tilde v-\tilde v_*)|^2\min\{ |v-v_*|^\gamma,|\tilde v-\tilde v_*|^\gamma\}.\]
Therefore,
\begin{align*}
I_K\le A_1^K+A_2^K\le C |v-v_*-(\tilde v-\tilde v_*)|^2\min\{ |v-v_*|^\gamma,|\tilde v-\tilde v_*|^\gamma\}
\end{align*}
as desired.
\vip
Finally, we estimate $J_K$.  The argument is analogous to that used  in Step 2 for
$$\int_{K\min\{x_K^\gamma,\tilde x_K^\gamma\}}^\infty
\int_0^{2\pi}|c_K-\tilde c_K|\,d\varphi\,dz.$$
If $x_K^\gamma=\tilde x_K^\gamma$, the conclusion is immediate. Otherwise, by symmetry, assume that $x_K^\gamma<\tilde x_K^\gamma$, so that $x_K^\gamma=|v-v_*|^\gamma$ and  $c_K=0$ for $z\ge K|v-v_*|^\gamma$.
Recalling \eqref{num} and the definition of \(\tilde c_K\), we obtain
\[
 \int^\infty_{K|v-v_*|} \int_0^{2\pi} |\tilde{c}_K|^2 d\varphi dz= \pi|\tv-\tv_*|^2  \int^\infty_{K|v-v_*|}  \big(1-\cos G(z/\tilde{x}_K^\gamma)\big)\indiq_{\{z\le K\tilde{x}_K^\gamma\}}\, dz,
\]
Using \eqref{con1}, the right hand side is bounded by 
\begin{align*}
   C |\tv-\tv_*|^{2+\gamma}  \int^\infty_{K|v-v_*|^\gamma/ \tilde{x}_K^{\gamma}} G^2(z)\,  dz
\le  \frac{C}{K^{2/\nu-1}} |v-v_*|^{\gamma-\frac{2\gamma}{\nu}}|\tv-\tv_*|^{2+\frac{2\gamma}{\nu}}.
\end{align*}
Since $\min\{\gamma-\frac{2\gamma}{\nu},  2+\frac{2\gamma}{\nu}\}> 0$, an application of   Young's inequality yields 
\[ J_K\le \frac{C}{K^{2/\nu-1}}( |v-v_*|^{2+\gamma}+|\tv-\tv_*|^{2+\gamma})\le \frac{C}{K^{1/\nu-1}}(1+ |v-v_*|^{2}+|\tv-\tv_*|^{2}).\]
Combining the estimates for $I_K$ and $J_K$,  we conclude that \eqref{ineq: KKp22cut} holds.

\vip

{\bf Step 4.} Finally, we establish \eqref{ineq: KKp22cut2}. First,  recalling  \eqref{num}, we have 
\[|c|^2=\frac{1-\cos G(z/|v-v_*|^\gamma)}{2}|v-v_*|^2,\; \; |c_K|^2=\frac{1-\cos G(z/x_K^\gamma)}{2}|v-v_*|^2\indiq_{\{z\le Kx_K^\gamma\}}.\] 
Observe that if $|v-v_*|^\gamma\le K$, then $c=c_K$ for $z\le Kx_K^\gamma$, while  $c_K=0$ for $z> Kx_K^\gamma$. 
Consequently, under the condition $|v-v_*|^\gamma\le K$, we have
\begin{align*}
\int_0^\infty\int_0^{2\pi}\Big|c-c_K\Big|^2d\varphi dz &= \int_{Kx_K^\gamma}^\infty\int_0^{2\pi}|c|^2d\varphi\, dz\\
&\le C |v-v_*|^2\int_{K|v-v_*|^\gamma}^\infty G^2(z/|v-v_*|^\gamma)\,dz \\
&\le C |v-v_*|^{2+\gamma} K^{1-2/\nu},
\end{align*}
where the last inequality follows from \eqref{con1} via the estimate  $\int_K^\infty G^2(z) \, dz \le CK^{1-2/\nu}$.

\vip

Conversely, when $|v-v_*|^\gamma> K,$ equivalently $|v-v_*|<K^{1/\gamma}$, Lemma \ref{lem:c} yields
\begin{align*}
\int_0^\infty\int_0^{2\pi}\big|c-c_K \big|^2d\varphi dz\le 2\int_0^\infty\int_0^{2\pi}(|c|^2+|c_K|^2)\, d\varphi dz \le C|v-v_*|^{2+\gamma}\le CK^{1+2/\gamma}.
\end{align*}
Hence, for all $v,v_*\in\R^3,$ we obtain 
\begin{align*}
    \int_0^\infty\int_0^{2\pi}\big|c-c_K \big|^2d\varphi dz\le CK^{1+2/\gamma}+C |v-v_*|^{2+\gamma} K^{1-2/\nu}.
\end{align*}
Gathering this with \eqref{ineq: KKp22cut}, we deduce that 
\begin{align*}
\int_0^\infty\int_0^{2\pi}\big| \tilde\Delta_{c,c_K} \big|^2 \, d\varphi dz&\le 2\int_0^\infty\int_0^{2\pi}\Big( \big|c-c_K\big|^2 + \big|\Delta_{c_K}\big|^2 \Big)d\varphi dz\\
    &\le C |v-v_*-(\tilde v-\tilde v_*)|^2\min\{ |v-v_*|^\gamma,|\tilde v-\tilde v_*|^\gamma\}\\
    &\hskip2cm+\frac{C}{K^{-1-2/\gamma}}+\frac{C}{K^{2/\nu-1}}(1+|v-v_*|^{2}+|\tv-\tv_*|^{2}).
\end{align*}
In addition, in view of\eqref{pa-range}, we find that $\max\{K^{1+2/\gamma}, K^{1-2/\nu}\}\le K^{1-1/\nu}$. Therefore,  \eqref{ineq: KKp22cut2} follows.

\end{proof}

\begin{lem}\label{fundest}
Under the same condition as Lemma \ref{lem: 3ccut}, we recall \eqref{atc} and  \eqref{CK}.  Then there exists  $C>0$ independent of $K$, such that  for any $v,v_*,\tv,\tv_*,y \in \rd$, and any $K\in [1,\infty)$,
\begin{align}\label{ineq: vc-vcK}
&\int_0^\infty \int_0^{2\pi} \Big( 
\big|v+c(v,v_*,z,\varphi)-\tv-c_K(\tv,\tv_*,z,\varphi+\hat \varphi) \big|^2 - |v-\tv|^2
\Big) d\varphi   dz\nonumber\\
\leq& C|v-\tilde v|^2 (1+|\tilde v-\tilde v_*-y|^\gamma)
   +C\Big(|y|^{2+2\gamma}+|y|^{2+\gamma}+|v_*-\tilde v_*|^2 |\tilde v-\tilde v_*-y|^\gamma\Big)\nonumber\\
   &\hskip3.87cm\quad+CK^{1-1/\nu}(1+|v|^2+|\tv|^2+|v_*|^2+|\tv_*|^2+|y|^2),
\end{align}
where $\hat \varphi:=\varphi^1+\varphi^2+\varphi^3$ with
\begin{align*}
\varphi^1:=&\varphi_{0}\big( v-v_*-y,
v-v_*\big),\\
\varphi^2:=&\varphi_{0}\big(\tv-\tv_*-y,
v-v_*-y\big),\\
\varphi^3:=&\varphi_{0}\big(\tv-\tv_*,\tv-\tv_*-y\big).
\end{align*}
\end{lem}

\begin{proof}
For simplification, we introduce
\begin{align*}
&c_v:=c(v,v_*, z,\varphi),\\
&c^1_v:=c(v,v_*+y, z,\varphi+\varphi^1),\\
&c^{12}_{K,\tv}:=c_K(\tv,\tv_*+y, z,\varphi+\varphi^1+\varphi^2),\\
&\hat c_{K,\tv}:=c_K(\tv,\tv_*, z, \varphi+\hat\varphi),
\end{align*}
and denote by 
\[I:=\int_0^\infty \int_0^{2\pi} \Big( 
\big|v+c_v-\tv-\hat c_{K,\tv} \big|^2 - |v-\tv|^2
\Big) d\varphi   dz.\]
Then, 
\begin{align*}
   I&= \int_0^\infty \int_0^{2\pi} \Big( 
2(v-\tv)\cdot(c_v-\hat c_{K,\tv})+|c_v-\hat c_{K,\tv}|^2
\Big) d\varphi   dz\\
&\le 2|v-\tv| (I_{11}+I_{12}+I_{13})+4(I_{21}+I_{22}+I_{23}),
\end{align*}
where 
\begin{align*}
&I_{11}:=  \Big| \int_0^\infty \int_0^{2\pi}(c_v-c^1_v)d\varphi dz\Big|,\quad\quad \quad
I_{12}:= \Big|\int_0^\infty\int_0^{2\pi}(c^1_v-c^{12}_{K,\tv})d\varphi dz\Big|,\\
&I_{13}:=\Big|\int_0^\infty\int_0^{2\pi}(c^{12}_{K,\tv}-\hat c_{K,\tv})d\varphi dz\Big|,\quad
I_{21}:=\int_0^\infty \int_0^{2\pi} |c_v-c^1_v|^2d\varphi   dz,\\
&I_{22}:=  \int_0^\infty \int_0^{2\pi} |c^1_v-c^{12}_{K,\tv}|^2d\varphi   dz,  \hskip0.8cm I_{23}:= \int_0^\infty \int_0^{2\pi} |c^{12}_{K,\tv}-\hat c_{K,\tv}|^2d\varphi   dz.
\end{align*}
Recalling the elementary inequality 
\begin{align*}
    \max\{|a|,|a+b|\}\ge |b|/2, \quad \forall\,   a,b\in\R^3,
\end{align*}
and using that $\gamma \in (-1,0)$, we obtain
  \beq\label{ineq:ab}
  \min\{|a|^\gamma,|a+b|^\gamma\}\le 2^{-\gamma}|b|^\gamma.
  \eeq  
Applying Remark \ref{inequ:c}, \eqref{ineq:ab} and \eqref{ineq: KKp22cut1},  we obtain
\[
   I_{11}\le C|y|\min\{|v-v_*|^{\gamma},|v-v_*-y|^{\gamma}\}\le C|y|^{1+\gamma},
 \]
 and
  \[
I_{13}\le C|y|\min\{|\tv-\tv_*|^{\gamma},|\tv-\tv_*-y|^{\gamma}\}\le C|y|^{1+\gamma}.
\]
Furthermore, \eqref{ineq: KKp22cut11} yields
\[I_{12}\le C|v-v_*-(\tilde v-\tilde v_*)| \min\{ |v-v_*-y|^\gamma,|\tilde v-\tilde v_*-y|^\gamma\}
+CK^{1-1/\nu}(1+|v-v_*-y|^{1+\gamma}).\]
Hence,
\begin{align*}
  &2|v-\tv| (I_{11}+I_{12}+I_{13})\\
    &\le C|v-\tv|\Big[|y|^{1+\gamma}+|v-v_*-(\tilde v-\tilde v_*)| \min\{ |v-v_*-y|^\gamma,|\tilde v-\tilde v_*-y|^\gamma\}\\
&\hskip5cm +CK^{1-1/\nu}(1+|v-v_*-y|^{1+\gamma})\Big]\\
   & \le C|v-\tilde v|^2 (1+|\tilde v-\tilde v_*-y|^\gamma)
   +C\Big(|y|^{2+2\gamma}+|v_*-\tilde v_*|^2 |\tilde v-\tilde v_*-y|^\gamma\Big)\\
   &\hskip5cm+\frac{C(1+|\tv|^2+|v|^2+|v_*|^2+|y|^2)}{K^{1/\nu-1}}.
\end{align*}
Similarly, utilizing  Remark \ref{inequ:c}, \eqref{ineq:ab} and \eqref{ineq: KKp22cut}, we have
\begin{align*}
    I_{21}\le C|y|^2\min\{|v-v_*|^{\gamma},|v-v_*-y|^{\gamma}\}\le C |y|^{2+\gamma},
    \end{align*}
and 
\begin{align*}
 I_{23}&\le  C|y|^2\min\{|\tv-\tv_*|^{\gamma},|\tv-\tv_*-y|^{\gamma}\}
+\frac{C}{K^{1/\nu-1}}(1+ |\tv-\tv_*-y|^{2}+|\tv-\tv_*|^{2})\\
 &\le C|y|^{2+\gamma}+\frac{C}{K^{1/\nu-1}}(1+ |\tv-\tv_*-y|^{2}+|\tv-\tv_*|^{2}).
\end{align*}
In addition, \eqref{ineq: KKp22cut2} implies 
  \begin{align*}
  I_{22}&\le C|v-v_*-(\tilde v-\tilde v_*)|^2 \min\{ |v-v_*-y|^\gamma,|\tilde v-\tilde v_*-y|^\gamma\}\\
&\hskip5cm+\frac{C}{K^{1/\nu-1}}(1+ |v-v_*-y|^{2}+|\tv-\tv_*-y|^{2}).
\end{align*}
Combining the above estimates, we deduce that
\begin{align*}
  & 4(I_{21}+I_{22}+I_{23})\\
     & \le C\Big(|y|^{2+\gamma}+|v-v_*-(\tilde v-\tilde v_*)|^2 \min\{ |v-v_*-y|^\gamma,|\tilde v-\tilde v_*-y|^\gamma\}\Big)\\
  &\hskip5cm+\frac{C(1+|v|^2+|\tv|^2+|v_*|^2+|\tv_*|^2+|y|^2)}{K^{1/\nu-1}}\\
   &\le C|v-\tilde v|^2 |\tilde v-\tilde v_*-y|^\gamma
  +C\Big(|y|^{2+\gamma}+|v_*-\tilde v_*|^2 |\tilde v-\tilde v_*-y|^\gamma\Big)\\
   &\hskip5cm+\frac{C(1+|v|^2+|\tv|^2+|v_*|^2+|\tv_*|^2+|y|^2)}{K^{1/\nu-1}}.
\end{align*}
Collecting all the above estimates, we conclude that \eqref{ineq: vc-vcK} holds.
\end{proof}

\section{Uniqueness}

Consider a weak solution $(\bV^N_t)_{t\in[0,T]}$  of  \eqref{mainsde}, let $F_t^N$ be its law for $t\in[0,T]$. In this section, we prove that $(\bV^N_t)_{t\in[0,T]}$ is pathwise unique. We first investigate some properties of $\bV$ via the Fisher information of its law  $(F_t^N)_{t\in[0,T]}$.  
\vip

\subsection{\texorpdfstring{Some properties of $\bV$}{Some properties of V}}
Let us first recall the following result.

\begin{theo}[Theorem 4.1, \cite{Fournier2025}]\label{Fini-Fisher}
Assume \eqref{con} and \eqref{pa-range}. Let $f_0\in L^1(\bb{R}^{3})\cap \ca{P}_{2}(\bb{R}^{3})$ have finite Fisher information, i.e. $I_1(f_0)<\infty$.
For  $N\ge2$, there exists, an exchangeable solution $(\bV^N_t)_{t\ge 0}=(V_t^1,\dots,V_t^N)_{t\ge0}$  to \eqref{mainsde}.
Moreover, the system is almost surely conservative: for all $t\ge 0$,
\[
\sum_{i=1}^N V^{i}_t=\sum_{i=1}^N V^{i}_0,
\qquad
\sum_{i=1}^N|V^{i}_t|^2=\sum_{i=1}^N|V^{i}_0|^2.
\]
Furthermore, for any $t\ge0$, let $F_t^N\in\ca{P}(\bb{R}^{3N})$ be the law of $\bV^N_t=(V_t^1,\dots,V_t^N)$, then $F_0^N=f_0^{\otimes N}$ and  $F_t^N$ solves equation \eqref{eq:  wmaster} and has finite Fisher information:
\begin{align}\label{Fis}
        I_N(F_t^N)\le I_N(F_0^N)=NI_1(f_0).
    \end{align}
\end{theo}

Once the Fisher information of the law of the particle system is controlled, the proximity of two particles can be controlled as well. This idea originated in \cite{FHM14}.
\begin{lem}\label{cor: -2}
    Under the same conditions as Theorem \ref{kpst}, for $N\ge2$, let $F_t^N$ be the law of a solution $\bV^N_t$ of \eqref{mainsde}, starting from  i.i.d $f_0$-distributed initial random variables $(V_0^i)_{i=1,\dots,N}$, and $F^{N:2}_{t}$ be  the $2$-marginal of $F^{N}_{t}$ for $t\in[0,T]$. Then we have for $t\in[0,T]$,
    \begin{align*}
        \int_{\R^6}|v_1-v_2|^{-2}F_t^{N:2} d\bv\le \frac{2}{N} I_N(F^N_0)=2I_1(f_0).
    \end{align*}
    Equivalently, 
       \begin{align*}
        \E[|V_t^i-V_t^j|^{-2}]\le \frac{2}{N} I_N(F^N_0)=2I_1(f_0), \quad  1\le i\ne j\le N.
    \end{align*}
\end{lem}
\begin{proof}
    Without loss of generality, let $(i,j)=(1,2)$. Then we manipulate the changing of variables defined by the following transformation
    $$w_1:=v_1-v_2, v_2:=v_2, \forall\,  v_1, v_2\in\R^3,$$
whose Jacobian  satisfies $| \frac{\partial w_1\partial v_2}{\partial v_1 \partial v_2}|=1$. 
Therefore,
    \begin{align*}
    \E[\abs{V^{1}_s-V^{2}_s}^{-2} ]&= \int_{\R^6} |v_1-v_2|^{-2} F^{N:2}_{s} (v_1,v_2) dv_1 dv_2\\
     &=\int_{\R^6} |w_1|^{-2} F^{N:2}_{s} (w_1+v_2,v_2) dw_1 dv_2.
    \end{align*}
Applying Hardy's inequality in the $w_1$-variable yields
\[ \E[\abs{V^{1}_s-V^{2}_s}^{-2} ]  \le 4 \int_{\R^6} \Big|\nabla_{w_1} \sqrt{F^{N:2}_{s} (w_1+v_2,v_2)}\Big|^2dw_1  dv_2.\]
Since $$\Big|\nabla_{w_1} \sqrt{F^{N:2}_{s} (w_1+v_2,v_2)} \Big|\le \Big|\nabla \sqrt{F^{N:2}_{s} (w_1+v_2,v_2)} \Big|,$$
we obtain 
  \begin{align*}   
  \E[\abs{V^{1}_s-V^{2}_s}^{-2} ]  \le 
   4 \int_{\R^6} \Big|\nabla \sqrt{F^{N:2}_{s} (v_1,v_2)}\Big|^2dv_1 dv_2=I_2(F_s^{N:2}).
 \end{align*}    
By \eqref{I-marg}, \eqref{tensor} and \eqref{Fis},  we have
\[ I_2(F_s^{N:2})\le \frac{2}{N} I_N(F^{N}_s)\le \frac{2}{N} I_N(F_0^N)=2I_1(f_0),\]
which completes the proof.
\end{proof}

Next, we establish the following lemma which plays a crucial role in the proof of pathwise uniqueness. 
\begin{lem}\label{lem: v-vf} 
 Assume \eqref{con} and  \eqref{pa-range}. For $t\in[0,T]$, let $F_t^N$ be the law of a solution $\bV^N_t$ to \eqref{mainsde}, and $F^{N:3}_{t}$ be the $3$-marginal of $F^{N}_{t}$. Assume that the initial distribution $f_0$ has finite Fisher information, i.e. $I_1(f_0)<\infty.$ Then for any $\alpha$ such that $0<\alpha\le -\gamma$,
\begin{align*}
&\E\Bigg[ 
\int_0^\infty\!\!\int_0^{2\pi}
\abs{c(V^{1}_s,V^{3}_s,z,\varphi)}
\,\big(\big|V^{1}_s-V^{2}_s+c(V^{1}_s,V^{3}_s,z,\varphi)\big|^{-1-\alpha}
+\big|V^{1}_s-V^{2}_s\big|^{-1-\alpha}\big)\,dz\,d\varphi\Bigg]\\
&\le C\Bigl(1+I_1(f_0)+m_2(f_0)\Bigr).
\end{align*}

\end{lem}

\begin{proof}
Recalling from \eqref{atc} that
\[
|c(v_1,v_3,z,\varphi)|
=|v_1-v_3|\sqrt{\frac{1-\cos G(z/|v_1-v_3|^\gamma)}{2}}
\le |v_1-v_3|\sin G(z/|v_1-v_3|^\gamma),
\]
and  from\eqref{z} that  $v_1'=v_1+c(v_1,v_3,z,\varphi)$ with 
\[
v_1'=\frac{v_1+v_3}{2}+\frac{|v_1-v_3|}{2}\sigma,
\qquad
v_3'=\frac{v_1+v_3}{2}-\frac{|v_1-v_3|}{2}\sigma,
\qquad \sigma\in \Sp^2,
\]
we obtain
\begin{align*}
&\int_0^\infty\!\!\int_0^{2\pi}
|c(v_1,v_3,z,\varphi)|\,(|v_1-v_2+c(v_1,v_3,z,\varphi)|^{-1-\alpha}+|v_1-v_2|^{-1-\alpha})\,dz\,d\varphi \\
&\le |v_1-v_3|^{1+\gamma}\int_0^\infty\!\!\int_0^{2\pi}
\sin G(z)\,(|v_1'-v_2|^{-1-\alpha}+|v_1-v_2|^{-1-\alpha})\,dz\,d\varphi.
\end{align*}
Making the change of variables $z=H(\theta)$,  equivalently, $\theta=G(z)$, transforms the above last expression into
\begin{align*}
& |v_1-v_3|^{1+\gamma}\int_0^{\pi/2}\!\!\int_0^{2\pi}
\beta(\theta)\,(|v_1'-v_2|^{-1-\alpha}+|v_1-v_2|^{-1-\alpha})\,\sin\theta\,d\theta\,d\varphi \\
&= |v_1-v_3|^{1+\gamma}\int_{\Sp^2}\beta(\theta)\,(|v_1'-v_2|^{-1-\alpha}+|v_1-v_2|^{-1-\alpha})\,d\sigma.
\end{align*}
Since  $0<(1+\gamma)+(1+\alpha)\le 2$, Young's inequality implies that
\[|v_1-v_3|^{1+\gamma} (|v_1'-v_2|^{-1-\alpha}+|v_1-v_2|^{-1-\alpha})\le C(|v_1'-v_2|^{-2}+|v_1-v_2|^{-2}+|v_1-v_3|^{2}+1).\]
Therefore, 
\begin{align}
&\E\Bigg[
\int_0^\infty\!\!\int_0^{2\pi}
|c(V^{1}_s,V^{3}_s,z,\varphi)|
\,\big|V^{1}_s-V^{2}_s+c(V^{1}_s,V^{3}_s,z,\varphi)\big|^{-1-\alpha}\,dz\,d\varphi\Bigg]\nonumber\\
&\le C\int_{\R^9}\int_{\Sp^2}
\beta(\theta)\,(|v_1'-v_2|^{-2}+|v_1-v_3|^2+|v_1-v_2|^{-2}+1)\,F^{N:3}_{s}(v_1,v_2,v_3)\,d\sigma \,dv_1\,dv_2\,dv_3.\nonumber
\end{align}
By the conservation of energy and Lemma \ref{cor: -2}, we have
\begin{align}\label{ineq: con r9}
    \int_{\R^9}(|v_1-v_3|^2+|v_1-v_2|^{-2}+1) F^{N:3}_{s}(v_1,v_2,v_3)\,dv_1\,dv_2\,dv_3\le 4 m_2(f_0)+2I_1(f_0)+1. 
\end{align}
It therefore remains to estimate the term involving $|v_1'-v_2|^{-2}$. To this end, we perform the change of variables $w_2:=v_2-v_1'$, which yields 
\begin{align*}
&\int_{\R^9}\int_{\Sp^2}
\beta(\theta)\,|v_1'-v_2|^{-2}\,F^{N:3}_{s}(v_1,v_2,v_3)\,d\sigma\; dv_1\,dv_2\,dv_3\\
=&
\int_{\R^9}\int_{\Sp^2}
\beta(\theta)\,|w_2|^{-2}\,F^{N:3}_{s}(v_1,w_2+v_1',v_3)\,d\sigma\; dv_1\,dw_2\,dv_3.
\end{align*}
Define \[g(w_2) := \int_{\R^6}\int_{\Sp^2} \beta(\theta)\, F_s^{N:3}(v_1,w_2+v_1',v_3) \,d\sigma\,dv_1\,dv_3.\]
Then
\[\int_{\R^9}\int_{\Sp^2}
\beta(\theta)\,|w_2|^{-2}\,F^{N:3}_{s}(v_1,w_2+v_1',v_3)\,d\sigma\; dv_1\,dw_2\,dv_3=\int_{\R^3}|w_2|^{-2}g(w_2) dw_2.\]
Applying Hardy's inequality  in the $w_2$-variable , it yields
\begin{align*}
\int_{\R^3}|w_2|^{-2}g(w_2) dw_2 \le 4 \int_{\R^3}\Big|\nabla_{w_2}\sqrt{g(w_2)}\Big|^2 dw_2.
\end{align*}
Using the identity $\nabla \sqrt{g}=\frac{\nabla g}{2\sqrt{g}}$,
together with the Cauchy-Schwarz inequality, we obtain
\begin{align*}
4\int_{\R^3}|\nabla_{w_2}\sqrt{g(w_2)}|^2\,dw_2
&=
\int_{\R^3}
\frac{|\nabla_{w_2}g(w_2)|^2}{g(w_2)}
\,dw_2\\
&=
\int_{\R^3}
\frac{\Big(
\int_{\R^6}\int_{\Sp^2}
\beta(\theta)
\,\big|
\nabla_{w_2}
F_s^{N:3}(v_1,w_2+v_1',v_3)
\big|
\,d\sigma\,dv_1\,dv_3
\Big)^2}
{g(w_2)}
\,dw_2\\
&\le
C
\int_{\R^3}
\int_{\R^6}\int_{\Sp^2}
\beta(\theta)
\frac{
\big|
\nabla_{w_2}
F_s^{N:3}(v_1,w_2+v_1',v_3)
\big|^2
}
{
F_s^{N:3}(v_1,w_2+v_1',v_3)
}
\,d\sigma\,dv_1\,dv_3\,dw_2.
\end{align*}
Returning to the original variable $v_2=w_2+v_1'$,  and recalling that $\int_{\Sp^2}\beta(\theta)\,d\sigma<\infty$ for $0<\nu<1$, we deduce  that the last term is bounded by 
\begin{align*}
C\int_{\R^9}\frac{|\nabla_{v_2}F^{N:3}_{s}(v_1,v_2,v_3)|^2}{F^{N:3}_{s}(v_1,v_2,v_3)}\,dv_1\,dv_2\,dv_3
\le  C\,I_3(F^{N:3}_{s}).
\end{align*}
 Moreover, by \eqref{I-marg}, \eqref{Fis} and \eqref{tensor}, we have
$$I_3(F^{N:3}_s)\le \frac{3}{N}I_N(F^N_{s})\le \frac{3}{N}I_N(F^N_0)=3I_1(f_0).$$ 
Therefore,
\[
\int_{\R^9}\int_{\Sp^2}
\beta(\theta)\,|v_1'-v_2|^{-2}\,F^{N:3}_{s}(v_1,v_2,v_3)\,d\sigma\, dv_1\,dv_2\,dv_3
\le CI_1(f_0),
\]
which together with \eqref{ineq: con r9} yields the result.

\end{proof}

\begin{lem}\label{lem: -alpha}
Assume \eqref{con} and \eqref{pa-range}.  For \,$t\in[0,T]$, let $F_t^N$ be the law of a solution $\bV^N_t$ to \eqref{mainsde} starting from an i.i.d. family of  $f_0$-distributed random variables $(V_0^1,\dots,V_0^N)$. Assume moreover that the initial distribution $f_0\in\ca{P}_2(\R^3)$ has finite Fisher information, i.e. $I_1(f_0)<\infty$.  Then for any $\alpha\in(0, -\gamma]$, we have
    \begin{align*}
        \E\Big[\sup_{0\le t \le T}\abs{V_t^{i}-V_t^{j}}^{-\alpha}\Big]\le C(1+T) (I_1(f_0)+m_2(f_0)+1),
    \end{align*}
    where $C$ depends on $(\gamma,\nu,\alpha)$.
\end{lem}
\begin{proof}
    By exchangeability, it suffices to consider  $(i,j)=(1,2)$. Recall \eqref{atc} and set $c^{i,j}_t := c(V_t^i, V_t^j,z,\varphi)$. We first observe that
    $$|V^{1}_t-V^{2}_t+2c_t^{1,2}|^2=|V^{1}_t-V^{2}_t|^2+4c_t^{1,2}\cdot(V^{1}_t-V^{2}_t)+4|c_t^{1,2}|^2= |V^{1}_t-V^{2}_t|^2,$$  which shows that the jump terms driven by the Poisson measure $O_{12}$ do not contribute to the evolution of $|V_t^1-V_t^2|$. Consequently, only collisions involving particle $1$ or particle $2$ with a third particle contribute, and we obtain
\begin{align*}
    \abs{V_t^{1}-V_t^{2}}^{-\alpha}=&\abs{V_0^{1}-V_0^{2}}^{-\alpha}\\
    &+ \sum_{j\ge 3}^N  \intot\int_0^\infty\int_0^{2\pi} 
 \Big(|V^{1}_\sm-V^{2}_\sm+c_s^{1,j}|^{-\alpha}
- |V^{1}_\sm-V^{2}_\sm|^{-\alpha}\Big) O_{1j}(ds,dz,d\varphi)
\\
&+ \sum_{i\ge 3}^N  \intot\int_0^\infty\int_0^{2\pi} 
\Big(|V^{1}_\sm-V^{2}_\sm-c_s^{2,i}|^{-\alpha} 
- |V^{1}_\sm-V^{2}_\sm|^{-\alpha}\Big) O_{2i}(ds,dz,d\varphi).
\end{align*}
For $x,y>0$,  the mean value theorem applied to the function $r\mapsto r^{-\alpha}$ yields
\[
x^{-\alpha}-y^{-\alpha}
\le
\alpha\bigl(x^{-1-\alpha}+y^{-1-\alpha}\bigr)|x-y|.
\]
Using this estimate and the symmetry between the two sums, we deduce that
\begin{align*}
&\E\Big[\sup_{0\le t \le T}\abs{V_t^{1}-V_t^{2}}^{-\alpha}\Big]\\
&\le \E[\abs{V_0^{1}-V_0^{2}}^{-\alpha}]+ C\int_0^T\int_0^\infty\int_0^{2\pi}
\E\Big[\abs{c^{1,3}_s}
\Big(|V^{1}_s -V^{2}_s+c^{1,3}_s|^{-1-\alpha} 
+|V^{1}_s-V^{2}_s|^{-1-\alpha}\Big)\Big] dsdzd\varphi.
\end{align*}
Moreover, by Lemma \ref{cor: -2},  we have 
\begin{align*}
    \E[\abs{V_0^{1}-V_0^{2}}^{-\alpha}]\le C(1+\E[\abs{V_0^{1}-V_0^{2}}^{-2}])\le C(1+I_1(f_0)).
\end{align*}
Finally, applying Lemma \ref{lem: v-vf}, we conclude the desired estimate.
\end{proof}

\subsection{Proof of Theorem \ref{kpst}}
Let $\hat{\bV}=(\hat V^i_t)_{i=1,\dots,N,\ t\ge 0}$  be another weak solution to \eqref{mainsde} with the same initial condition as $\bV$.  By enlarging the probability space if necessary,  we may assume that $\hat{\bV}$ is driven by the same family of Poisson random measures $(O_{ij})_{1\le i,j\le N}$, namely, 
\begin{align}\label{vhat}
\hat V^{i}_t
&=V^i_0+\sum_{j=1}^N\int_0^t\int_0^\infty\int_0^{2\pi}
c(\hat V^{i}_{s-},\hat V^{j}_{s-},z,\varphi)\, O_{ij}(ds,dz,d\varphi),
\qquad i=1,\dots,N.
\end{align}
Also, we may rewrite the equation satisfied by  $\bV$  in law as 
\begin{align}\label{vhat'}
 V^{i}_t
&=V^i_0+\sum_{j=1}^N\int_0^t\int_0^\infty\int_0^{2\pi}
c( V^{i}_{s-}, V^{j}_{s-},z,\varphi+\varphi_{s,i,j})\,
O_{ij}(ds,dz,d\varphi),
\qquad i=1,\dots,N,
\end{align}
where $\varphi_{s,i,j}=\varphi_0(\hat V^{i}_{s-}-\hat V^{j}_{s-}, V^{i}_{s-}-V^{j}_{s-})$. 
Fix $A\ge 1$  and define  the stopping times
\[
\tau_{N,A}:=\inf\Bigl\{t\ge 0:\ \exists\, i\neq j\ \text{such that}\ |V^i_t-V^j_t|^{-1}\ge A\Bigr\},
\]
and
\[
\hat\tau_{N,A}:=\inf\Bigl\{t\ge 0:\ \exists\,i\neq j \ \text{such that}\ |\hat V^i_t-\hat V^j_t|^{-1}\ge A\Bigr\}.
\]
By It\^o's formula, for $t\le \tau_{N,A}\land \hat\tau_{N,A}$,
\begin{align*}
  |\hat V^1_t-V^1_t|^2 = \sum_{j=2}^N\int_0^t \int_0^\infty \int_0^{2\pi}\Big(|\hat V^1_{s-}-V^1_{s-}+c( \hat V^{1}_{s-},  \hat V^{j}_{s-},z,\varphi)-c(  V^{1}_{s-},  V^{j}_{s-},z,\varphi+\varphi_{s,i,j})|^2\\  -|\hat V^1_{s-}-V^1_{s-}|^2 \Big)\, O_{1j}(ds,dz,d\varphi).
\end{align*}
It follows from  remark~\ref{inequ:c} (with $v=\hat V^1_{s-}$, $v_*= \hat V^{j}_{s-}$, $\tv=V^{1}_{s-}$, $\tv_*=V^{j}_{s-}$) that 
\begin{align*}
& \int_0^\infty \int_0^{2\pi}
  \Bigl(2(\hat V^1_{s-}-V^1_{s-})\cdot\bigl[c( \hat V^{1}_{s-},  \hat V^{j}_{s-},z,\varphi)-c(  V^{1}_{s-},  V^{j}_{s-},z,\varphi+\varphi_{s,i,j})\bigr]\Bigr)\,dz\,d\varphi\\
 & \le
2|\hat V^1_{s-}-V^1_{s-}|\Big|\int_0^\infty \int_0^{2\pi}\bigl[c( \hat V^{1}_{s-},  \hat V^{j}_{s-},z,\varphi)-c(  V^{1}_{s-},  V^{j}_{s-},z,\varphi+\varphi_{s,i,j})\bigr]\,dz\,d\varphi\Big|\\
 &\le C(|\hat V^1_{s-}-V^1_{s-}|^2+|\hat V^{j}_{s-}- V^{j}_{s-}|^2)\min{\{|\hat V^1_{s-}-\hat V^{j}_{s-}|^\gamma, |V^{1}_{s-}-V^{j}_{s-}|^\gamma\}},
\end{align*}
and that 
\begin{align*}
&\int_0^\infty \int_0^{2\pi}
   \bigl|c( \hat V^{1}_{s-},  \hat V^{j}_{s-},z,\varphi)-c(  V^{1}_{s-},  V^{j}_{s-},z,\varphi+\varphi_{s,i,j})\bigr|^2 \,dz\,d\varphi \\
   &\le C(|\hat V^1_{s-}-V^1_{s-}|^2+|\hat V^{j}_{s-}- V^{j}_{s-}|^2)\min{\{|\hat V^1_{s-}-\hat V^{j}_{s-}|^\gamma, |V^{1}_{s-}-V^{j}_{s-}|^\gamma\}}.
\end{align*}
Since $s\le \tau_{N,A}\land \hat\tau_{N,A}$, namely,  $| V^{1}_{s-}- V^{j}_{s-}|^{-1}\le A$,   we have
\begin{align*}
 &\int_0^\infty \int_0^{2\pi}
  \Bigl(2(\hat V^1_{s-}-V^1_{s-})\cdot\bigl[c( \hat V^{1}_{s-},  \hat V^{j}_{s-},z,\varphi)-c(  V^{1}_{s-},  V^{j}_{s-},z,\varphi+\varphi_{s,i,j})\bigr]\\
  &\qquad\qquad\quad +\ \bigl|c( \hat V^{1}_{s-},  \hat V^{j}_{s-},z,\varphi)-c(  V^{1}_{s-},  V^{j}_{s-},z,\varphi+\varphi_{s,i,j})\bigr|^2\Bigr)\,dz\,d\varphi\\
  &\le C_A \bigl(|\hat V^{1}_{s-}- V^{1}_{s-}|^2+|\hat V^{j}_{s-}- V^{j}_{s-}|^2\bigr).
\end{align*}
 Using exchangeability, we deduce that 
\begin{align*}
\E\bigl[|\hat V^1_{t\wedge \tau_{N,A}\wedge \hat\tau_{N,A}}- V^1_{t\wedge \tau_{N,A}\wedge \hat\tau_{N,A}}|^2\bigr]
\le C_A\int_0^t\E\bigl[|\hat V^{1}_{s\wedge \tau_{N,A}\wedge \hat\tau_{N,A}}- V^{1}_{s\wedge \tau_{N,A}\wedge \hat\tau_{N,A}}|^2\bigr] \,ds.
\end{align*}
Applying Gr\"onwall's lemma yields, for every $A>0$,
\begin{align*}
    \E\bigl[|\hat V^1_{t\wedge \tau_{N,A}\wedge \hat\tau_{N,A}}- V^1_{t\wedge \tau_{N,A}\wedge \hat\tau_{N,A}}|^2\bigr]=0,
\end{align*}
and consequently
\[\hat V^1_{t\wedge \tau_{N,A}\wedge \hat\tau_{N,A}} =V^1_{t\wedge \tau_{N,A}\wedge \hat\tau_{N,A}}, \quad\text{a.s.}.\]
By exchangeability, the same conclusion holds for every particle, so that 
\[
(\bV_{t\wedge \tau_{N,A}\wedge \hat\tau_{N,A}})_{0\le t\le T}=\ (\hat{\bV}_{t\wedge \hat\tau_{N,A}\wedge \tau_{N,A}})_{0\le t\le T}, \quad\text{a.s.}.
\]
Furthermore,  for every $T>0$, we have 
\[
\lim_{A\to\infty}\bbP(\tau_{N,A}\le T)=0.
\]
Indeed, by  exchangeability and the union bound,  for $-1<\gamma<0$ and $\nu\in(0,1)$,
\begin{align*}
\bbP(\tau_{N,A}\le T)
\le N^2\,\bbP\Bigl(\sup_{0\le t\le T}|V^1_t-V^2_t|^{-1}\ge A\Bigr)
\le N^2\frac{\E[\sup_{0\le t\le T}|V^1_t-V^2_t|^{-\alpha}]}{A^\alpha}, 
\end{align*} 
where $0<\alpha\le-\gamma$. By  Lemma~\ref{lem: -alpha}, the right hand side converges to zero as $A\to\infty$. 
The same argument applies to $\hat\tau_{N,A}$, yielding
\[ \tau_{N,A}\wedge\hat\tau_{N,A} \longrightarrow\infty \qquad\text{in probability},\] as $A\to\infty$,
which immediately implies, by letting $A\to\infty$, that 
\[
(\bV_t)_{0\le t\le T}=\ (\hat{\bV}_t)_{0\le t\le T}, \quad\text{a.s.},
\]
which completes the proof of pathwise uniqueness.

Having established well-posedness of the Kac particle system, we next derive
uniform moment estimates. Together with the conservation of energy and the Fisher information bounds established above, these estimates will be needed later in the proof
of quantitative propagation of chaos. The following lemma shows that any finite moment of order $q>4$ is propagated uniformly on compact time intervals.

\begin{lem}[Propagation of moments]\label{Part-mom}
Under the same assumptions as in Theorem \ref{kpst}, let
$(\bV^N_t)_{t\ge 0}$ be the solution to \eqref{mainsde} starting from an i.i.d. family of  $f_0$-distributed random variables $(V_0^1,\dots,V_0^N)$. Assume further that
$m_q(f_0)<+\infty$ for some $q>2$. Then, for every $1\le i\le N$,
\begin{align*}
    \sup_{0\le t\le T}\E[|V^i_t|^q]\le C_{T,f_0},
\end{align*}
where $C_{T,f_0}>0$.
\end{lem}

\begin{proof}
By exchangeability, it suffices to consider the case $i=1$. By It\^o's formula,
\begin{align*}
|V^1_t|^q
=
|V^1_0|^q
+
\sum_{j=1}^N
\intot\int_0^\infty\int_0^{2\pi}
\Big(
|V^1_\sm+c(V^{1}_\sm,V^{j}_\sm,z,\varphi)|^q
-|V^1_\sm|^q
\Big)
O_{1j}(ds,dz,d\varphi).
\end{align*}
For simplicity, write  $c_s^{1,j}:=c(V^{1}_s,V^{j}_s,z,\varphi)$.
Taking expectations and using exchangeability, we obtain
\begin{align*}
\E[|V^1_t|^q]
\le&\ \E[|V^1_0|^q]
+
\intot\int_0^\infty\int_0^{2\pi}
\E\Big[
|V^1_s+c_s^{1,2}|^q-|V^1_s|^q
\Big]\,ds\,dz\,d\varphi\\
=&\ \E[|V^1_0|^q]
+
\intot\int_0^\infty\int_0^{2\pi}
\E\Big[
|V^1_s+c_s^{1,2}|^q-|V^1_s|^q
-qc_s^{1,2}\cdot V^1_s|V^1_s|^{q-2}
\Big]\,ds\,dz\,d\varphi\\
&\hskip1.35cm
+
q\intot\int_0^\infty\int_0^{2\pi}
\E\Big[
c_s^{1,2}\cdot V^1_s|V^1_s|^{q-2}
\Big]\,ds\,dz\,d\varphi.
\end{align*}
Since $|c_s^{1,2}|\le |V^1_s-V^2_s|$, Taylor's formula gives
\begin{align*}
\left|
|V^1_s+c_s^{1,2}|^q-|V^1_s|^q
-qc_s^{1,2}\cdot V^1_s|V^1_s|^{q-2}
\right|
&\le
C_q\big(|V^1_s|^{q-2}+|c_s^{1,2}|^{q-2}\big)|c_s^{1,2}|^2\\
&\le
C_q\big(|V^1_s|^{q-2}+|V^2_s|^{q-2}\big)|c_s^{1,2}|^2.
\end{align*}
Using Lemma \ref{lem:c}-(i),  we deduce
\begin{align*}
\E[|V^1_t|^q]
\le&\ \E[|V^1_0|^q]
+
C_q\intot
\E\Big[
\big(|V^1_s|^{q-2}+|V^2_s|^{q-2}\big)
|V^1_s-V^2_s|^{2+\gamma}
\Big]\,ds\\
&\hskip3cm
+
C_q\intot
\E\Big[
|V^1_s|^{q-1}|V^1_s-V^2_s|^{1+\gamma}
\Big]\,ds.
\end{align*}
Since $-1<\gamma<0$, we have $q+\gamma<q$. It then follows from Young's inequality that
\begin{align*}
\E[|V^1_t|^q]\le\ \E[|V^1_0|^q]+C_q\intot
\E\Big[
1+|V^1_s|^q+|V^2_s|^q
\Big]\,ds
\le\ \E[|V^1_0|^q]
+
C_q\intot
\E\Big[
1+|V^1_s|^q
\Big]\,ds,
\end{align*}
where we used the exchangeability in the last inequality. Gr\"onwall's inequality then yields
\begin{align*}
    \sup_{0\le t\le T}\E[|V^1_t|^q]
    \le
    C_q\big(\E[|V^1_0|^q]+T\big)e^{C_qT}.
\end{align*}
Since $\E[|V^1_0|^q]=m_q(f_0)<+\infty$, the proof is complete.
\end{proof}

\section{Boltzmann processes with cutoff}
\setcounter{equation}{0} 
\label{sec:ConvCutoff}

Recall the definition of the cutoff kernel $B_K$  in \eqref{Def: kernel cut}. Let $f_0\in \mathcal P_2(\mathbb R^3)$  have finite Fisher
information and finite $q$-th moments for some $q>12$. Then, under \eqref{pa-range},  the cutoff Boltzmann equation
\eqref{Bol}, admits a unique solution $(f_t^K)_{t\geq0}\in C\bigl([0,\infty), L^3(\mathbb R^3)\cap \mathcal P_2(\mathbb R^3)\bigr)$, satisfy \eqref{eq: inv3}. Moreover, by \cite[Theorem~1.5]{IMBERT24}, the Fisher information of $f_t^K$ is non-increasing in time, since the condition $\Lambda_b\ge 1$ is satisfied for any angular kernel $b$. We record below several estimates on $f_t^K$ that will be used later.

\subsection{Some properties of $f_t^K$}

\begin{lem}\label{lem: prop fk}
    Let $f^K_t$ be the cutoff solution defined above and assume that  $f_0$ has finite Fisher information. Then the following estimates hold uniformly in $t\ge 0$ and $K\ge 1$:
\vip
    (i) For any $0\le \alpha\le 2$, $\sup_{w\in\R^3} \int_{\R^3}|v-w|^{-\alpha}f_t^K(v) dv\le 1+I_1(f_0)$.
    \vip
    (ii) For any $1\le p\le 3$, $\int_{\R^3} |f_t^K|^p dv\le C(1+I_1(f_0))^3$.
\vip
    (iii) There exists a constant $M_2$, independent of $K$ and $t$, such that $\int_{\R^3} |f_t^K|^2 dv\ge M_2$.
\end{lem}
\begin{proof}
   We first prove the singular moment estimate. Since $0\le \alpha\le2,$ we have 
    \begin{align*}
        \int_{\R^3}|v-w|^{-\alpha}f_t^K(v) dv\le& 1+\int_{\R^3}|v-w|^{-2}f_t^K(v) dv\le 1+\int_{\R^3}|v'|^{-2}f_t^K(v'+w) dv'.
\end{align*}
By Hardy's inequality and the decay of Fisher information, we obtain
\[\int_{\R^3}|v'|^{-2}f_t^K(v'+w) dv'\le 4\int_{\R^3}\Big|\nabla_{v'}\sqrt{f_t^K(v'+w)} \Big|^2 dv' \le I_1(f_0).\]
Therefore, $\sup_{w\in\R^3} \int_{\R^3}|v-w|^{-\alpha}f_t^K(v) dv\le 1+I_1(f_0)$ holds.
The proof of the second inequality is similar. We have 
    \begin{align*}
        \int_{\R^3} |f_t^K|^p dv\le 1+ \int_{\R^3} |f_t^K|^3 dv= 1+ \int_{\R^3} |\sqrt{f_t^K}|^6 dv.
    \end{align*}
By the Sobolev inequality and the decay of Fisher information, we have 
\[\int_{\R^3} |\sqrt{f_t^K}|^6 dv \le C(I_1(f_t^K))^3\le C(I_1(f_0))^3,\]
which implies (ii).
    It remains to prove the lower $L^2$ bound. By conservation of energy, choose  $R= \sqrt{3m_2(f_t^K)}= \sqrt{3m_2(f_0)}=3.$ Hence
    \begin{align*}
        \int_{\R^3} |f_t^K|^2 dv&\ge \int_{\R^3} |f_t^K|^2\indiq_{\{|v|\le R\}} dv \ge |\{|v|\le R\}|^{-1}\Big(\int_{\R^3} f_t^K \indiq_{\{|v|\le R\}} dv\Big)^2\\
        &\ge |\{|v|\le R\}|^{-1} \Big(1-\int_{\R^3} f_t^K \indiq_{\{|v|> R\}} dv\Big)^2\ge 4(9|\{|v|\le 3\}|)^{-1}=1/(81\pi)=:M_2,
    \end{align*}
    where $|\{|v|\le R\}|$ denotes the volume of the ball $\{v\in\R^3:|v|\le R\}$.
\end{proof}

\begin{lem}[Propagation of polynomial moments]\label{lem: prop mom}
    Let $(f^K_t)_{0\le t\le T}$ be the cutoff solution starting from $f_0\in\ca{P}_2(\R^3)$ defined above. Then, for $q>2$,
    \begin{align*}
        \sup_{0\le t\le T}m_q(f^K_t)\le C_{T,q,f_0},
    \end{align*}
   where the constant $C_{T,q,f_0}=C_q(1+m_q(f_0))\exp(C_qT)$.
\end{lem}
\begin{proof}
Recalling \eqref{CK}, we set $c_K:=c_K(v,v_\ast,z,\varphi)$. If $(f_t^K)_{t\ge0}$ solves \eqref{Bol} associated with the cutoff kernel $B_K$ introduced in \eqref{Def: kernel cut}, then it satisfies the weak formulation
\begin{align}\label{wBol}
\int_{\R^3}\phi(v)f^K_t(dv)
=\int_{\R^3}\phi(v)f_0(dv) 
+
\int_0^t\int_{\R^3}\int_{\R^3}
\ca{A}\phi(v,v_\ast)\,
f^K_s(dv_\ast)f^K_s(dv)\,ds,
\end{align}
for every $\phi\in C_b^2(\R^3)$, where
\[
\ca{A}\phi(v,v_\ast)
:=
\int_0^\infty\int_0^{2\pi}
\bigl[\phi(v+c_K)-\phi(v)\bigr]
\,d\varphi\,dz.
\]
For $q>2$, taking $\phi(v)=|v|^q$ in \eqref{wBol}, (or, more rigorously, replacing $\phi$ by $|v|^q\wedge L$ and letting $L\to\infty$), we obtain
    \begin{align*}
\int_{\bb{R}^3}|v|^qf^K_t(dv)=\int_{\bb{R}^3}|v|^qf_0(v)dv+
\int_0^t\int_{\bb{R}^3}\int_{\bb{R}^3}\ca{G}(v,v_\ast)f^K_s(v_\ast)f^K_s(v)dv_*dvds,
    \end{align*}
where
    \begin{align*}
\ca{G}(v,v_\ast):=\int_{0}^{\infty}\int_0^{2\pi}
  [|v + c_K|^q-|v|^q-qv\cdot c_K|v|^{q-2}]dz d\varphi
  +q|v|^{q-2}v\cdot\int_{0}^{\infty}\int_0^{2\pi}
   c_K dz d\varphi.
    \end{align*}
 Using $|c_K|\le |v-v_*|$ and the remainder estimate in Taylor's formula, we obtain 
\[
\Big||v+c_K|^q-|v|^q-q|v|^{q-2}v\cdot c_K\Big|
\le C_q\bigl(|v|^{q-2}+|c_K|^{q-2}\bigr)|c_K|^2\le C_q\bigl(|v|^{q-2}+|v_*|^{q-2}\bigr)|c_K|^2.
\]
Moreover, Lemma~\ref{lem:c}-(ii) yields
\[
\Big|
\int_0^\infty\!\!\int_0^{2\pi}
c_K\,d\varphi\,dz
\Big|
\le C|v-v_*|^{1+\gamma},
\qquad
\int_0^\infty\!\!\int_0^{2\pi}
|c_K|^2\,d\varphi\,dz
\le C|v-v_*|^{2+\gamma},
\]
which implies  that
\begin{align*}
|\mathcal G(v,v_*)|
&\le C_q\Big[
\bigl(|v|^{q-2}+|v_*|^{q-2}\bigr)|v-v_*|^{2+\gamma}
+|v|^{q-1}|v-v_*|^{1+\gamma}
\Big].
\end{align*}
Since $-1<\gamma<0$, we have $q+\gamma<q$. It then follows from Young's inequality that
\[
|\mathcal G(v,v_*)|
\le C_q\bigl(1+|v|^q+|v_*|^q\bigr).
\]
Therefore, substituting this estimate into the weak formulation yields
\begin{align*}
\int_{\R^3}|v|^q f^K_t(dv)
&\le \int_{\R^3}|v|^q f_0(v)\,dv
+ C_q \Big(1+\int_0^t \int_{\R^3} |v|^q f^K_s(v)\,dv\,ds\Big).
\end{align*}
Applying Gr\"onwall's lemma, we conclude that
\begin{align*}
     \sup_{0\le t\le T} m_q(f^K_t)\le C_q(1+m_q(f_0))\exp(C_qT),
\end{align*}
  which proves the desired moment estimate.
\end{proof}

\begin{lem}\label{lem: ftKft}
Assume \eqref{con} and \eqref{pa-range}. Let  $(f_t)_{t\ge0}$ be the
solution to the non-cutoff Boltzmann equation   \eqref{Bol} and  let $(f^K_t)_{t\ge0}$ be the cutoff solution introduced above. Assume that $f_0\in\ca{P}_2(\R^3)$ has finite Fisher information, i.e. $I_1(f_0)<\infty$. Then, for every $0\le t\le T$ and $K\ge 1$, there exists some constant $C_{T,f_0}$ depends on $T$ exponentially and $I_1(f_0)$, such that
    \beq\label{cutf}
\cW_2^2(f_t^K,f_t)\le C_{T,f_0}K^{1-1/\nu}.
\eeq
\end{lem}

\subsection{The Boltzmann process with cutoff}
    
\begin{prop}\label{Bp}
Assume \eqref{con} \eqref{pa-range}.  Let $f_0\in\cP_2(\rd)$ have finite Fisher information. Fix $K\ge 1$,  let
$M(ds,d\alpha,dz,d\varphi)$ be a Poisson measure on
$[0,\infty)\times[0,1]\times[0,\infty)\times[0,2\pi)$ with  intensity $ds d\alpha dzd\varphi$. Consider $W_0\sim f_0$ (independent of $M$) and let $W_t^{*K}(\alpha)$ be a c\`adl\`ag $\alpha$-stochastic process with distribution $f_t^K$ at  $t\ge 0$. For $t\ge0$, consider the cutoff nonlinear stochastic equation 
\beq\label{cbp}
W_t^K=W_0+\intot\int_0^1\int_0^\infty\int_0^{2\pi} c_K(W_\sm^K, W_\sm^{*K}(\alpha),z,\varphi)M(ds,d\alpha,dz,d\varphi).
\eeq
Then \eqref{cbp} admits a unique strong solution, and this solution is pathwise unique.  Moreover, $W_t^K$ is $f_t^K$-distributed for each $t\ge0$.
\end{prop}
\begin{proof}
The pathwise existence and uniqueness of \eqref{cbp} follows immediately  from the cutoff assumption. Indeed, the
total collision rate is finite since $\int_{\mathbb S^2}B_K(r,\theta)\,d\sigma
=
K(r^\gamma\wedge K)
\leq K^2.$
Consequently, the associated Poisson random measure generates only finitely many jumps on every bounded time interval. The process can therefore be constructed recursively between successive jump times. It remains to identify the time marginals. Let $\phi$ be a smooth test function. Applying It\^o's formula for jump processes to $\phi(W_t^K)$ and taking expectations, we obtain the weak formulation satisfied by the law of $W_t^K$. Using that  $W_t^{*K}(\alpha)$ has law $f_t^K$, it follows that the law of $W_t^K$ solves the cutoff Boltzmann equation with initial datum $f_0$. By the uniqueness of weak solutions to the cutoff Boltzmann equation, we conclude that $\ca{L}(W_t^K)=f_t^K,  t\ge0$. 

\end{proof}

\subsection{Proof of Lemma \ref{lem: ftKft}}
We consider the  non-cutoff Boltzmann process 
 \beq\label{cbpw}
W_t=W_0+\intot\int_0^1\int_0^\infty\int_0^{2\pi} c(W_\sm, W_\sm^{*}(\alpha),z,\varphi)M(ds,d\alpha,dz,d\varphi),
\eeq
where $W_t^{*}(\alpha)$  is an $f_t$-distributed  $\alpha$ random variable.  As proved in \cite[Proof of Lemma 4.6]{MR2398952}, \eqref{cbpw} admits a unique solution in law and satisfies
$\mathcal L(W_t)=f_t$ for all $t\ge0$. Let $(W_t^{*}(\alpha),W_t^{*K}(\alpha))$  be an optimal coupling between $f_t$ and $f_t^K$, namely,
\begin{align*}
    \int_0^1 |W_t^{*}(\alpha)-W_t^{*K}(\alpha)|^2d\alpha =\cW_2^2 (f_t,f_t^K).
\end{align*}
We consider also the cutoff Boltzmann process
\[W_t^K=W_0+\intot\int_0^1\int_0^\infty\int_0^{2\pi}c_K(W_\sm^K, W_\sm^{*K}(\alpha),z,\varphi+\varphi_{s,\alpha,K})M(ds,d\alpha,dz,d\varphi),
 \]
where $\varphi_{s,\alpha,K}:=\varphi_0(W_s-W_s^{*}(\alpha), W_s^K-W_s^{*K}(\alpha))$.  By Proposition~\ref{Bp}, this equation admits a unique strong solution
satisfying $\mathcal L(W_t^K)=f_t^K$ for all $t\ge0$.
 Fix $T>0$ and restrict attention to the time interval $[0,T]$.
Define 
$$\tilde\Delta_{c,c_K}:=c(W_s, W_s^{*}(\alpha),z,\varphi)-c_K(W^K_s, W_s^{*K}(\alpha),z,\varphi+\varphi_{s,\alpha,K}).$$
 Note that $\tilde\Delta_{c,c_K}$ is precisely the quantity appearing in Lemma \ref{lem: 3ccut} with $v=W_s$, $v_*=W_s^{*}(\alpha)$, $\tv=W^K_s$ and $\tv_*=W_s^{*K}(\alpha)$.
Then, applying It\^o's formula  and taking expectations, we obtain
\begin{align*}
 \E[|W_t-W^K_t|^2] 
=   \intot\int_0^1\int_0^\infty\int_0^{2\pi} \E[2\tilde\Delta_{c,c_K}\cdot(W_{s}-W^K_{s})+|\tilde\Delta_{c,c_K}|^2] ds d\alpha dz d\varphi.
\end{align*}
Using estimates \eqref{ineq: KKp22cut11} and \eqref{ineq: KKp22cut2}  from  Lemma \ref{lem: 3ccut}, the right hand side is bounded by 
\begin{align}\label{right-b}
& 2\intot\int_0^1\E\Big[|W_{s}-W^K_{s}|\Big|\int_0^\infty\int_0^{2\pi} \tilde\Delta_{c,c_K}dz d\varphi\Big|\Big]dsd\alpha+\intot\int_0^1\int_0^\infty\int_0^{2\pi} \E[|\tilde\Delta_{c,c_K}|^2] ds d\alpha dz d\varphi
\notag\\
\le & C\intot\int_0^1 \E\Big[\Big(|W_{s}-W^K_{s}|^2+|W^*_{s}(\alpha)-W^{*K}_{s}(\alpha)|^2\Big)|W^K_{s}-W^{*K}_{s}(\alpha)|^\gamma\Big]ds d\alpha\nonumber\\
&+CK^{1-1/\nu}\int_0^t\int_0^1\E\Big[1+|W_{s}|^2+|W^K_{s}|^2+|W^*_{s}(\alpha)|^2+|W^{*K}_{s}(\alpha)|^2\Big] d\alpha ds.
\end{align}
We next estimate the terms appearing on the right hand side of \eqref{right-b}. First,  recalling \eqref{eq: inv3}, we have
\begin{align*}
\int_0^1\E\Big[1+|W_s|^2&+|W^K_s|^2+|W^*_s(\alpha)|^2+|W^{*K}_s(\alpha)|^2\Big] d\alpha= 1+2\int_{\R^3}|v|^2 (f_s+f_s^K)(dv)\le C.
\end{align*}
Next, by Lemma~\ref{lem: prop fk}-(i),  we have 
\begin{align*}
    \E[|W^K_{s}-W^{*K}_{s}(\alpha)|^\gamma]=\int_{\R^3} |v-W^{*K}_{s}(\alpha)|^\gamma f^K_s(dv)\le \sup_{w\in\R^3}\int_{\R^3} |v-w|^\gamma f^K_s(dv) \le 1+I_1(f_0) <\infty.
    \end{align*}
Since $\ca{L}(W^{*K}_{s}(\alpha))=f_s^K$, and using again
Lemma~\ref{lem: prop fk}-(i), we obtain
\[\int_0^1 |W^K_{s}-W^{*K}_{s}(\alpha)|^\gamma d\alpha=\int_{\R^3} |W^{K}_{s}-v|^\gamma f^K_s(dv)\ \le\sup_{w\in\R^3}\int_{\R^3} |v-w|^\gamma f^K_s(dv) \le 1+I_1(f_0) <\infty.
    \]
Moreover, recalling  $\int_0^1 |W_s^{*}(\alpha)-W_s^{*K}(\alpha)|^2d\alpha =\cW_2^2 (f_s,f_s^K)\le \E[|W_{s}-W^K_{s}|^2]$, we thus have 
      \begin{align*}
&\int_0^1 \E\Big[(|W_{s}-W^K_{s}|^2+|W^{*}_{s}(\alpha)-W^{*K}_{s}(\alpha)|^2)|W^K_{s}-W^{*K}_{s}(\alpha)|^\gamma\Big] d\alpha\\
    &\le C_{f_0} \E[|W_s-W^K_s|^2]+C_{f_0}\cW_2^2 (f_s,f_s^K)\le C_{f_0} \E[|W_s-W^K_s|^2].
\end{align*}
Substituting the above estimates into \eqref{right-b} and recalling that $t\le T$, we obtain
 \begin{align*}
     \E[|W_t-W^K_t|^2]\le C_{f_0}\int_0^t  \E[|W_s-W^K_s|^2]ds+CTK^{1-1/\nu}.
 \end{align*}
An application of Grönwall's lemma then yields the desired estimate.

\section{Couplings}
In this section, we construct a suitable  coupling between the Kac's particle system and the solution $f$ to \eqref{Bol}, following the approach developed in  \cite{MR3476628,MR3769742,MR3456347}.  
The main idea is to realize both the particle system and a family of Boltzmann processes on the same probability space and to drive them by a common Poisson random measure. This coupling will enable us to compare their trajectories pathwise and derive quantitative estimates for the propagation of chaos.

\subsection{Construction of optimal map}
A key ingredient is the construction of suitable $f$-distributed random variables appearing in the Boltzmann process \eqref{cbp}. These variables must be chosen in an optimal way so that they remain as close as possible to the particles involved in the corresponding collisions of Kac's particle system. To this end, we introduce a variant of the optimal coupling construction developed in \cite[Lemma~3]{MR3476628}, which will be used throughout the sequel.

\vip

\begin{lem}\label{ww2p}
Let $f_t\in C([0,\infty), \ca{P}_2(\R^3))$ admit a density for every $t>0$. Fix $N\ge 2$ and $\bw\in \R^{3N}$. For $j=1,...,N$, there exists an $\rd$-valued function $\Pi^{j}_t(\bw,\alpha),$ measurable in $(t,\bw,\alpha)\in[0,\infty]\times \R^{3N}\times[0,1]$ such that the following properties hold.

\vip

(i) For any $\bw\in \R^{3N}$ and  $t\ge 0$, $$\int_0^1 \frac{1}{N}\sum_{j=1}^N|\Pi_t^{j}(\bw,\alpha)-w_j|^2 d\alpha=\cW^2_2(f_t, \bar{\mu}_{\bw}),$$
where  $\bar{\mu}_{\bw}=\frac{1}{N}\sum_{j=1}^N\delta_{w_j}$ is the associated empirical measure.

\vip
(ii) For any $\bw\in  \R^{3N},\ t\ge 0$, and  any $f_t$-integrable function $\phi,$ we have $$\int_0^1 \frac{1}{N}\sum_{j=1}^N\phi(\Pi_t^{j}(\bw,\alpha)) d\alpha=\int_\rd\phi(u)f_t(du).$$

\vip
(iii) Let $\bY\in \R^{3N}$ be any  exchangeable random vector, for $t\ge 0, j=1,...,N$, and  any $f_t$-integrable function $\phi,$  we have $$\E\left[\int_0^1\phi(\Pi^{j}_t(\bY,\alpha))d\alpha \right]=\int_\rd\phi(u)f_t(du).$$

\end{lem}

\begin{proof}
For any fixed $j\in \{1,...,N\}$, we are going to construct a measurable map $\Pi_t^{j}:\mathbb{R}_+\times \R^{3N}\times[0,1]\rightarrow \R^3$,\,  $(t,\bw,\alpha)\mapsto\Pi_t^{j}(\bw,\alpha).$ 
\vip
Let $n\ge1$, and for $\be\in \R^{3n}$,  define the empirical measure $\mu_{\be}=\frac{1}{n}\sum_{i=1}^n\delta_{e_i}$.  Since $f_t\in C([0,\infty),\cP_2(\rd))$, a measurable selection result [see, e.g., Corollary 5.22 of Villani \cite{Vtot2003}] yields a measurable map $(t,\be)\mapsto R_{t,\mu_{\be}}$ such that $R_{t,\mu_{\be}}\in \mathcal{P}(\rd\times\rd)$  is an optimal coupling between $f_t$ and $\mu_{\be}.$  For every Borel set $B\subseteq\rd$,  we  define 
$$
F(t,\be,B):=\frac{R_{t,\mu_{\be}}\Big(B, \{e_1\}\Big)}{R_{t,\mu_{
\be}}\Big(\rd, \{e_1\}\Big)}=\frac{nR_{t,\mu_{\be}}\Big(B, \{e_1\}\Big)}{\sharp\{l: e_l=e_1\}},
$$
for  $t\ge 0,\ \be\in\R^{3n}$.  Since  $R_{t,\mu_{\be}}$ is measurable in $(t, \be)$, then $F$ is a measurable probability kernel from $\R_+\times\R^{3n}$ to $\R^3$. 
By the measurable representation theorem for probability kernels, there exists a measurable map
$g:\R_+\times\R^{3n}\times[0,1]\mapsto\R^3$ such that, $g(t,\be,\alpha)$ has law  $F(t,\be,\cdot)$  whenever $\alpha$ is uniformly distributed on $[0,1]$.
\vip
We now take $n=N$. For $\bw=(w_1,\ldots,w_N)\in\R^{3N}$, define
\begin{equation}\label{cp}
\Pi_t^j(\bw,\alpha)
:=
g(t,\bw^j,\alpha),
\end{equation}
where $\bw^j\in\R^{3N}$ denotes the vector obtained from $\bw$
by moving its $j$-th coordinate to the first position and keeping the
relative order of the remaining coordinates unchanged.  We write $\sum_{j=1}^N\delta_{w_j}=\sum_{i=1}^m l_i\delta_{a_i},$
where $a_1,\ldots,a_m$ are the distinct values among
$w_1,\ldots,w_N$, and  $l_i\ge1$ denotes the multiplicity of
$a_i$, so that $\sum_{i=1}^m l_i=N$.  Recalling \eqref{cp} and using that
$g(t,\be,\alpha)$ has distribution $F(t,\be,\cdot)$, we obtain
\begin{align*}
\int_0^1 |\Pi_t^j(\bw,\alpha)-w_j|^2\,d\alpha
&=
\int_0^1 |g(t,\bw^j,\alpha)-w_j|^2\,d\alpha\\
&=
\int_{\R^3}
|u-w_j|^2
\frac{R_{t,\bar\mu_{\bw}}(du,\{w_j\})}
     {R_{t,\bar\mu_{\bw}}(\R^3,\{w_j\})}\\
&=
\frac{N}{\#\{l:\,w_l=w_j\}}
\int_{\R^3}
|u-w_j|^2
R_{t,\bar\mu_{\bw}}(du,\{w_j\})\\
&=
\frac{N}{l_{i(j)}}
\int_{\R^3}
|u-a_{i(j)}|^2
R_{t,\bar\mu_{\bw}}(du,\{a_{i(j)}\}),
\end{align*}
where $i(j)$ is the unique index such that $a_{i(j)}=w_j$.  Consequently,
\begin{align*}
\int_0^1
\frac{1}{N}
\sum_{j=1}^N|\Pi_t^j(\bw,\alpha)-w_j|^2\,d\alpha
&=
\sum_{j=1}^N
\frac1{l_{i(j)}}\int_{\R^3}
|u-a_{i(j)}|^2
R_{t,\bar\mu_{\bw}}(du,\{a_{i(j)}\})\\
&=
\sum_{i=1}^m
\int_{\R^3}
|u-a_i|^2
R_{t,\bar\mu_{\bw}}(du,\{a_i\})\\
&=
\int_{\R^3\times\R^3}
|u-v|^2
R_{t,\bar\mu_{\bw}}(du,dv)=
\mathcal W_2^2(f_t,\bar\mu_{\bw}),
\end{align*}
which proves (i).
\vip
To prove (ii), we argue similarly. For any
$\bw\in\R^{3N}$,
\begin{align*}
\frac1N
\sum_{j=1}^N
\int_0^1
\phi(\Pi_t^j(\bw,\alpha))\,d\alpha
&=
\sum_{j=1}^N
\frac1{l_{i(j)}}
\int_{\R^3}
\phi(u)R_{t,\bar\mu_{\bw}}(du,\{a_{i(j)}\}),\\
&=
\sum_{i=1}^m
\int_{\R^3}
\phi(u)\,
R_{t,\bar\mu_{\bw}}(du,\{a_i\})\\
&=
\int_{\R^3\times\R^3}
\phi(u)\,
R_{t,\bar\mu_{\bw}}(du,dv)=
\int_{\R^3}\phi(u)\,f_t(du),
\end{align*}
which establishes (ii).

\vip

Finally,  (iii) follows immediately by taking expectations in (ii) with  $\bw$ replaced by the exchangeable random vector $\bY$, and then using exchangeability to identify the expectation for each fixed $j$. Indeed, recalling the definition of $\Pi_t^j$, for any bounded measurable function $\phi:\R^3\to\R$, we have
\begin{align*}
\int_0^1 \phi\big(\Pi_t^j(\bY,\alpha)\big)\,d\alpha
=\int_0^1 \phi\big(g(t,\bY^j,\alpha)\big)\,d\alpha
=\int_{\R^3}\phi(u)\frac{R_{t,\bar\mu_{\bY}}(du,\{Y_j\})}{R_{t,\bar\mu_{\bY}}(\R^3,\{Y_j\})}.
\end{align*}
Since $\bY$ is exchangeable, the random variables $\int_{\rd}\phi(u)\frac{R_{t,\bar{\mu}_{\bY}}\big(du, \{Y_j\}\big)}{R_{t,\bar{\mu}_{\bY}}\big(\rd, \{Y_j\}\big)}$ has the same distribution for $j=1,\ldots,N$. 
Therefore,
\begin{align*}
\E\left[
\int_0^1 \phi\big(\Pi_t^j(\bY,\alpha)\big)\,d\alpha\right]=\frac{1}{N}\sum_{j=1}^N\E\left[
\int_0^1 \phi\big(\Pi_t^j(\bY,\alpha)\big)\,d\alpha\right] 
=\E\left[\frac{1}{N}\sum_{j=1}^N\int_0^1 \phi\big(\Pi_t^j(\bY,\alpha)\big)\,d\alpha\right].
\end{align*}
Applying point {\rm (ii)} with $\bw=\bY$, we obtain
$\frac{1}{N}\sum_{j=1}^N\int_0^1 \phi\big(\Pi_t^j(\bY,\alpha)\big)\,d\alpha=\int_{\R^3}\phi(u)\,f_t(du).$
Taking expectations yields
$\E\!\left[
\int_0^1 \phi\big(\Pi_t^j(\bY,\alpha)\big)\,d\alpha
\right]=\int_{\R^3}\phi(u)\,f_t(du),$
which proves (iii).

\end{proof}

\subsection{Construction of couplings}\label{sec: sec cou}
In this section, we develop a coupling argument inspired by \cite{MR3476628,MR3769742,MR3621429}. Related coupling constructions have been successfully applied to Kac particle systems associated with the Boltzmann and Landau equations for hard potentials and Maxwell molecules. Our setting, however, requires several additional ingredients due to the presence of moderately soft potentials..

\vip

To construct the coupling, consider a family of  i.i.d. $f_0$-distributed random variables $(V_0^i)_{i=1,...,N}$, we enlarge the Poisson random measures $O_{ij}$ (independent of  $(V_0^i)_{i=1,...,N}$), introduced in \eqref{mea: oij},  by adding an auxiliary variable $\alpha\in[0,1]$. More precisely,  let
$(M_{ij})_{1\le i\le j\le N}$ be a family of i.i.d. Poisson random measures with intensity $\frac{1}{N}ds d\alpha dz d\varphi$. For $1\le j<i\le N$, we impose the symmetry
\begin{align}\label{mea: mij}
    M_{ij}(ds,d\alpha,dz,d\varphi)=M_{ji}(ds,d\alpha,dz,d\varphi).
\end{align}

For $i=1,\dots,N$, define
\begin{align}\label{def: new  v}
V^{i}_t=V^i_0 + \sum_{j=1}^N\intot\int_0^1\int_0^\infty\int_0^{2\pi} 
c(V^{i}_\sm,V^{j}_\sm,z,\varphi) M_{ij}(ds,d\alpha,dz,d\varphi).
\end{align}
Since the jump coefficient $c$ does not depend on $\alpha$, the system $\bV=(V^1,...,V^N)$ defined in \eqref{def: new v} has the same law as the original particle system \eqref{mainsde}. Indeed, this follows by setting $O_{ij}(ds,dz,d\varphi)= M_{ij}(ds,[0,1],dz,d\varphi)$. Thus, the introduction of the auxiliary variable $\alpha$ does not change the dynamics in law, it only provides an additional source of randomness that will be used to construct the coupling.

\vip

We now construct a family of cutoff Boltzmann processes coupled with the particle system defined in \eqref{def: new v}. Inspired by \cite{MR3769742,MR3456347,lxz}, for $i=1,\dots,N$, we define
\begin{align}\label{def: wik}
     W^{i,K}_t=V^i_0 + \sum_{j=1}^N\intot\int_0^1\int_0^\infty\int_0^{2\pi} 
c_K(W^{i,K}_\sm,\Pi_s^{j,K}(\bW^K_\sm,\alpha),z,\varphi+\hat\varphi_{i,j,s}) M_{ij}(ds,d\alpha,dz,d\varphi),
\end{align}
where   
$\hat \varphi_{i,j,s}:=\varphi^1_{i,j,s}+\varphi^2_{i,j,s}+\varphi^3_{i,j,s}$  with
\begin{align*}
\varphi_{i,j,s}^1=&\varphi_{0}\big( V^i_\sm-V^j_\sm+\e Y(\alpha),
V^i_\sm-V^j_\sm\big),\\
\varphi^2_{i,j,s}=&\varphi_{0}\big(W_\sm^{i,K}-W_\sm^{j,K}+\e Y(\alpha),
V^i_\sm-V^j_\sm+\e Y(\alpha)\big),\\
\varphi^3_{i,j,s}=&\varphi_{0}\big(W_\sm^{i,K}-W_\sm^{j,K},W_\sm^{i,K}-W_\sm^{j,K}+\e Y(\alpha)\big),
\end{align*}
and $Y(\alpha)$ is an $\alpha$-random variable  with uniform distribution on $B(0,1)$ (independent of everything else), and $0<\epsilon<1$. For $i=j$, we can conventionally set $\hat \varphi_{i,i,s}=0$, since the self-collision coefficient vanishes.

\vip

The process $W^{i,K}$ is designed to follow the cutoff Boltzmann dynamics, while being driven by the same Poisson random measures as the particle system $\bV$. This common source of randomness   allows for a pathwise comparison between $W^{i,K}$ and $V^i$. In contrast to the couplings introduced in \cite{MR3769742,MR3621429,lxz}, we incorporate an additional auxiliary $\alpha$ random variable $Y$ to smooth the empirical measure of $\bW$. This regularization technique, which is particularly useful in the presence of singular collision kernels, is inspired by the ideas developed in \cite{MR3572320,MR3784497}.
\vip

A further difficulty arises from the fact that, unlike in the Nanbu particle system studied in \cite{MR3784497}, the family $(W^{i,K})_{1\le i\le N}$ is not independent. Indeed, the particles are coupled through the common Poisson random measures $(W^{i,K})_{1\le i\le N}$. To overcome this issue, we introduce below a decoupling procedure that enables us to compare compare $(W^{i,K})_{1\le i\le N}$ with a family of independent cutoff Boltzmann processes.

\vip

We introduce a new family of independent Poisson measures $(\tM_{ij}(ds,d\alpha,dz,d\varphi))_{1\le i, j\le N}$  on $[0,\infty)\times [0,1]\times[0,\infty)\times [0,2\pi)$ with intensity measure $\frac{1}{N}ds d\alpha dz d\varphi$ (independent of everything else) for $i=1,...,N$.
We now fix $1\le k\le N$  and divide $\{1,\dots,N\}$ into blocks
\[
I_m:=\{mk+1,\dots,\min\{(m+1)k,N\}\},
\qquad
m=0,\dots,\left\lceil \frac Nk\right\rceil-1 .
\]
For $i\in I_m$, define
\[
N_{ij}(ds,d\alpha,dz,d\varphi)
:=
\begin{cases}
\tM_{ij}(ds,d\alpha,dz,d\varphi),
& j\in I_m,\\
M_{ij}(ds,d\alpha,dz,d\varphi),
& j\notin I_m .
\end{cases}
\]
Then, for each fixed $m$, the family $\bigl(N_{ij}\bigr)_{i\in I_m,\ 1\leq j\leq N}$
consists of independent Poisson random measures with intensity $\frac{1}{N}ds d\alpha dz d\varphi$.

\vip

We now couple the Kac particle system with a family of cutoff Boltzmann
processes. Recall that $\Pi^{j,K}(\bW^K,\alpha)$ denotes the
measurable map introduced in Lemma~\ref{ww2p}, and $\bV$ denotes
the particle system defined in \eqref{mainsde}.   For $i=1,\dots,N$, define
\begin{align}\label{def: w tw}
 \tW^{i,K,k}_t=V^i_0 + \sum_{j=1}^N\intot\int_0^1\int_0^\infty\int_0^{2\pi} 
c_K(\tW^{i,K,k}_\sm,\Pi_s^{j,K}(\bW^K_\sm,\alpha),z,\varphi+\hat  \varphi_{i,j,s}&+\tilde \varphi_{i,j,s}) \notag\\
&
\times N_{ij}(ds,d\alpha,dz,d\varphi),
\end{align}
where   $\tilde{\varphi}_{i,j,s}:= \varphi_0(W^{i,K}_\sm-\Pi_s^{j,K}(\bW^K_\sm,\alpha),\tW^{i,K,k}_\sm-\Pi_s^{j,K}(\bW^K_\sm,\alpha)).$

\begin{rem}
Since $(V^i_0)_{i=1,\dots,N}$ and $(M_{ij})_{i,j=1,\dots,N}$ are exchangeable and  for each $i=1,...,N$, $W_t^{i,K}$  and $V_t^{i}$ are both unique in law, the family $\Big((W^{1,K}_t,V^{1}_t)_{t\geq 0},...,(W^{N,K}_t,V^{N}_t)_{t\geq 0}\Big)$ is exchangeable.
\end{rem}

\begin{lem}\label{coupling1}
Assume \eqref{pa-range}. Let
$f_0\in\mathcal P_q(\mathbb R^3)$ for $q>4$ and have finite Fisher information. Then, for every $i=1,\dots,N$ and every $t\geq0$,
\[
\mathcal L(W_t^{i,K})=\mathcal L(\tilde W_t^{i,K,k})=f_t^K .
\]
Moreover, for each block $I_m$, the processes
$
\bigl\{(\tilde W_t^{i,K,k})\bigr\}_{i\in I_m,}
$
are independent and identically distributed, with common one-time marginal
$f_t^K$ at time $t$.
\end{lem}

\begin{proof}
We first show that, for each fixed $i$, the process $W^{i,K}$ satisfying \eqref{def: wik} has the same
law as the cutoff Boltzmann process defined in \eqref{cbp}.

Define a Poisson random measure $Q_i(ds,d\xi,dz,d\varphi)$ on
$[0,\infty)\times[0,N]\times[0,\infty)\times[0,2\pi)$ by shifting the
$\alpha$-variable in each $M_{ij}$. More precisely, 
\[Q_i(A):=\sum_{j=1}^NM_{ij}\big(\{(s,\alpha,z,\varphi):(s,j-1+\alpha,z,\varphi)\in A\}\big),\]
for every Borel set $A\subset [0,\infty)\times[0,N]\times[0,\infty)\times[0,2\pi)$.
Then $ Q_i(ds,d\xi,dz,d\varphi)$ is a Poisson random measure  with intensity
$\frac{1}{N}ds  d\xi dz d\varphi$ on $[0,\infty)\times [0,N]\times[0,\infty)\times [0,2\pi)$, independent of $(V_0^i)_{i=1,\dots N}$. 
For $\xi\in[j-1,j)$, set $\Pi_s^K(\mathbf W_{s-}^K,\xi):=\Pi_s^{j,K}(\mathbf W_{s-}^K,\xi-j+1).
$
With this notation, $W^{i,K}$ can be rewritten as
\begin{align*}
W^{i,K}_t=&V^i_0 + \intot\int_0^N\int_0^\infty\int_0^{2\pi} 
c(W^{i,K}_\sm,\Pi_s^K(\bW^K_\sm,\xi),z,\varphi+\hat\varphi_{i,s}(\xi)) Q_i(ds,d\xi,dz,d\varphi),
\end{align*}
where $\hat\varphi_{i,s}(\xi)$ denotes the corresponding angular shift $\hat\varphi_{i,\lceil \xi \rceil,s}$ in \eqref{def: wik}.

Since the map $\Pi_s^K(\mathbf W_{s-}^K,\cdot)$ pushes forward the uniform
measure on $[0,N]$ onto $f_s^K$, and since $c_K$ is $2\pi$-periodic in
the angular variable $\varphi$, the angular shift does not change the compensator of the
driving Poisson random measure. It follows that $W^{i,K}$ satisfies the martingale
problem associated with the cutoff Boltzmann process \eqref{cbp}. By uniqueness
of the cutoff Boltzmann process, we conclude that  
$$
\mathcal L(W_t^{i,K})=f_t^K, \quad t\ge0.
$$

It remains to identify the law of $\tilde W^{i,K,k}$. For any fixed $i$, it is straightforward to verify that
\[
\mathcal L(\tilde W_t^{i,K,k})
=
\mathcal L(W_t^{i,K})
=
f_t^K,
\qquad t\ge0.
\]
Indeed, the family of driving Poisson random measures
$(N_{ij})_{1\le j\le N}$ has the same law and the same intensity measure as
$(M_{ij})_{1\le j\le N}$. Consequently, $\tilde W^{i,K,k}$ admits the same
martingale characterization as $W^{i,K}$, and thus has the same distribution.
\vip
Finally, let $I_m$ be a fixed block. By construction, for distinct
$i,i'\in I_m$, the families of Poisson random measures
$(N_{ij})_{1\le j\le N}$ and $(N_{i'j})_{1\le j\le N}$ are independent.
Since the processes $\bigl(\tilde W^{i,K,k}\bigr)_{i\in I_m}$ are built from
independent initial conditions and independent driving Poisson random measures,
they are mutually independent. Together with the identification of their
common marginal law, this completes the proof.

 \end{proof}

\begin{rem}\label{rem: exchang}
Fix a block $I_m$ and an index $i\in I_m$. By construction, for every
$l\notin I_m$, the family of noises $(M_{lj})_{1\le j\le N}$ shares exactly one
common noise with $(N_{ij})_{1\le j\le N}$. In contrast, if $l\in I_m$ and
$l\ne i$, then these two families have no common noise.

The exchangeability properties established above yield the following  exchangeability properties of the conditional law. For any
$0\le s\le T$, if  $l_1,l_2\notin I_m$ 
, then
\begin{align*}
\cL(W^{i,K}_s,\tW^{i,K,k}_s,\Pi_s^{l_1,K}(\bW^K_s,\alpha) ) = \cL(W^{i,K}_s,\tW^{i,K,k}_s,\Pi_s^{l_2,K}(\bW^K_s,\alpha) ).
\end{align*}
 Moreover,
\begin{align*}
&\cL\Big(
\Pi_s^{l_1,K}(\bW^K_s,\alpha)
\,\Big|\,
\sigma\big(\{W^{i,K}_s,\tW^{i,K,k}_s\}_{i\in I_m}\big)
\Big) \\
&\hspace{2cm}
=
\cL\Big(
\Pi_s^{l_2,K}(\bW^K_s,\alpha)
\,\Big|\,
\sigma\big(\{W^{i,K}_s,\tW^{i,K,k}_s\}_{i\in I_m}\big)
\Big).
\end{align*}
Similarly, if $l_1,l_2\in I_m\setminus\{i\}$, then
\begin{align*}
&\cL\Big(
\Pi_s^{l_1,K}(\bW^K_s,\alpha)
\,\Big|\,
\sigma\big(W^{i,K}_s,\tW^{i,K,k}_s\big)
\Big)\\
&\hspace{2cm}
=
\cL\Big(
\Pi_s^{l_2,K}(\bW^K_s,\alpha)
\,\Big|\,
\sigma\big(W^{i,K}_s,\tW^{i,K,k}_s\big)
\Big).
\end{align*}
Finally, fix $i_1,i_2\in I_m$ and $l\notin I_m$. Then, for any nonnegative symmetric measurable function
$\Psi:(\R^3)^k\to \R_{+}$ and any nonnegative measurable function
$F:(\R^3)^3\to \R_{+}$, we have
\begin{align*}
&\E\Big[
F\big(
V_s^{l},
\Pi_s^{l,K}(\bW^K_s,\alpha),
\tW^{i_1,K,k}_s
\big)
\Psi\big(\{\tW^{j,K,k}_s\}_{j\in I_m}\big)
\Big]\\
&\hspace{2cm}
=
\E\Big[
F\big(
V_s^{l},
\Pi_s^{l,K}(\bW^K_s,\alpha),
\tW^{i_2,K,k}_s
\big)
\Psi\big(\{\tW^{j,K,k}_s\}_{j\in I_m}\big)
\Big].
\end{align*}

\end{rem}

\begin{lem}\label{thele}
Assume \eqref{con} and \eqref{pa-range}. Let $N\ge1$, $1\leq k\leq N$, and $K\ge1$. Consider $(W_t^{i,K})_{t\in[0,T]}, i=1,...,N$ solving \eqref{def: wik} and $(\tilde W_t^{i,K,k})_{t\in[0,T]},  i=1,...,N$ solving \eqref{def: w tw}. Then, for every $T>0$, there exists a constant $C_{T,f_0}>0,$ independent of $N$, $k$, and $K$, such that for all $i=1,\dots,N$ and all $t\in[0,T]$,
\begin{align}
\E\left[
\left|W_t^{i,K}-\tilde W_t^{i,K,k}\right|^2
\right]
\leq&
C_{T,f_0}\left(
\frac{k}{N}
+
\frac{1}{K^{1/\nu-1}}
\right),\label{decoup}\\
\E\left[
\left|W_t^{i,K}-\tilde W_t^{i,K,k}\right|
\right]
\leq&
C_{T,f_0}\left(
\frac{k}{N}
+
\frac{1}{K^{1/\nu-1}}
\right).\label{decoup1}
\end{align}
\end{lem}
\begin{rem}
It is natural that the estimates in \eqref{decoup} and \eqref{decoup1} have the same order of error $O(k/N)$. Indeed, this is consistent with the elementary behavior of a Poisson process with small intensity parameter. If $(N_t)_{t\ge 0}$ is a Poisson process with rate $1$, then both $\E[N_t]$ and $\E[|N_t|^2]$ are of order $O(t)$ as $t\to 0$.
\end{rem}

\begin{preuve} 
By exchangeability, it suffices to establish the desired estimate for $i=1$. For $j=1,\dots,N$, denote 
\begin{align*}
c_{K}^{1j}
&:=
c_K\Bigl(
W_{s-}^{1,K},
\Pi_s^{j,K}(\mathbf W_{s-}^K,\alpha),
z,
\varphi+\hat\varphi_{1,j,s}
\Bigr),
\\
\tilde c_{K}^{1j}
&:=
c_K\Bigl(
\tilde W_{s-}^{1,K,k},
\Pi_s^{j,K}(\mathbf W_{s-}^K,\alpha),
z,
\varphi+\hat\varphi_{1,j,s}
+\tilde\varphi_{1,j,s}
\Bigr).
\end{align*}
Since $N_{1j}=\tilde M_{1j}$ for $1\leq j\leq k$, and
$N_{1j}=M_{1j}$ for $j>k$, we have
\begin{align*}
&W_t^{1,K}-\tW_t^{1,K,k}\\
&=\sum_{j=1}^k\intot\int_0^1\int_0^\infty\int_0^{2\pi} 
c_K^{1j} M_{1j}(ds,d\alpha,dz,d\varphi)-\sum_{j=1}^k\intot\int_0^1\int_0^\infty\int_0^{2\pi} 
\tilde{c}^{1j}_K \tM_{1j}(ds,d\alpha,dz,d\varphi)\\
&\hskip2cm+\sum_{j=k+1}^N\intot\int_0^1\int_0^\infty\int_0^{2\pi} 
(c_K^{1j}-\tilde{c}_K^{1j}) M_{1j}(ds,d\alpha,dz,d\varphi).
\end{align*}
Applying It\^o's formula for jump processes to $|W_t^{1,K}-\tilde W_t^{1,K,k}|^2$ , we have 
\begin{align*}
    |W_t^{1,K}-\tW_t^{1,K,k}|^2=J^{1}_t+J^{2}_t+J^{3}_t,
\end{align*}
where
\begin{align*}
   & J^{1}_t:=\sum_{j=1}^k\intot\int_0^1\int_0^\infty\int_0^{2\pi} |W_{s-}^{1,K}-\tW_{s-}^{1,K,k}+c^{1j}_K|^2-|W_{s-}^{1,K}-\tW_{s-}^{1,K,k}|^2M_{1j}(ds,d\alpha,dz,d\varphi),\\
     &J^{2}_t:=\sum_{j=1}^k\intot\int_0^1\int_0^\infty\int_0^{2\pi} |W_{s-}^{1,K}-\tW_s^{1,K,k}-\tilde{c}^{1j}_K|^2-|W_s^{1,K}-\tW_{s-}^{1,K,k}|^2 \tilde M_{1j}(ds,d\alpha,dz,d\varphi),\\
     &J^{3}_t:=\sum_{j=k+1}^N\intot\int_0^1\int_0^\infty\int_0^{2\pi} |W_{s-}^{1,K}-\tW_{s-}^{1,K,k}+c^{1j}_K-\tilde{c}^{1j}_K|^2-|W_{s-}^{1,K}-\tW_{s-}^{1,K,k}|^2M_{1j}(ds,d\alpha,dz,d\varphi).
\end{align*}
We first estimate $\E[J_t^1]$. Taking expectations, we have
\[\E[J_t^1]=\frac1N\sum_{j=1}^k\int_0^t\int_0^1\int_0^\infty\int_0^{2\pi}\E\Big[2c_K^{1j}\cdot\big(W_s^{1,K}-\tilde W_s^{1,K,k}\big)+|c_K^{1j}|^2\Big]\,d\varphi\,dz\,d\alpha\,ds. \]
Applying Lemma~\ref{lem:c}-(ii) and Young's inequality, we obtain 
\begin{align*}
\E[J_t^1]
&\le\frac{C}{N}\sum_{j=1}^k\intot\int_0^1 \E\Big[|W_s^{1,K}-\tW_s^{1,K,k}||W_{s}^{1,K}-
\Pi_s^{j,K}(\mathbf W_{s}^K,\alpha)|^{1+\gamma}\\
&\hskip4cm+|W_{s}^{1,K}-
\Pi_s^{j,K}(\mathbf W_{s}^K,\alpha)|^{2+\gamma}\Big]ds d\alpha\\
&\le \frac{C}{N}\sum_{j=1}^k\int_0^T\int_0^1 \E\Big[1+|W_s^{1,K}|^2+|\tW_s^{1,K,k}|^2+|
\Pi_s^{j,K}(\mathbf W_{s}^K,\alpha)|^{2}\Big]d\alpha ds .
\end{align*}
Since Lemma \ref{ww2p}-(iii),  and using that $\ca{L}(W_s^{1,K})=\ca{L}(\tilde W_s^{1,K,k})=f_s^K$ and $\sup_{s\in[0,T]}m_2(f_s^K)<\infty$,  
we have 
\[\int_0^1 \E\Big[1+|W_s^{1,K}|^2+|\tW_s^{1,K,k}|^2+|
\Pi_s^{j,K}(\mathbf W_{s}^K,\alpha)|^{2}\Big]d\alpha<C(1+m_2(f_s^K))<\infty.\]
Therefore, $\E[J_t^1]\le \frac{CT k}{N}$.
The same argument also yields $\sup_{0\le t\le T}   \E[J_t^2]\le \frac{CT k}{N}$.

It remains to control $\E[J_t^3]$, where both processes are driven by the same
Poisson random measures. Taking expectations gives
\begin{align*}
\mathbb E[J_t^3]
&=
\frac{1}{N}\sum_{j=k+1}^N
\int_0^t\int_0^1\int_0^\infty\int_0^{2\pi}
\mathbb E\left[
2(W_s^{1,K}-\tW_s^{1,K,k})\cdot(c_K^{1j}-\tilde c_K^{1j})
+
|c_K^{1j}-\tilde c_K^{1j}|^2
\right]
\,ds\,d\alpha\,dz\,d\varphi .
\end{align*}
We first notice that by Lemma \ref{ww2p}-(ii) and  Lemma \ref{lem: prop fk}-(i),
\begin{align}\label{Est:5.1+4.2}
    \int_0^1\frac{1}{N}\sum_{j=1}^N|\tW_s^{1,K,k}-\Pi_s^{j,K}(\bW^K_s,\alpha)|^\gamma d\alpha = \int_{\R^3}|\tW_s^{1,K,k}-u|^\gamma f_t^K(du)\le 1+I_1(f_0).
\end{align}
Using \eqref{ineq: KKp22cut}  in Lemma~\ref{lem: 3ccut} with $v=W_s^{1,K}$, $\tv=\tW_s^{1,K,k}$, $v_*=\tv_*=\Pi_s^{j,K}(\bW^K_s,\alpha)$, together with $\sup_{s\in[0,T]}m_2(f_s^K)<\infty$, we deduce
\begin{align*}
    &\frac{1}{N}\sum_{j=k+1}^N\int_0^1\int_0^\infty\int_0^{2\pi}
\mathbb E\left[
|c_K^{1j}-\tilde c_K^{1j}|^2
\right]
\,d\alpha\,dz\,d\varphi\\
\le& C\E\Big[|W_s^{1,K}-\tW_s^{1,K,k}|^2 \int_0^1\frac{1}{N}\sum_{j=1}^N|\tW_s^{1,K,k}-\Pi_s^{j,K}(\bW^K_s,\alpha)|^\gamma d\alpha\Big]\\
&+\frac{C}{K^{1/\nu-1}}\E\left[1+|W_s^{1,K}|^{2}+|\tW_s^{1,K,k}|^{2}+\int_0^1\frac{1}{N}\sum_{j=1}^N|\Pi_s^{j,K}(\bW^K_s,\alpha)|^2 d\alpha\right]\\
\le& C_{f_0}\E[|W_s^{1,K}-\tW_s^{1,K,k}|^2]+\frac{C}{K^{1/\nu-1}}.
\end{align*}
Proceeding as in the previous step and applying \eqref{ineq: KKp22cut1} from Lemma~\ref{lem: 3ccut}  together with  \eqref{Est:5.1+4.2}, we deduce that
\begin{align*}
& \frac{1}{N}\sum_{j=k+1}^N\mathbb E\left[
|W_s^{1,K}-\tW_s^{1,K,k}|\Big|\int_0^1\int_0^\infty\int_0^{2\pi}(c_K^{1j}-\tilde c_K^{1j})\,d\alpha\,dz\,d\varphi\Big|
\right]\\
\le& C\mathbb E\left[
|W_s^{1,K}-\tW_s^{1,K,k}|^2\Big|\frac{1}{N}\sum_{j=1}^N\int_0^1 |\tW_s^{1,K,k}-\Pi_s^{j,K}(\bW^K_\sm,\alpha)|^\gamma\,d\alpha\Big|
\right]\\
\le& C_{f_0}\mathbb E\left[
|W_s^{1,K}-\tW_s^{1,K,k}|^2\right].
\end{align*}
Therefore,
\begin{align*}
\frac{1}{N}\sum_{j=k+1}^N\int_0^1\int_0^\infty\int_0^{2\pi}
\mathbb E\left[
(W_s^{1,K}-\tW_s^{1,K,k})\cdot(c_K^{1j}-\tilde c_K^{1j})
\right]\,d\alpha\,dz\,d\varphi 
\le&C_{f_0}\mathbb E\left[
|W_s^{1,K}-\tW_s^{1,K,k}|^2\right].
\end{align*}
Combining the above estimates,  we obtain
\begin{align*}
    \E[J_t^3]
    \le& C_{f_0}\int_0^t \E[|W_s^{1,K}-\tW_s^{1,K,k}|^2] ds +\frac{CT}{K^{1/\nu-1}}.
\end{align*}
Following from Gr\"onwall's inequality, we obtain  for all $t\in[0,T]$,
\begin{align*}
\sup_{0\le t\le T}\E[|W_t^{1,K}-\tW_t^{1,K,k}|^2]\le C_{T,f_0} \Big(\frac{k}{N}+\frac{1}{K^{1/\nu-1}}\Big).    \end{align*}
To prove \eqref{decoup1}, we proceed as above and decompose
\begin{align*}
    |W_t^{1,K}-\tW_t^{1,K,k}|=L^{1}_t+L^{2}_t+L^{3}_t,
\end{align*}
where
\begin{align*}
   & L^{1}_t:=\sum_{j=1}^k\intot\int_0^1\int_0^\infty\int_0^{2\pi} |W_{s-}^{1,K}-\tW_{s-}^{1,K,k}+c^{1j}_K|-|W_{s-}^{1,K}-\tW_{s-}^{1,K,k}|M_{1j}(ds,d\alpha,dz,d\varphi),\\
     &L^{2}_t:=\sum_{j=1}^k\intot\int_0^1\int_0^\infty\int_0^{2\pi} |W_{s-}^{1,K}-\tW_s^{1,K,k}-\tilde{c}^{1j}_K|-|W_\sm^{1,K}-\tW_{s-}^{1,K,k}|\tilde M_{1j}(ds,d\alpha,dz,d\varphi),\\
     &L^{3}_t:=\sum_{j=k+1}^N\intot\int_0^1\int_0^\infty\int_0^{2\pi} |W_{s-}^{1,K}-\tW_{s-}^{1,K,k}+c^{1j}_K-\tilde{c}^{1j}_K|-|W_{s-}^{1,K}-\tW_{s-}^{1,K,k}|M_{1j}(ds,d\alpha,dz,d\varphi).
\end{align*}
Taking expectations and applying Lemma~\ref{lem:c}-(ii), we obtain
\begin{align*}
    \E[L_t^1]& \le  \frac{1}{N}\sum_{j=1}^k\intot\int_0^1\int_0^\infty\int_0^{2\pi} \E[|c^{1j}_K|]ds d\alpha dz d\varphi\\
  &\le \frac{C }{N}\sum_{j=1}^k\int_0^T\E\Big[\int_0^1|W_s^{1,K}-\Pi_s^{j,K}(\bW^K_s,\alpha)|^{1+\gamma}d\alpha\Big] ds.
\end{align*}
Proceeding exactly as in the estimate of $\E[J_t^1]$,  we deduce that
\[
  \E[L_t^1] \le \frac{C}{N}\sum_{j=1}^k\int_0^T\int_0^1 \E\Big[1+|W_s^{1,K}|^2+|
\Pi_s^{j,K}(\mathbf W_{s}^K,\alpha)|^{2}\Big]ds d\alpha\le \frac{CTk}{N}. \]
The same argument yields $\sup_{0\le t\le T}   \E[L_t^2]\le \frac{CT k}{N}$. 
To estimate $\E[L_t^3]$, we have
\begin{align*}
\mathbb E[L_t^3]
&\le
\frac{1}{N}\sum_{j=k+1}^N
\int_0^t\int_0^1\int_0^\infty\int_0^{2\pi}
\mathbb E\left[
|c_K^{1j}-\tilde c_K^{1j}|
\right]
\,ds\,d\alpha\,dz\,d\varphi .
\end{align*}
Using inequality \eqref{ineq: KKp22cut3}  from Lemma~\ref{lem: 3ccut}, together with the uniform moment estimates for $f_s^K$ in Lemma \ref{lem: prop fk}-(i), we deduce
\begin{align*}
    \E[L_t^3]
    \le& C_{f_0}\int_0^t\E[|W_s^{1,K}-\tW_s^{1,K,k}|] ds +\frac{C}{K^{1/\nu-1}}\int_0^T\int_0^1\E\Big[1+|W_s^{1,K}|^{2}\\
    &\hskip4.5cm+|\tW_{s}^{1,K,k}|^{2}+\frac{1}{N}\sum_{j=1}^N|
\Pi_s^{j,K}(\mathbf W_{s}^K,\alpha)|^{2}\Big]ds d\alpha\\
    \le& C_{f_0}\int_0^t\E[|W_s^{i,K}-\tW_s^{i,K,k}|] ds +\frac{CT}{K^{1/\nu-1}}.
\end{align*}
Gr\"onwall's inequality then gives, for all $t\in[0,T]$,

\begin{align*}
\sup_{0\le t\le T}\E[|W_t^{1,K}-\tW_t^{1,K,k}|]\le C_{T,f_0} \Big(\frac{k}{N}+\frac{1}{K^{1/\nu-1}}\Big).    \end{align*}
\end{preuve}

\subsection{\texorpdfstring{Large-$N$ limit of many Boltzmann processes}{Large-N limit of many Boltzmann processes}}
We estimate $\cW_2^2\left(f_t^K,\frac{1}{N}\sum_{i=1}^N \delta_{W^{i,K}_t} \right)$, where $(W^{i,K}_t)_{i=1,\dots,N,\ t\ge0}$ is defined in \eqref{def: wik}. By Lemma \ref{ww2p}-(i) and symmetry, for $t\ge0$,
\begin{align*}
  \E \left[  \int_0^1 \frac{1}{N}\sum_{j=1}^N|\Pi_t^{j,K}(\bW^K_t,\alpha)-W^{j,K}_t|^2 d\alpha \right]=\E\left[\cW_2^2\Big(f_t^K,\frac{1}{N}\sum_{i=1}^N \delta_{W^{i,K}_t} \Big)\right].
\end{align*}

Fournier and Guillin \cite[Theorem 1]{MR3383341} proved that, for a probability measure $f$ on $\mathbb{R}^3$ and a sequence of $f$-distributed, $\mathbb{R}^3$-valued i.i.d. random variables $(X_1,\cdots,X_N)$,
\[\E\Big[\cW_2^2\Big(f,\frac{1}{N}\sum_{i=1}^N\delta_{X_i}\Big)\Big]\le\frac{C_p(\int |v|^pf(dv))^{2/p}}{N^{1/2}},\]
for any $p>4$ with $C_p>0$.  
We thus have to prove the following result making use of the second coupling technique mentioned  previously because the family of processes  $(W_t^{i,K})_{t\ge0}$ are not independent.

\begin{prop}\label{thep}
Assume \eqref{con} and \eqref{pa-range}. Let  $K\ge1$. Assume that
$f_0\in\cP_q(\R^3)$ for some $q>4$ and that $f_0$ has finite Fisher
information. Let $(f_t^K)_{t\ge0}$ be the unique weak solution to the cutoff
Boltzmann equation \eqref{Bol} associated with the collision kernel $B_K$
defined in \eqref{Def: kernel cut}. Then, for every $t\in[0,T]$,
\beq\label{spe}
\E\left[\cW_2^2\Big(f_t^K,\frac{1}{N}\sum_{i=1}^N \delta_{W^{i,K}_t} \Big)\right]\le C\Big(N^{-1/3}+\frac{1}{K^{1/\nu-1}}\Big).
\eeq
\end{prop}

In order to prove Proposition  \ref{thep},  let's recall the  following lemma which comes from \cite[Lemma 21]{MR3621429}, see also \cite[Lemma 5.6]{lxz}.
\begin{lem}\label{line}
Let $N\ge2$, $\mu\in\cP_2(\rd)$. Let $W_1,...,W_N$ be a family of exchangeable $\rd$-valued random variables. Then for any $1\le k_0\le N$, we set $\mu_{k_0}:=k_0^{-1}\sum_{i=1}^{k_0} \delta_{W_i}$,
\[\E\left[\cW_2^2(\mu_N, \mu)\right]\le \E\left[\cW_2^2(\mu_{k_0}, \mu)\right]+\frac{l}{N}\Big(2m_2(\mu)+2\E[|W_1|^2]\Big),\]
where $r, l$  are the unique non-negative integers satisfying $N=k_0r+l$ and $l\le k_0-1$.
\end{lem}

We now move to verify Proposition \ref{thep}.
\begin{proof}[Proof of Proposition \ref{thep}]

Fix $k_0\le N$.  For $t\in[0,T]$, applying Lemma \ref{line} with $W_i=W_t^{i,K}$, $i=1,...,N$, and $\mu=f_t^K$, and using Lemma \ref{thele}, we obtain
\begin{align*}
    \E\Big[\cW_2^2\left(f_t^K,\frac{1}{N}\sum_{i=1}^N \delta_{W^{i,K}_t} \right)\Big]\le& \E\left[\cW_2^2\left(\frac{1}{k_0}\sum_{i=1}^{k_0}\delta_{W_t^{i,K}},f_t^K\right)\right]
    +\frac{k_0}{N}\Big(4m_2(f_t^K)\Big)\\
    \le & 2\E[|W^{1,K}_t-\tW^{1,K,k_0}_t|^2]+2\E\left[\cW_2^2\left(\frac{1}{k_0}\sum_{i=1}^{k_0}\delta_{\tW_t^{i,K,k_0}},f_t^K\right)\right]+12\frac{k_0}{N}.
\end{align*}
By Lemma \ref{lem: prop mom}, we know that $f_t^K\in\cP_q(\rd) $ with $q>4$, hence by Theorem 1 in \cite{MR3383341}, we have 
\[\E\left[\cW_2^2\left(\frac{1}{k_0}\sum_{i=1}^{k_0}\delta_{\tW_t^{i,K,k_0}},f_t^K\right)\right]\le Ck_0^{-1/2}.\] Using Lemma \ref{thele}, we finally conclude that 
\[\E\Big[\cW_2^2\left(f_t^K,\frac{1}{N}\sum_{i=1}^N \delta_{W^{i,K}_t} \right)\Big]\le C \frac{k_0}{N}+\frac{C}{\sqrt {k_0}}+\frac{C}{K^{1/\nu-1}}.\]
This finishes the proof by taking $k_0\sim N^{2/3}$.

\end{proof}

\section{Bound on the Lebesgue norm of an empirical measure}\label{ttt}
To control the negative moments involving the processes $W^{i,K}$, one possible strategy is to prove that their empirical measure has suitable regularity after an appropriate smoothing. However, unlike in \cite{MR3784497}, the processes $W^{i,K}$ are not independent. Fortunately, Lemma \ref{thele} shows that they are close to an independent family $\tW^{i,K,k}$. In this section, we prove that, after regularization, the empirical measure of $\tW^{i,K,k}$ is almost bounded in $L^2$, with an error depending on $k$.

\begin{nota}\label{nnn}
For $\epsilon>0$, set
\[
\psi_\epsilon(x)=\frac{3}{4\pi\epsilon^3}\bbd{1}_{\{|x|\le \epsilon\}},
\]
and write $\mu*\psi_\epsilon$ for the corresponding blob approximation of a measure $\mu$.
\end{nota}

\vskip1mm

The goal of this section is to prove the following crucial estimate.
\begin{prop}\label{norm-bound}
For $t\ge0$, let $(X^{i}_t)_{i=1,\dots,k}$ be i.i.d. random variables with distribution $\tf_t\in L^2(\R^3)\cap \cP_q(\R^3)$ with $q>4$ and set
$\mu_{{\bf {X}_t}}^k=k^{-1}\sum_{i=1}^k \delta_{X_t^{i}}$.
Fix $\delta\in(0,1)$, recall $\epsilon_k=k^{-(1-\delta)/3}$ and define
$\bar\mu_{{\bf {X}_t}}^k:=\mu_{{\bf {X}_t}}^k*\psi_{\epsilon_k}$, where $\psi_\epsilon$ was defined in Notation \ref{nnn}. Then we have
 $$
 \bb{P}\Big(\|\bar\mu_{{\bf {X}_t}}^k\|_{L^2}\le 144 \|\tf_t\|_{L^2})\Big)\ge 1- C_{q,\delta}k^{1-\delta q/3},
 $$
 where $C_{q,\delta}$ is independent of $k$.
\end{prop}
Recall the definition of $\tW^{i,K,k}_t$ in \eqref{def: w tw} and assume $1\le k\le N/2$ in Subsection \ref{sec: sec cou}. We denote
$\mu_{{\bf {\tW}^{K,k}_t}}^{k}:=k^{-1}\sum_{i=k+1}^{2k} \delta_{\tW^{i,K,k}_t}$, $\bar\mu_{{\bf {\tW}^{K,k}_t}}^{k}= \mu_{{\bf {\tW}^{K,k}_t}}^{k}*\psi_{\epsilon_k}$, and 
$$\Omega_{t,k}=\{\|\bar\mu_{{\bf {\tW}}^K_t}^k\|_{L^2}\le 144\|f^K_t\|_{L^2}\}.$$
Since $\tW^{i,K,k}_t$ are i.i.d. for indices $1\le i\le k$, Proposition \ref{norm-bound} directly gives the following corollary.
\begin{cor}\label{cor: norm-bound}
For  fixed $1\le k\le N/2$ and $t\ge0$, then we have
\begin{align*}
    \bb{P}(\Omega_{t,k})\ge 1- C_{q,\delta}k^{1-\delta q/3}.
\end{align*}
\end{cor}

Throughout the section, we fix $k\ge 1$, $\delta\in(0,1)$,  $\epsilon_k=k^{-(1-\delta)/3}$ and adopt the assumptions and notations of Proposition \ref{norm-bound}. 

\vskip1mm

\subsection{Proof of Proposition \ref{norm-bound}}

Consider the partition $\scr{P}_k$ of  $\R^3$ in cubes with side length $\e_k$ and its subset $\scr{P}_k^\delta$  consisting of cubes that have non-empty intersection with  $B(0, k^{\delta/3})$. Then we observe that $\# (\scr{P}_k^\delta)\le (2(k^{\delta/3}+\epsilon_k)\epsilon_k^{-1})^3\le 64 k^\delta\epsilon_k^{-3}=64k$. We split the proof into several steps.

\vskip1mm 

{\it Step 1.}   For  $(x_1,..., x_k)\in(B(0,k^{\delta/3}))^k$, we denote the empirical measure of $(x_1,..., x_k)$ by  $\mu_{{\bf x}}^k=k^{-1}\sum_{i=1}^k \delta_{x_i}$.
The goal of this step is to show that
\[\| \mu_{{\bf x}}^k*\psi_{\epsilon_k} \|_{L^2}\le 36 \Big(k^{-2}\epsilon_k^{-3}\sum_{D\in\scr{P}_k^\delta}(\#\{i\in\{1,...,k\} : x_i\in D\})^2\Big)^{1/2}.\]
Indeed, recalling that $\psi_{\epsilon_k}(x)=(3/(4\pi\epsilon_k^3))\bbd{1}_{\{|x|\le \epsilon_k\}}$, we observe that
\begin{align*}
\mu_{{\bf x}}^k*\psi_{\epsilon_k}(v)
& = k^{-1} \sum_{i=1}^{k} \psi_{\epsilon_k}(v-x_i)\\
& =  \frac{3}{4\pi k\epsilon_k^3} \# \big\{ i\in\{1,...,k\}: x_i\in B(v,\epsilon_k) \big\}.
\end{align*}
Hence,
$$
\mu_{{\bf x}}^k*\psi_{\epsilon_k}(v) \le \frac{3}{4\pi k\epsilon_k^3} \sum_{D\in\scr{P}_k^\delta}\# \big \{  i\in\{1,...,k\} :  x_i\in D \big \} \bbd{1}_{\{D\cap B(v, \epsilon_k) \neq \emptyset\}}.
$$
We then deduce that
\begin{align*}
&\|\mu_{{\bf x}}^k*\psi_{\epsilon_k}\|_{L^2}\\
&\le \frac{3}{4\pi k\epsilon_k^3} \| \sum_{D\in\scr{P}_k^\delta}\# \{ i\in\{1,...,k\}:  x_i\in D\} \bbd{1}_{\{D\cap B(\cdot, \epsilon_k)\neq \emptyset\}} \|_{L^2}.
 \end{align*}
On the other hand,  let $A := \| \sum_{D\in\scr{P}_k^\delta}\# \{i\in\{1,...,k\} :  x_i\in D\} \bbd{1}_{\{D\cap B(\cdot, \epsilon_k)\neq \emptyset\}} \|_{L^2}$, then
\begin{align*}
A^2
&=\int_{\bb{R}^3}\Big(\sum_{D\in\scr{P}_k^\delta} \# \{i:  x_i\in D\} \bbd{1}_{\{D\cap B(v, \epsilon_k)\neq \emptyset\}}\Big)^2 dv\\
&=\int_{\bb{R}^3}\Big(\sum_{D,D^\prime\in\scr{P}_k^\delta} \# \{i:  x_i\in D\}\# \{i:  x_i\in D^\prime\} \bbd{1}_{\{D\cap B(v, \epsilon_k)\neq \emptyset, D^\prime\cap B(v, \epsilon_k)\neq \emptyset\}}\Big) dv.
\end{align*}
From  $x^2+y^2\ge 2xy$ and a symmetry argument, we see that
\begin{align*}
A^2 \le \sum_{D\in\scr{P}_k^\delta} (\# \{i: x_i\in D\})^2 \int_{\bb{R}^3}
 \bbd{1}_{\{D\cap B(v, \epsilon_k)\neq \emptyset\}} \sum_{D^\prime\in\scr{P}_k^\delta}  \bbd{1}_{\{D^\prime \cap B(v, \epsilon_k)\neq \emptyset \}} dv.
\end{align*}
But,  for each $v\in\bb{R}^3$,
$\sum_{D^\prime\in\scr{P}_k^\delta}  \bbd{1}_{\{D^\prime \cap B(v, \epsilon_k)\neq \emptyset \}}=\#\{D^\prime\in\scr{P}_k^\delta: D^\prime \cap B(v, \epsilon_k)\neq \emptyset\} \le 3^3.$
And for each $D\in\scr{P}_k^\delta$, $\{v\in\bb{R}^3: D\cap B(v, \epsilon_k)\neq \emptyset\}$ is included by a ball of radius $3\e_k$. Therefore,
 $\int_{\bb{R}^3}
 \bbd{1}_{\{D\cap B(v, \epsilon_k)\neq \emptyset\}} dv \le 4\pi(3\epsilon_k)^3/3$. Hence,
 $$
 A^2 \le \frac{3^3 4\pi (3\e_k)^3}{3}\sum_{D\in\scr{P}_k^\delta} \big(\# \{i: x_i\in D\} \big)^2.
 $$
 Consequently,
 \begin{align*}
 \|\mu_{{\bf x}}^k*\psi_{\epsilon_k}(v)\|_{L^2}
 &\le \frac{3}{4\pi k\epsilon_k^3} A \\
 &\le   \big(\frac{3}{4\pi} \big)^{1/2} (9)^{3/2}
\Big(k^{-2}\epsilon_k^{-3}\sum_{D\in\scr{P}_k^\delta} \big(\#\{ i: x_i\in D\} \big)^2\Big)^{1/2}.
 \end{align*}
Since $(9)^{3/2}=27\le 36$, this ends the step.

\vskip1mm 

{\it Step 2.} In this step, we 
will show that  there are some constants $C>0$ and $c>0$ (depending on $\delta$ and $f_0$, such that for all $0\le t\le T$
 \[\bb{P}[(\Omega_{t,k}^2)^c]\le C \exp{(-ck^{\delta/2})},\]
 where \[\Omega_{t,k}^2:=\left\{k^{-2}\epsilon_k^{-3}\sum_{D\in\scr{P}_k^\delta}\Big(\#\{i\in\{1,...,k\} : X_t^i\in D\}\Big)^2 \le 8 \| \tf_t\|_{L^2}^2 \right\}.\]
 To this end, we introduce, for $D\in\scr{P}_k^\delta$, $A_D=\#\{i: X_t^i\in D\}$. Then $A_D\sim B(k, \tf_t(D))$
and we have
 \beq\label{Prop-proba}
 \bb{P}(A_D\ge x)\le \exp(-x/8) \quad \text{for\, all}\quad  x\ge 2k\tf_t(D).
 \eeq
Indeed, $\bb{P}(A_D\ge x)\le e^{-x} \bb{E}[\exp(A_D)]=e^{-x}\exp[k\log(1+\tf_t(D)(e-1))] \le e^{-x}\exp[k(e-1)\tf_t(D)]$. If  $x\ge 2k\tf_t(D)$, we thus have
$$
\bb{P}(A_D\ge x) \le \exp[-x+x(e-1)/2] \le \exp(-x/8).
$$
Next, it follows  from  the H\"{o}lder inequality that
$$
\|\tf_t\|_{L^2}^2 \ge \sum_{D\in\scr{P}_k^\delta} \int_D |\tf_t(v)|^2 dv \ge  \e_k^{-3} \sum_{D\in\scr{P}_k^\delta}(\tf_t(D))^2.
$$
On the other hand,  we observe from $\# (\scr{P}_k^\delta)\le 64 k^\delta\epsilon_k^{-3}$ that
$$
\|\tf_t\|_{L^2}^2  \ge  64^{-1}k^{-\delta}\e_k^3 \sum_{D\in\scr{P}_k^\delta} \|\tf_t\|_{L^2}^2.
$$
Using the two previous inequalities, we find that
$$
8\|\tf_t\|_{L^2}^2  \ge \sum_{D\in\scr{P}_k^\delta} \big(4\e_k^{-3}(\tf_t(D))^2+
4\times 64^{-1}k^{-\delta}\e_k^3\|\tf_t\|_{L^2}^2 \big).
$$
Consequently, on  $(\Omega_{t,k}^2)^c$, we have
$$
\sum_{D\in\scr{P}_k^\delta} A_D^2 >k^2\e_k^{3}8\|\tf_t\|_{L^2}^2 \geq k^2\e_k^{3}
\sum_{D\in\scr{P}_k^\delta} \big(4\e_k^{-3}(\tf_t(D))^2+ 16^{-1}k^{-\delta}\e_k^3\|\tf_t\|_{L^2}^2 \big),
$$
so that there is
at least one $D\in\scr{P}_k^\delta$ with $A_D^2 \ge k^2 \e_k^{3}
\big[4\e_k^{-3}(\tf_t(D))^2+16^{-1}k^{-\delta}\e_k^3\|\tf_t\|_{L^2}^2 \big] $. Hence,
 $$
\bb{P}[(\Omega_{t,k}^2)^c] \le \sum_{D\in\scr{P}_k^\delta} \bb{P}\Big(A_D \ge k \e_k^{3/2}
\big[4\e_k^{-3}(\tf_t(D))^2+16^{-1}k^{-\delta}\e_k^3\|\tf_t\|_{L^2}^2\big]^{1/2}\Big).
 $$
 But we can apply \eqref{Prop-proba}, because $x_k:=k \e_k^{3/2} \big[4\e_k^{-3}(\tf_t(D))^2+16^{-1}k^{-\delta}\e_k^3\|\tf_t\|_{L^2}^2\big]^{1/2}$ enjoys
the property that $x_k\ge k \e_k^{3/2} [4\e_k^{-3}(\tf_t(D))^2]^{1/2}=2k\tf_t(D)$:
  $$ \bb{P}[(\Omega_{t,k}^2)^c] \le \sum_{D\in\scr{P}_k^\delta} \exp(-x_k/8).$$
Using that  $x_k \ge k \e_k^{3/2}(16^{-1}k^{-\delta}\e_k^3\|\tf_t\|_{L^2}^2)^{1/2}= ck^{\delta/2}\|\tf_t\|_{L^2}$,
that $\#(\scr{P}_k^\delta) \le 64k$ and that $\|\tf_t\|_{L^2}\ge M_2$, we deduce that
  $$
  \bb{P}[(\Omega_{t,k}^2)^c] \le \sum_{D\in\scr{P}_k^\delta} \exp(- ck^{\delta/2}\|\tf_t\|_{L^2}/8)
\le 64k \exp(-c M_2k^{\delta/2}/8) \le C\exp(-cM_2 k^{\delta/2}/10).
  $$
  This ends the step.

  \vskip1mm 
  
  {\it Step 3.}   
 We denote $\Omega_{t,k}^1=\{\textit{for all $i=1,...,k$},\  |X_t^i|\le k^{\delta/3}\}$.
 Finally, we show that  on $\Omega^2_{t,k}\cap \Omega_{t,k}^1$,  $\|\bar\mu_{{\bf X}_t}^k\|_{L^2}\le 144 (1+\|\tf_t\|_{L^2})$. This is sufficient to deduce the final result since
 \begin{align*}
     \mathbb{P}((\Omega_{t,k}^1)^c)\le k \mathbb{P}(|X_t^{1}|\ge k^{\delta/3})
     \le \E[|X_t^{1}|^q] k^{1-q\delta/3}\le Ck^{1-q\delta/3}.
 \end{align*}
 We have
 \benu[label=(\roman*)]
 \item for all $i=1,...,k$, $X_t^{i}\in B(0, k^{\delta/3})$
(according to $\Omega_{t,k}^1$);

 \item $k^{-2}\epsilon_k^{-3}\sum_{D\in\scr{P}_k^\delta}\Big(\#\big\{i\in\{1,...,k\} : \, X_t^{i}\in D\big \}\Big)^2\le 2^{3} \| \tf_t \|_{L^2}^2$ (according to $\Omega_{t,k}^2$).
 \eenu

Using Step 1 with  $\bar\mu_{{\bf {X}_t}}^k=\mu_{{\bf {X}_t}}^k*\psi_{\epsilon_k}$,
we deduce that on $$\Omega^2_{t,k}\cap \Omega_{t,k}^1,$$  we
have
\begin{align*}
\|\bar\mu_{{\bf {X}_t}}^k\|_{L^2}\le&    36\Big(k^{-2}\epsilon_k^{-3}\sum_{D\in\scr{P}_k^\delta}(\#\{i\in\{1,...,k\} : X_t^{i}\in D\})^2\Big)^{1/2}\\
\leq &  36\times 2^{3/2}\| \tf_t \|_{L^2}.
\end{align*}
This completes the proof, since  $36\times 2^{3/2}\leq 144$ and we finally have 
\begin{align*}
    \bb{P}\Big(\|\bar\mu_{{\bf {X}_t}}^k\|_{L^2}\le 144 \|\tf_t\|_{L^2}\Big)\ge \bb{P}(\Omega^2_{t,k}\cap \Omega_{t,k}^1)\ge 1- \mathbb{P}((\Omega_{t,k}^1)^c)- \mathbb{P}((\Omega_{t, k}^2)^c) \ge 1- C_{q,\delta}k^{1-\delta q/3}.
\end{align*}

\section{Proof of Theorem \ref{th: prop chaos}}
 We first set $\e=\e_k=k^{-(1-\delta)/3}$ in \eqref{def: wik}, where  $1\le k \le N/2$ and $0<\delta<1$ are parameters that will be specified later. For $t\in[0,T]$, we denote the empirical measure associated with $\bW_t^K$ by 
\begin{align*}
\mu^{N}_{\bW^K_t}:= \frac{1}{N}\sum_{i=1}^N \delta_{W^{i,K}_t}.
\end{align*}
Applying the triangle inequality for $\cW_2$, together with Lemma~\ref{lem: ftKft} and Proposition~\ref{thep}, we obtain
 \begin{align}\label{ineq: main proof}
    \E[\cW_2^2(\mu^N_{\bV_t},f_t) ]\le& 8\Big(\E[\cW_2^2(\mu^N_{\bV_t},\mu^N_{\bW_t^K})]+\E[\cW_2^2(\mu^N_{\bW_t^K},f^K_t) ]+\E[\cW_2^2(f^K_t,f_t) ]\Big)\nonumber\\
    \le& 8N^{-1}\sum_{j=1}^N\E[|V^j_t-W^{j,K}_t|^2]+C_{T,f_0}\Big(N^{-1/3}+K^{1-1/\nu}\Big)\nonumber\\
    =& 8\E[|V^1_t-W^{1,K}_t|^2]+C_{T,f_0}\Big(N^{-1/3}+K^{1-1/\nu}\Big),
\end{align}
where the last equality follows from the exchangeability of the particles. Therefore, it only remains to estimate
$\E[|V^1_t-W^{1,K}_t|^2]$.
Using It\^o's formula, we have
\begin{align*}
&\E[|V^{1}_t-W^{1,K}_t|^2]\\
&= \frac{1}{N}\intot \int_0^1 \int_0^\infty \int_0^{2\pi} \E\Big[ |V^{1}_s-W^{1,K}_s + \Delta^{1,1}(s,\alpha,z,\varphi) |^2 - 
 |V^{1}_s-W^{1,K}_s|^2 \Big] d\varphi dz d\alpha ds\\
 &\quad+\frac{1}{N}\sum_{j=2}^N\intot \int_0^1 \int_0^\infty \int_0^{2\pi} \E\Big[ |V^{1}_s-W^{1,K}_s + \Delta^{1,j}(s,\alpha,z,\varphi) |^2 - 
 |V^{1}_s-W^{1,K}_s|^2 \Big] d\varphi dz d\alpha ds,
\end{align*}
where
\begin{align*}
\Delta^{1,j}(s,\alpha,z,\varphi)=c(V^{1}_s,V_s^{j},z,\varphi)-c_K(W^{1,K}_s,\Pi_s^{j,K}(\bW^K_s,\alpha),z,\varphi+\hat\varphi_{1,j,s}),
\end{align*}
and $\hat\varphi_{1,j,s}$ is defined on \eqref{def: wik}. We split the proof into several steps.

\vskip2mm

{\it Step 0.} We first treat the term corresponding to $j=1$. Since, by convention,
$\hat\varphi_{1,1,s}=0$, we have
 $$\Delta^{1,1}(s,\alpha,z,\varphi)=-c_K(W^{1,K}_s,\Pi_s^{1,K}(\bW^K_s,\alpha),z,\varphi).$$  
 Proceeding exactly as in the estimate of $\E[J_t^1]$ in the proof of Lemma~\ref{thele}, we obtain
\begin{align*}
&\frac{1}{N}\intot \int_0^1 \int_0^\infty \int_0^{2\pi}
\E\Big[
|V^{1}_s-W^{1,K}_s+\Delta^{1,1}(s,\alpha,z,\varphi)|^2
-|V^{1}_s-W^{1,K}_s|^2
\Big]\,d\varphi\,dz\,d\alpha\,ds
\le \frac{CT}{N}.
\end{align*}

We next turn to the contributions corresponding to $j\ge 2$.
Applying Lemma~\ref{fundest}  with $v=V_s^1, v_*=V_s^j, \tv=W_s^{1,K}, \tv_*=\Pi_s^{j,K}(\bW^K_s,\alpha)$ and $y= -\e_k Y(\alpha)$, we obtain
\begin{align*}
&\frac{1}{N}\sum_{j=2}^N\intot \int_0^1 \int_0^\infty \int_0^{2\pi} \E\Big[ |V^{1}_s-W^{1,K}_s + \Delta^{1,j}(s,\alpha,z,\varphi) |^2 - 
 |V^{1}_s-W^{1,K}_s|^2 \Big] d\varphi dz d\alpha ds\\
&\le C\int_0^t \big(B_1^K(s)+B_2^K(s)+B_3^K(s)+B_4^K(s)\big)ds. 
\end{align*}
where
\begin{align*}
    B_1^K(s):=&\frac{1}{N}\sum_{j=1}^N \int_0^1\E\left[ |V^{1}_s-W^{1,K}_s|^2(1+|\e_k Y(\alpha)+W^{1,K}_s-
\Pi_s^{j,K}(\bW^K_s,\alpha)|^\gamma) \right]d\alpha,\\
B_2^K(s):=& \int_0^1\E[|\e_k Y(\alpha)|^{2+2\gamma}+|\e_k Y(\alpha)|^{2+\gamma}]d\alpha,\\
B_3^K(s):=& \frac{1}{N}\sum_{j=2}^N\int_0^1\E\left[ |V^{j}_s-\Pi_s^{j,K}(\bW^K_s,\alpha)|^2|\e_k Y(\alpha)+W^{1,K}_s-
\Pi_s^{j,K}(\bW^K_s,\alpha)|^\gamma \right] d\alpha,\\
B_4^K(s):=&\frac{CK^{-1/\nu+1}}{N}\sum_{j=1}^N\int_0^1\E[1+|V^{1}_s|^2+|V^{j}_s|^2+|W^{1,K}_s|^2+|\Pi_s^{j,K}(\bW^K_s,\alpha)|^2+|\e_k Y(\alpha)|^{2}] d\alpha.
\end{align*}

{\it Step 1.} We first estimate the terms $B_1^K(s)$, $B_2^K(s)$, and $B_4^K(s)$. Since all second moments involved are finite, it follows that for any $t\in[0,T]$,
\begin{align}\label{ineq: b4K}
    \int_0^t B_4^K(s) ds\le C_TK^{-1/\nu+1}.
\end{align}
Moreover, since $|Y(\alpha)|\le 1$,
\begin{align}\label{ineq: b2K}
    \int_0^t B_2^K(s) ds\le 
T(\e_k^{2+2\gamma}+\e_k^{1+\gamma})
    \le C_T k^{-(1-\delta)(2+2\gamma)/3}.
\end{align}
We next consider $B_1^K(s)$. By definition, 
 \[   \int_0^t B_1^K(s) ds=  \int_0^t\E\Big[|V^{1}_s-W^{1,K}_s|^2\Big(1+\frac{1}{N}\sum_{j=1}^N\int_0^1 |\e_k Y(\alpha)+W^{1,K}_s-
\Pi_s^{j,K}(\bW^K_s,\alpha)|^\gamma d\alpha\Big)\Big]ds.
\]
Recall that $Y(\alpha)$ is independent of $\Pi_s^{j,K}$. By Lemma \ref{ww2p}-(ii), we have 
\begin{align*}
\frac{1}{N}\sum_{j=1}^N\int_0^1 |\e_k Y(\alpha)+W^{1,K}_s-
\Pi_s^{j,K}(\bW^K_s,\alpha)|^\gamma d\alpha
=\int_0^1 \int_{\R^3}|\e_k Y(\alpha)+W^{1,K}_s-
x|^\gamma f^K_s(dx) d\alpha.
\end{align*}
Since $\gamma\in(-1,0)$, Lemma~\ref{lem: prop fk}-(i) yields $\int_{\R^3}|\e_k Y(\alpha)+W^{1,K}_s-
x|^\gamma f^K_s(dx)\le 1+I_1(f_0).$
Therefore, 
\begin{align}\label{ineq: b1K}
\int_0^t B_1^K(s) ds\le C_{f_0}\int_0^t \E[|V^{1}_s-W^{1,K}_s|^2] ds.
\end{align}

{\it Step 2.} We now estimate $B_3^K(s)$. By exchangeability,  for every $j\ge 2$,
$$
\cL\big(V^{j}_s,\Pi_s^{j,K}(\bW^K_s,\alpha),W^{1,K}_s\big)
=
\cL\big(V^{k+1}_s,\Pi_s^{k+1,K}(\bW^K_s,\alpha),W^{1,K}_s\big)
=
\cL\big(V^{1}_s,\Pi_s^{1,K}(\bW^K_s,\alpha),W^{k+1,K}_s\big).
$$ 
Hence, using that $(N-1)/N\le1$
\begin{align*}
B_3^K(s)=&\frac{N-1}{N}\int_0^1\E\left[ |V^{k+1}_s-\Pi_s^{k+1,K}(\bW^K_s,\alpha)|^2|\e_k Y(\alpha)+W^{1,K}_s-
\Pi_s^{k+1,K}(\bW^K_s,\alpha)|^\gamma\right] d\alpha\\
\le&\ \int_0^1
\E\left[
|V^{1}_s-\Pi_s^{1,K}(\bW^K_s,\alpha)|^2
|\e_kY(\alpha)+W^{k+1,K}_s-
\Pi_s^{1,K}(\bW^K_s,\alpha)|^\gamma
\right]\,d\alpha\\
=:&\ B_{30}^K(s)+B_{31}^K(s)+B_{32}^K(s),
\end{align*}
where
\begin{align*}
B_{30}^K(s):=& \int_0^1 \E\Big[ |V_s^{1} -
\Pi_s^{1,K}(\bW^K_s,\alpha)|^2\Big(|\e_kY(\alpha)+W^{k+1,K}_s-
\Pi_s^{1,K}(\bW^K_s,\alpha)|^\gamma\\
&\hskip4.5cm -|\e_kY(\alpha)+\tW^{k+1,K,k}_s-
\Pi_s^{1,K}(\bW^K_s,\alpha)|^\gamma\Big) \Big] d\alpha,\\
     B_{31}^K(s):=& \int_0^1 \E\left[ |V_s^{1} -
\Pi_s^{1,K}(\bW^K_s,\alpha)|^2|\e_kY(\alpha)+\tW^{k+1,K,k}_s-
\Pi_s^{1,K}(\bW^K_s,\alpha)|^\gamma\indiq_{\Omega_{s,k}} \right] d\alpha,\\
 B_{32}^K(s):=& \int_0^1 \E\left[ |V_s^{1} -
\Pi_s^{1,K}(\bW^K_s,\alpha)|^2|\e_kY(\alpha)+\tW^{k+1,K,k}_s-
\Pi_s^{1,K}(\bW^K_s,\alpha)|^\gamma\indiq_{\Omega^c_{s,k}} \right] d\alpha.
\end{align*}

{\it Step 3.}
We first consider $B_{30}^K(s)$. Using $|x|^\gamma-|y|^\gamma\le C|x-y|(|x|^{\gamma-1}+|y|^{\gamma-1}),$ we obtain 
\[
B_{30}^K(s)\le C\E\big[B_{301}^K(s)+B_{302}^K(s)\big],
\]
where 
\begin{align*}
B_{301}^K(s)&:=\int_0^1 |V_s^{1} -
\Pi_s^{1,K}(\bW^K_s,\alpha)|^2|W^{k+1,K}_s-\tW^{k+1,K,k}_s|\\
&\hskip5cm\times |\e_kY(\alpha)+W^{k+1,K}_s-\Pi_s^{1,K}(\bW^K_s,\alpha)|^{\gamma-1}d\alpha,\\
B_{302}^K(s)&:=\int_0^1 |V_s^{1} -
\Pi_s^{1,K}(\bW^K_s,\alpha)|^2|W^{k+1,K}_s-\tW^{k+1,K,k}_s|\\
&\hskip5cm\times |\e_kY(\alpha)+\tW^{k+1,K,k}_s-\Pi_s^{1,K}(\bW^K_s,\alpha)|^{\gamma-1} d\alpha.
\end{align*} 
We only estimate the first singular term $B_{301}^K(s)$, since the second one is treated similarly.   Applying $|x+y|^2|x|^{\gamma-1}\le 2|x|^{1+\gamma}+2|y|^2|x|^{\gamma-1}$ with $x=\e_kY(\alpha)+W^{k+1,K}_s-
\Pi_s^{1,K}(\bW^K_s,\alpha)$, $y=-(\e_kY(\alpha)+W^{k+1,K}_s-V^{1}_s) ,$ we get
\begin{align*}
B_{301}^K(s)
&\le
2|W^{k+1,K}_s-\tW^{k+1,K,k}_s|
\int_0^1
\Big(
|\e_kY(\alpha)+W^{k+1,K}_s-\Pi_s^{1,K}(\bW^K_s,\alpha)|^{\gamma+1}\\
&\qquad\qquad\qquad
+
|\e_kY(\alpha)+W^{k+1,K}_s-V^{1}_s|^2
|\e_kY(\alpha)+W^{k+1,K}_s-\Pi_s^{1,K}(\bW^K_s,\alpha)|^{\gamma-1}
\Big)\,d\alpha.
\end{align*}
Using Young's inequality, we have 
\begin{align*}
&|\e_kY(\alpha)+W^{k+1,K}_s-V^{1}_s|^2
|\e_kY(\alpha)+W^{k+1,K}_s-\Pi_s^{1,K}(\bW^K_s,\alpha)|^{\gamma-1}\\
&\le |\e_kY(\alpha)+W^{k+1,K}_s-\Pi_s^{1,K}(\bW^K_s,\alpha)|^{-2}
+|\e_kY(\alpha)+W^{k+1,K}_s-V^{1}_s|^{\frac{4}{1+\gamma}}.
\end{align*}
Since  $|Y(\alpha)|\le 1$, thus 
\begin{align*}
B_{301}^K(s)
&\quad\le
C|W^{k+1,K}_s-\tW^{k+1,K,k}_s|
\int_0^1
\Big(
1+|W^{k+1,K}_s|^{\frac{4}{1+\gamma}}
+|\Pi_s^{1,K}(\bW^K_s,\alpha)|^{1+\gamma}\\
&\qquad\qquad\qquad
+|\e_kY(\alpha)+W^{k+1,K}_s-\Pi_s^{1,K}(\bW^K_s,\alpha)|^{-2}
+|W^{k+1,K}_s-V^{1}_s|^{\frac{4}{1+\gamma}}
\Big)\,d\alpha.
\end{align*}
We next use exchangeability to control the terms involving $\Pi_s^{1,K}$. 
For all  $j\notin \{k+1,...,2k\},$
\begin{align*}
    \cL(W^{k+1,K}_s,\tW^{k+1,K,k}_s,\Pi_s^{j,K}(\bW^K_s,\alpha) ) = \cL(W^{k+1,K}_s,\tW^{k+1,K,k}_s,\Pi_s^{1,K}(\bW^K_s,\alpha) ).
\end{align*}
Hence, 
\begin{align*}
&\E\Bigg[
|W^{k+1,K}_s-\tW^{k+1,K,k}_s|
\int_0^1
\Big(
|\Pi_s^{1,K}(\bW^K_s,\alpha)|^{\gamma+1}\\
&\hskip5cm +
|\e_kY(\alpha)+W^{k+1,K}_s-\Pi_s^{1,K}(\bW^K_s,\alpha)|^{-2}
\Big)\,d\alpha
\Bigg]\\
&=\ \E\Bigg[
|W^{k+1,K}_s-\tW^{k+1,K,k}_s|
\int_0^1\frac{1}{N-k}
\sum_{j\notin\{k+1,\dots,2k\}}
\Big(
|\Pi_s^{j,K}(\bW^K_s,\alpha)|^{\gamma+1}\\
&\hspace{5cm}
+
|\e_kY(\alpha)+W^{k+1,K}_s-\Pi_s^{j,K}(\bW^K_s,\alpha)|^{-2}
\Big)\,d\alpha
\Bigg]\\
&\le\ \E\Bigg[
|W^{k+1,K}_s-\tW^{k+1,K,k}_s|
\int_0^1\frac{1}{N-k}
\sum_{j=1}^{N}
\Big(
|\Pi_s^{j,K}(\bW^K_s,\alpha)|^{\gamma+1}\\
&\hspace{5cm}
+
|\e_kY(\alpha)+W^{k+1,K}_s-\Pi_s^{j,K}(\bW^K_s,\alpha)|^{-2}
\Big)\,d\alpha
\Bigg].
\end{align*}
By Lemma \ref{ww2p}-(ii),  the assumption $k\le N/2$ and Lemma \ref{lem: prop fk}-(i), the last inequality is bounded by 
\begin{align*}
&2\E\Bigg[
|W^{k+1,K}_s-\tW^{k+1,K,k}_s|
\int_0^1\int_{\R^3}
\Big(
|x|^{\gamma+1}
+
|\e_kY(\alpha)+W^{k+1,K}_s-x|^{-2}
\Big)
f_s^K(dx)\,d\alpha
\Bigg]\\
&\le\ C_{f_0}\E[|W^{k+1,K}_s-\tW^{k+1,K,k}_s|],
\end{align*}
where the  last inequality follows from the fact that, with respect to the auxiliary variable $\alpha$, for any $\bw\in (\R^3)^N$ the random variables $Y(\alpha)$ and $\Pi_s^{j,K}(\bw,\alpha)$ are independent. Denote by $\E_\alpha$ the expectation with respect to $\alpha$. Then, for fixed $w\in\R^3 $, we have
\begin{align*}
&\frac{1}{N}\sum_{j=1}^N
\int_0^1
|\e_kY(\alpha)+w-\Pi_s^{j,K}(\bw,\alpha)|^{-2}
\,d\alpha\\
&\quad=
\frac{1}{N}\sum_{j=1}^N
\E_\alpha\left[
\E_\alpha\left[
|\e_kY(\alpha)+w-\Pi_s^{j,K}(\bw,\alpha)|^{-2}
\,\big|\,Y
\right]
\right]\\
&\quad=
\frac{1}{N}\sum_{j=1}^N
\E_\alpha\left[
\int_{\R^3}
|\e_kY(\alpha)+w-x|^{-2}
f_s^K(dx)
\right]\\
&\quad=
\int_0^1\int_{\R^3}
|\e_kY(\alpha)+w-x|^{-2}
f_s^K(dx)\,d\alpha.
\end{align*}
The corresponding term involving $\tW^{k+1,K,k}_s$ in $B_{302}^K(s)$ is estimated in the same way. Consequently, for every $M\ge 1$,
\begin{align*}
B_{30}^K(s)
\le&\ C_{f_0}\E\left[
|W^{k+1,K}_s-\tW^{k+1,K,k}_s|
\Big(
1+|V^{1}_s|^{\frac{4}{1+\gamma}}
+|W^{k+1,K}_s|^{\frac{4}{1+\gamma}}
+|\tW^{k+1,K,k}_s|^{\frac{4}{1+\gamma}}
\Big)
\right]\\
\le&\ C_{f_0}M
\E[|W^{k+1,K}_s-\tW^{k+1,K,k}_s|]\\
&\quad
+C_{f_0}\E\left[
\Big(
1+|V^{1}_s|^{\frac{4}{1+\gamma}}
+|W^{k+1,K}_s|^{\frac{4}{1+\gamma}}
+|\tW^{k+1,K,k}_s|^{\frac{4}{1+\gamma}}
\Big)^2\right.\\
&\hspace{4cm}\left.
\times
\indiq_{\{
|V^{1}_s|^{\frac{4}{1+\gamma}}
+|W^{k+1,K}_s|^{\frac{4}{1+\gamma}}
+|\tW^{k+1,K,k}_s|^{\frac{4}{1+\gamma}}\ge M
\}}
\right]\\
\le&\ C_{f_0}M\E[|W^{k+1,K}_s-\tW^{k+1,K,k}_s|]
+
CM^{2-\frac{(1+\gamma)q}{4}}\\
&\quad\times
\E\left[
\Big(
1+|V^{1}_s|^{\frac{4}{1+\gamma}}
+|W^{k+1,K}_s|^{\frac{4}{1+\gamma}}
+|\tW^{k+1,K,k}_s|^{\frac{4}{1+\gamma}}
\Big)^{q(1+\gamma)/4}
\right]\\
\le&\ \frac{C_{f_0}Mk}{N}
+
C_{q,f_0}M^{2-\frac{(1+\gamma)q}{4}}+\frac{C_{f_0}M}{K^{1/\nu-1}}.
\end{align*}
The last inequality holds thanks to \eqref{decoup1} in Lemma \ref{thele}, Corollary \ref{Part-mom} and Lemma \ref{lem: prop mom}.

Since $q>8/(1+\gamma)$ optimizing in $M$ by taking $M\sim (N/k)^{\frac{4}{(1+\gamma)q-4}}$ gives 
\begin{align}\label{ineq: b320}
    B_{30}^K(s)\le C_{q,f_0,\gamma} \Big[\Big(\frac{k}{N}\Big)^{\frac{(1+\gamma)q-8}{(1+\gamma)q-4}}+\frac{N^{\frac{4}{(1+\gamma)q-4}}}{K^{1/\nu-1}}\Big].
\end{align}

{\it Step 4.}
We now estimate $B_{32}^K(s)$. As in the previous step from Remark \ref{rem: exchang} again,   for $j\notin \{k+1,...,2k\},$
\begin{align*}
    \cL(\Pi_s^{j,K}(\bW^K_s,\alpha)|\cF^k ) = \cL(\Pi_s^{1,K}(\bW^K_s,\alpha)|\cF^k),
\end{align*}
where $\cF^k=\sigma(\tW^{k+1,K,k}_s,...,\tW^{2k,K,k}_s)$.

Recall $k\le N/2$, H\"older's inequality and Corollary \ref{cor: norm-bound} give
\begin{align}\label{ineq: b322}
   B_{32}^K(s)\le& 2\int_0^1 \E\Big[ \Big( |
V^{1,K}_s-\Pi_s^{1,K}(\bW^K_s,\alpha)|^4+|\tW^{k+1,K,k}_s+\e_kY(\alpha)-
\Pi_s^{1,K}(\bW^K_s,\alpha)|^{2\gamma}\Big) 
 \indiq_{\Omega^c_{s,k}}\Big] d\alpha\nonumber\\
 \le& 8\int_0^1 \E\Big[ \Big( |
V^{1}_s|^4+|\Pi_s^{1,K}(\bW^K_s,\alpha)|^4+|\tW^{k+1,K,k}_s+\e_kY(\alpha)-
\Pi_s^{1,K}(\bW^K_s,\alpha)|^{2\gamma}\Big) 
 \indiq_{\Omega^c_{s,k}}\Big] d\alpha\nonumber\\
 \le& C\mathbb{P}^{1-\frac{4}{q}}(\Omega^c_{s,k})+8\int_0^1 \E\Big[ \frac{1}{N-k}\sum_{j=1}^N\Big(|\Pi_s^{j,K}(\bW^K_s,\alpha)|^4\notag\\
 &\hskip5cm+|\tW^{k+1,K,k}_s+\e_kY(\alpha)-
\Pi_s^{j,K}(\bW^K_s,\alpha)|^{2\gamma}\Big) 
 \indiq_{\Omega^c_{s,k}}\Big] d\alpha\nonumber \\
 \le& C\mathbb{P}^{1-\frac{4}{q}}(\Omega^c_{s,k})+16\int_0^1 \E\Big[ \frac{1}{N}\sum_{j=1}^N\Big(|\Pi_s^{j,K}(\bW^K_s,\alpha)|^4\notag\\
 &\hskip5cm+|\tW^{k+1,K,k}_s+\e_kY(\alpha)-
\Pi_s^{j,K}(\bW^K_s,\alpha)|^{2\gamma}\Big) 
 \indiq_{\Omega^c_{s,k}}\Big] d\alpha\nonumber\\
 \le& \mathbb{P}^{1-\frac{4}{q}}(\Omega^c_{s,k})+16(m_4(f^K_s)+1+I_1(f_0)) \mathbb{P}(\Omega^c_{s,k})\nonumber\\
 \le& C_{T,q,f_0}k^{(1-\delta q/3)(1-4/q)}.
\end{align}

We finally estimate $B_{31}^K$(s). We shall use the $L^2$ norm of
$\bar\mu^{k}_{\bf{\tW^{K,k}}_s}$ to control the negative $\gamma$-moments. Recall first that, for any probability density $g$ on $\R^3$ with $\|g\|_{L^2}<\infty$, and for any $w\in\R^3$,
\begin{align*}
    \int_{\R^3} |w-x|^\gamma g(x)dx=&  \int_{\R^3} |w-x|^\gamma \indiq_{|w-x|\ge 1}g(x)dx+ \int_{\R^3} |w-x|^\gamma \indiq_{|w-x|< 1}g(x)dx\\
    \le& 1+ \Big(\int_{\R^3} |w-x|^{2\gamma} \indiq_{|w-x|< 1}dx\Big)^{1/2}\Big(\int_{\R^3} g^2(x) dx\Big)^{1/2}\\
   =& 1+ \Big(\int_{\R^3} |z|^{2\gamma} \indiq_{|z|< 1}dz\Big)^{1/2}\Big(\int_{\R^3} g^2(x) dx\Big)^{1/2}\\ 
  =& 1+\Big(\int_0^\infty r^{2+2\gamma}\indiq_{r<1}dr\Big)^{1/2}\norm{g}_{L^2}\le C (\norm{g}_{L^2}+1),
\end{align*}
where $C>0$ depends only on $\gamma$. 

We also recall the definition of $\mu^{k}_{\bf{\tW^{K,k}}_s}$ from Section \ref{ttt}. The random variable $Y=Y(\alpha)$ is $\alpha$ uniformly distributed on $B(0,1)$ and is independent of $\Pi_s^{i,K}(\bW^K_s,\alpha)$. Therefore,
\begin{align*}
   & \int_0^1\frac{1}{k}\sum_{j=k+1}^{2k}|\e_kY(\alpha)+\tW^{j,K,k}_s+z|^\gamma d\alpha\\
    &=
    \int_{\R^3}\int_{\R^3}
    |x+w+z|^\gamma \psi_{\e_k}(x)\,\mu^{k}_{\bf{\tW^{K,k}}_s}(dw)\,dx\\
    &=
    \int_{\R^3}|w+z|^\gamma
    \bar\mu^{k}_{\bf{\tW^{K,k}}_s}(dw).
\end{align*}

Using these estimates and exchangeability or take a look at Remark \ref{rem: exchang}, we obtain
\begin{align*}
    B_{31}^K(s)=&  \frac{1}{k}\sum_{j=k+1}^{2k}\int_0^1 \E\left[ |V_s^{1} -
\Pi_s^{1,K}(\bW^K_s,\alpha)|^2(|\e_kY(\alpha)+\tW^{j,K,k}_s-
\Pi_s^{1,K}(\bW^K_s,\alpha)|^\gamma)\indiq_{\Omega_{s,k}} \right] d\alpha\\
     =& \int_0^1 \E\Big[  |
V_s^{1}-\Pi_s^{1,K}(\bW^K_s,\alpha)|^2
\Big(\int_{\R^3}\int_{\R^3}|w+x-
\Pi_s^{1,K}(\bW^K_s,\alpha)|^\gamma\psi_{\e_k}(x) \mu^{k}_{\bf{\tW^{K,k}}_s}(dw)dx\Big) \indiq_{\Omega_{s,k}}\Big] d\alpha\\
=&  \int_0^1 \E\left[  |
V_s^{1}-\Pi_s^{1,K}(\bW^K_s,\alpha)|^2\Big(\int_{\R^3}|w-
\Pi_s^{1,K}(\bW^K_s,\alpha)|^\gamma\bar\mu^{k}_{\bf{\tW^{K,k}}_s}(dw)\Big) \indiq_{\Omega_{s,k}}\right] d\alpha\\
\le& C\int_0^1 \E\left[|
V_s^{1}-\Pi_s^{1,K}(\bW^K_s,\alpha)|^2\Big(\norm{\bar\mu^{k}_{\bf{\tW^{K,k}}_s}}_{L^2}+1\Big)\indiq_{\Omega_{s,k}}\right] d\alpha\\
\le& C\big(1+\|f^K_s\|_{L^2}\big)\int_0^1 \E\left[|
V^1_s-W^{1,K}_s|^2+|
W^{1,K}_s-\Pi_s^{1,K}(\bW^K_s,\alpha)|^2\right] d\alpha\\
=& C\big(1+\|f^K_s\|_{L^2}\big)\Big(\E[|
V^1_s-W^{1,K}_s|^2]+\int_0^1 \frac{1}{N}\sum_{i=1}^N\E\left[|
W_s^{i,K}-\Pi_s^{i,K}(\bW^K_s,\alpha)|^2\right]d\alpha\Big)\\
=& C\big(1+\|f^K_s\|_{L^2}\big) \Big(\E[|
V^1_s-W^{1,K}_s|^2]+\E[\cW^2_2(\mu^N_{\bf{W^K_s}},f^K_s)]\Big),
\end{align*}
where the last inequality follows from Lemma \ref{ww2p}. By Proposition \ref{thep}, we conclude that 
\begin{align}\label{ineq: b321}
    B_{31}^K(s)\le C\big(1+\|f^K_s\|_{L^2}\big)\Big(\E[|
V^1_s-W^{1,K}_s|^2]+N^{-1/3}\Big)\le C_{f_0}\Big(\E[|
V^1_s-W^{1,K}_s|^2]+N^{-1/3}+K^{1-1/\nu}\Big) . 
\end{align}
Finally, by inequalities \eqref{ineq: b1K}, \eqref{ineq: b2K}, \eqref{ineq: b4K},   \eqref{ineq: b320}, \eqref{ineq: b321} and \eqref{ineq: b322}, we deduce
\begin{align*}
    \E[|V^{1}_t-W^{1,K}_t|^2]\le& \frac{CT}{N}+C\int_0^t (B_1^K(s)+B_2^K(s)+B_4^K(s))ds 
    +C\intot (B_{30}^K(s)+B_{31}^K(s)+B_{32}^K(s)) ds\\
    \le& C_{f_0}\int_0^t \E[|V^{1}_s-W^{1,K}_s|^2] ds+C_{T,f_0}\Big(K^{1-1/\nu}+\frac{N^{\frac{4}{(1+\gamma)q-4}}}{K^{1/\nu-1}}\Big)\\
    +&C_{T,q,f_0}\Big[N^{-1/3}+\big(\frac{k}{N}\big)^{\frac{(1+\gamma)q-8}{(1+\gamma)q-4}}+k^{-(1-\delta)(2+2\gamma)/3}+k^{(1-\delta q/3)(1-4/q)}\Big].
\end{align*}
 Choosing $\delta=6/q$, then $$k^{-(1-\delta)(2+2\gamma)/3}+k^{(1-\delta q/3)(1-4/q)}=k^{-(1-6/q)(2+2\gamma)/3}+k^{-(1-4/q)} \le 2k^{-(1-6/q)(2+2\gamma)/3},$$ when $q>8/(1+\gamma).$ By Gr\"onwall's inequality, we deduce  
\begin{align*}
    \E[|V^{1}_t-W^{1,K}_t|^2]\le C_{T,q,f_0}\frac{N^{\frac{4}{(1+\gamma)q-4}}}{K^{1/\nu-1}}+C_{T,q,f_0}\Big[N^{-1/3}+\big(\frac{k}{N}\big)^{\frac{(1+\gamma)q-8}{(1+\gamma)q-4}}+k^{-(1-6 /q)(2+2\gamma)/3}\Big].
\end{align*}
Now we return to \eqref{ineq: main proof}, we finally deduce for any $1\le K,$ $1\le k\le N/2,$  we have
\begin{align*}
    \E[\cW_2^2(\mu^N_{\bV_t},f_t) ]\le C_{T,q,f_0}\frac{N^{\frac{4}{(1+\gamma)q-4}}}{K^{1/\nu-1}}
    +C_{T,q,f_0}\Big[N^{-1/3}+\big(\frac{k}{N}\big)^{\frac{(1+\gamma)q-8}{(1+\gamma)q-4}}+k^{-(1- 6/q)(2+2\gamma)/3}\Big].
\end{align*}
Letting $K\to\infty$ and  finally optimizing in $k\sim N^{\ell_*(q,\gamma)}$ with 
$$
\ell_*(q,\gamma)= \frac{3q((1+\gamma)q-8)}
{(1+\gamma)(5+2\gamma)q^2-(44+32\gamma+12\gamma^2)q+48(1+\gamma)},
$$
which is obtained by balancing the last two terms. Thus, \eqref{eq:resu} follows and we complete the proof.

\vskip4mm \noindent\textbf{Acknowledgements} 
AZ is supported by National Key R\&D Program of China No. 2024YFA1015300, Beijing Natural Science Foundation No. 1242009, National Natural Science Foundation of China No. 11801536. LX is supported by the National Natural Science Foundation of China No. 12101028.  We are deeply grateful to Nicolas Fournier for valuable comments, and to Zhenfu Wang for helpful conversations.

\noindent\textbf{Statements} 
The authors have no relevant financial or non-financial interests to disclose. Data sharing is not applicable to this article as no datasets were used.

\end{document}